\documentclass[12pt,reqno]{amsart}

\usepackage{amsmath,amsfonts,amsthm,amssymb,amsxtra,enumitem,color,comment,changepage}
\usepackage{bbm} 
\usepackage[hidelinks]{hyperref} 

\newtheorem{theorem}{Theorem}[section]
\newtheorem{proposition}[theorem]{Proposition}
\newtheorem{lemma}[theorem]{Lemma}
\newtheorem{corollary}[theorem]{Corollary}

\newtheorem{definition}[theorem]{Definition}

\theoremstyle{remark}

\newtheorem{remark}[theorem]{Remark}

\DeclareRobustCommand{\SkipTocEntry}[5]{}

\newcommand{\1}{\mathbbm{1}}

\let\eps\epsilon
\renewcommand{\epsilon}{\varepsilon}
\newcommand{\I}{\mathbb{I}}

\newcommand{\Haus}{\mathcal{H}}

\newcommand{\R}{\mathbb{R}}

\renewcommand{\S}{\mathbb{S}}

\DeclareMathOperator{\dist}{dist}
\DeclareMathOperator{\Div}{div}
\DeclareMathOperator{\dom}{dom}

\DeclareMathOperator{\supp}{supp}

\DeclareMathOperator{\Tr}{Tr}

\newcommand{\limplus}{{\mathchoice{\vcenter{\hbox{$\scriptstyle +$}}}
		{\vcenter{\hbox{$\scriptstyle +$}}}
		{\vcenter{\hbox{$\scriptscriptstyle +$}}}
		{\vcenter{\hbox{$\scriptscriptstyle +$}}}
}}
\newcommand{\limminus}{{\mathchoice{\vcenter{\hbox{$\scriptstyle -$}}}
		{\vcenter{\hbox{$\scriptstyle -$}}}
		{\vcenter{\hbox{$\scriptscriptstyle -$}}}
		{\vcenter{\hbox{$\scriptscriptstyle -$}}}
}}

\begin{document}

\newpage

\title[Uniform bounds for Neumann heat kernels]{Uniform bounds for Neumann heat kernels and their traces in convex sets}

\author{Rupert L. Frank}
\address[Rupert L. Frank]{Mathe\-matisches Institut, Ludwig-Maximilians Universit\"at M\"unchen, The\-resienstr.~39, 80333 M\"unchen, Germany, and Munich Center for Quantum Science and Technology, Schel\-ling\-str.~4, 80799 M\"unchen, Germany, and Mathematics 253-37, Caltech, Pasa\-de\-na, CA 91125, USA}
\email{r.frank@lmu.de}

\author{Simon Larson}
\address{\textnormal{(Simon Larson)} Department of Mathematical Sciences, Chalmers University of Technology and the University of Gothenburg, SE-41296 Gothenburg, Sweden}
\email{larsons@chalmers.se}

\thanks{\copyright\, 2026 by the authors. This paper may be reproduced, in its entirety, for non-commercial purposes.\\
	Partial support through US National Science Foundation grant DMS-1954995 (R.L.F.), the German Research Foundation grants EXC-2111-390814868 and TRR 352-Project-ID 470903074 (R.L.F.), as well as the Swedish Research Council grant no.~2023-03985 (S.L.) is acknowledged.}

\begin{abstract}
    We prove a bound on the heat trace of the Neumann Laplacian on a convex domain that captures the first two terms in its small-time expansion, but is valid for all times and depends on the underlying domain only through very simple geometric characteristics. This is proved via a precise and uniform expansion of the on-diagonal heat kernel close to the boundary. Most of our results are valid without the convexity assumption and we also consider two-term asymptotics for the heat trace for Lipschitz domains.
\end{abstract}

\maketitle

\tableofcontents

\section{Introduction and main results}

\subsection{Results on the heat trace}

Let $\Omega \subset \R^d$, $d\geq 2$, be an open set. The Neumann Laplacian $-\Delta_\Omega$ in $L^2(\Omega)$ is defined through the quadratic form
\begin{equation*}
	u \mapsto \int_\Omega |\nabla u |^2\,dx \quad \mbox{for }u \in H^1(\Omega)\,.
\end{equation*}
If $\Omega$ is bounded and has Lipschitz boundary, then the spectrum of $-\Delta_\Omega$ is discrete, nonnegative, and accumulates only at $\infty$; see, e.g., \cite{LTbook}. Let $\{\lambda_k(\Omega)\}_{k\geq 1}$ denote the eigenvalues ordered in a nondecreasing manner and repeated according to their (finite) multiplicity, so that
\begin{equation*}
 	0=\lambda_1(\Omega)\leq \lambda_2(\Omega) \leq \lambda_3(\Omega)\leq \ldots \to \infty\,.
\end{equation*} 
Let $\{u_k\}_{k\geq 0}\subset L^2(\Omega)$ be an associated family of real-valued orthonormal eigenfunctions.

Our main interest in this paper is on the \emph{heat trace},
$$
\Tr(e^{t\Delta_\Omega}) = \sum_{k=1}^\infty e^{-t\lambda_k(\Omega)}
\qquad\text{for}\ t>0
$$
and, in particular, in its behavior of small $t$. This, as well as its generalizations to other boundary conditions or to manifolds, is a classical topic in spectral theory with many important contributions including \cite{MinakshisundaramPleijel,Kac,McKeanSinger,vdBerg_84,BrossardCarmona,vdBerg_87,BransonGilkey}.
 In particular, when $\partial\Omega$ is smooth, then $t^{d/2} \Tr(e^{t\Delta_\Omega})$ has a complete asymptotic expansion in integer powers of $\sqrt t$.

In this paper we are interested in the low regularity setting, where one cannot expect a complete asymptotic expansion. In the above-described setting of the Neumann Laplacian on a Lipschitz set there is the following celebrated theorem of Brown~\cite{Brown93}.

\begin{theorem}[{\cite[Main Theorem]{Brown93}}]\label{thm: Brown trace asymptotics}
    Let $d\geq 2$ and $\Omega \subset \R^d$ be an open and bounded set with Lipschitz boundary. Then
    \begin{equation}\label{eq: two-term heat asymptotics}
    \Tr(e^{t\Delta_\Omega})= \frac{1}{(4\pi t)^{d/2}}\biggl(|\Omega| + \frac{\sqrt{\pi t}}{2}\Haus^{d-1}(\partial\Omega) + o(\sqrt{t})\biggr) \quad \mbox{as }t\to 0^\limplus\,.
    \end{equation}
\end{theorem}

Our main interest in this paper is towards the case where $\Omega$ is \emph{convex}. Of course, as any convex set has Lipschitz boundary, Brown's result is applicable, but we aim to make the error term $o(\sqrt t)$ more explicit and, in particular, to clarify its dependence on the geometry of $\Omega$. This question is motivated by applications to certain spectral shape optimization problems, which we studied in \cite{LarsonJST,FrankLarson_CPAM26}. The theorem that we prove in this paper is a crucial ingredient for one of the results in \cite{FrankLarson_Inventiones25}, which in turn is fundamental in \cite{FrankLarson_CPAM26}.

Our heat trace bound in this case will depend on $\Omega$ through its \emph{inradius},
\begin{equation}
    \label{eq:inrad}
    r_{\rm in}(\Omega) := \sup_{x\in\Omega} d_\Omega(x) \,,
\end{equation}
which is defined in terms of the distance to the complement
\begin{equation*}
    d_\Omega(x) := \dist(x,\Omega^c) \,.
\end{equation*} 
The following is the main result of this paper.

\begin{theorem}\label{thm: main thm convex}
    Fix $d\geq 2$ and $\epsilon>0$. There exists a constant $C_{d,\epsilon}$ so that if $\Omega \subset \R^d$ is open, convex, and bounded, then, for all $t>0$,
	\begin{align*}
		\biggl|(4\pi t)^{d/2}\Tr(e^{t\Delta_\Omega}) - &|\Omega| - \frac{\sqrt{\pi t}}{2}\Haus^{d-1}(\partial\Omega)\biggr|\\
        &\leq
        C_{d, \epsilon} \Haus^{d-1}(\partial\Omega)\sqrt{t}\biggl[\Bigl(\frac{\sqrt{t}}{r_{\rm in}(\Omega)}\Bigr)^{1/2-\epsilon} + \Bigl(\frac{\sqrt{t}}{r_{\rm in}(\Omega)}\Bigr)^{d-1}\biggr].
	\end{align*}
\end{theorem}

In the setting of convex sets, our main theorem improves Theorem~\ref{thm: Brown trace asymptotics} in two ways: (i) the bound in Theorem \ref{thm: main thm convex} is valid for all $t>0$ and not only asymptotically, and (ii) Theorem~\ref{thm: main thm convex} provides explicit control of the size of the $o$-term in \eqref{eq: two-term heat asymptotics}. 

Theorem~\ref{thm: main thm convex} tells us that for convex sets the little $o(\sqrt{t})$-term can be improved on the algebraic scale (uniformly). In contrast, such improvements are not possible in the Lipschitz setting, as pointed out by Brown in \cite{Brown93}. Indeed, Brown argues that for any $\epsilon>0$ there is a Lipschitz set $\Omega \subset\R^d$ such that
\begin{equation*}
    \limsup_{t\to 0^\limplus} \frac{|(4\pi t)^{d/2}\Tr(e^{t\Delta_\Omega}) - |\Omega| - \frac{\sqrt{\pi t}}{2}\Haus^{d-1}(\partial\Omega)|}{t^{1/2+\epsilon}}= \infty\,.
\end{equation*}
For the corresponding two-term asymptotic expansion for the Dirichlet (rather than Neumann) Laplacian it was shown in \cite[Theorem 1.1]{FrankLarson_JMP20} that the $o(\sqrt{t})$-remainder cannot be improved in the class of all Lipschitz sets, not only on the algebraic scale but in any quantitative sense whatsoever.

Although our main aim is to prove Theorem \ref{thm: main thm convex}, many of the main steps of the argument work in the setting of Lipschitz sets and, as a bi-product of our version of the argument, we shall prove a version of Theorem~\ref{thm: Brown trace asymptotics} with rather explicit estimates for the $o$-remainder. We will follow the overall strategy of Brown, but even in the Lipschitz case we felt the need at several places to add more details or to modify the argument. To deal with the convex case, we need to make several of the argument much more precise.

We believe that the machinery that we develop to prove Theorem \ref{thm: main thm convex} is also applicable to other subclasses of Lipschitz domains and leads to an improvement of Theorem~\ref{thm: main thm convex}, provided for this class of sets one can improve a few `little-o results' valid in the Lipschitz case; see, in particular, how in Subsection \ref{sec:goodconvex} the results of Subsection \ref{sec:goodlipschitz} are improved, or how Lemma \ref{lem: off-diagonal kernel bound convex} improves upon Lemma \ref{heatnapriori}. This paper being already rather long, we have decided not to investigate this further, but rather to focus on the theorem that we need in \cite{FrankLarson_Inventiones25}.

Theorems \ref{thm: Brown trace asymptotics} and \ref{thm: main thm convex} are obtained by writing the heat trace as an integral of the on-diagonal heat kernel over $\Omega$ and by studying this heat kernel in various regions of space. Let us explain this procedure in some more detail and explain some of the difficulties, before turning to the required heat kernel bounds in the following subsection.

It is known that, for any $t>0$, the operator $e^{t\Delta_\Omega}$ is an integral operator; see \cite[Theorems 2.3.6 and 2.4.4]{Davies_heatkernels}. Its integral kernel will be denoted by $k_\Omega(t,\cdot,\cdot)$, that is,
$$
(e^{t\Delta_\Omega}f)(x) = \int_\Omega k_\Omega(t,x,y) f(y)\,dy
\qquad\text{for all}\ f\in L^2(\Omega) \,.
$$ 
If the spectrum of $-\Delta_\Omega$ is discrete, then the kernel $k_\Omega$ can be represented, in the sense of a Schmidt decomposition of a compact operator, as the series
\begin{equation*}
	k_\Omega(t, x, y) = \sum_{k\geq 1} e^{-t\lambda_k(\Omega)}u_k(x)u_k(y)\,.
\end{equation*}
For the heat trace, we have the representation
\begin{equation}\label{eq:heattraceheatkernel}
    \Tr(e^{t\Delta_\Omega})= \int_\Omega k_\Omega(t, x, x)\,dx \,.
\end{equation}
For general background on heat kernels we recommend Davies's book \cite{Davies_heatkernels}.

The small $t$-asymptotics are proved, roughly speaking, by splitting the integral in~\eqref{eq:heattraceheatkernel} into a bulk part, corresponding to points $x$ with $d_\Omega(x)\gg \sqrt t$ and a boundary region corresponding to points $x$ with $d_\Omega(x)\lesssim \sqrt t$. Roughly speaking, the integral of the heat kernel over the bulk contributes the term $(4\pi t)^{-d/2}|\Omega|$ and the integral over the boundary region contributes the term $(4\pi t)^{-d/2} (\sqrt{\pi t}/2) \mathcal H^{d-1}(\partial\Omega)$. 

As such, our main work consists in obtaining quantitative approximations for $k_\Omega$. The main philosophy is that in the bulk the heat kernel should be well-approximated by the heat kernel of the Laplacian in $\R^d$, while in the boundary region $k_\Omega$ should be well-approximated by the Neumann heat kernel in a suitably chosen half-space. This intuition can be made precise when the boundary is $C^1$, but in the Lipschitz (or even convex case) it is in the approximation in the boundary region where the technical work starts and where tools from geometric measure theory and from parabolic PDEs become necessary. Indeed, the approximation of the heat kernel by a half-space heat kernel only works at points near which the boundary is well-approximated by a hyperplane at the scale $\sqrt{t}$. One needs to quantify to which extent these points constitute `most' of the boundary region and one needs to deal with the remaining points.

In order to prove the desired approximation results for $k_\Omega$ we follow the ideas of Brown who, in turn, is inspired by the study of the Neumann problem for Laplace's equation in Lipschitz domains due to Dahlberg and Kenig \cite{DaKe}. An a priori estimate that Brown does not have in the Lipschitz case, but that we have in the convex case, is the heat kernel bound by Li and Yau \cite{LiYau86}.

\subsection{Results for the heat kernel}
Our first result provides an accurate approximation of $k_\Omega$ in the bulk of $\Omega$, i.e.\ for $x$ so that $d_\Omega(x)\gg \sqrt{t}$. This result is essentially due to Brown \cite[Eq.~(2.4+)]{Brown93}. However, as will be explained below, we prove a more general statement and track more explicitly how the implicit constant can be controlled.

Before stating our first result we recall that an open set $\Omega\subset\R^d$ is said to have the \emph{extension property} if there is a bounded, linear operator $\mathcal{E}_\Omega \colon H^1(\Omega) \to H^1(\R^d)$ such that, if $u\in H^1(\Omega)$, then $\mathcal{E}_\Omega u(x) = u(x)$ for almost every $x \in \Omega$. It is well-known that open sets with uniformly Lipschitz boundary have the extension property (see, e.g., \cite[Theorem 2.92]{LTbook}). In particular, any bounded open Lipschitz set has the extension property.
\begin{theorem}\label{thm: diagonal bulk Lipschitz}
    Let $d\geq 1$, $\eta,\delta >0$ and let $\Omega\subset\R^d$ be an open set with the extension property. There is a constant $C_{\Omega,\eta,\delta}$ such that for all  $t>0$ and $x\in\Omega$ with $d_\Omega(x)\geq \eta \sqrt{t}$ we have
    $$
    \Bigl|k_\Omega(t, x, x)- (4\pi t)^{-d/2}\Bigr| \leq C_{\Omega,\eta,\delta}  \, t^{-d/2} e^{-\frac{d_\Omega(x)^2}{(1+\delta)t}}\,.
    $$
\end{theorem}

For the Neumann Laplacian in a convex set we shall prove the following result, which improves upon the inequality obtained in Theorem~\ref{thm: diagonal bulk Lipschitz} in that the implicit constant can be chosen independent of the underlying geometry.

\begin{theorem}\label{thm: diagonal bulk convex}
Let $d\geq 2$ and $\eta, \delta>0$. There is a constant $C_{d,\eta,\delta}$ such that for any open convex set $\Omega \subset \R^d$, all $t>0$ and $x\in\Omega$ with $d_\Omega(x)\geq \eta \sqrt{t}$ we have
	\begin{equation*}
		\Bigl|k_\Omega(t, x, x)- (4\pi t)^{-d/2}\Bigr| \leq C_{d,\eta,\delta} \, t^{-d/2}e^{-\frac{d_\Omega(x)^2}{(1+\delta)t}}\,.
	\end{equation*}
\end{theorem}

Since the heat kernel of the Laplace operator in $L^2(\R^d)$ is given by 
\begin{equation*}
    k_{\R^d}(t, x, y) =(4\pi t)^{-d/2}e^{-\frac{|x-y|^2}{4t}} \,,
\end{equation*}
Theorems~\ref{thm: diagonal bulk Lipschitz} and~\ref{thm: diagonal bulk convex} both show that the difference $k_\Omega -k_{\R^d}$ is small (relative to $k_{\R^d}$) on the diagonal $x=y$ as long as $d_\Omega(x)\gg \sqrt{t}$.

For $x \in \Omega$ with $d_\Omega(x)\lesssim \sqrt{t}$ we shall instead prove that $k_\Omega$ is well-approximated by the half-space heat kernel. For the half-space $\R^d_\limplus:=\{x= (x_1, \ldots, x_d) \in \R^d: x_d>0\}$ the heat kernel is explicitly given by
\begin{equation*}
    k_{\R^d_\limplus}(t, x, y) = k_{\R^d}(t, x, y) + k_{\R^d}(t, x, \tau(y)) \,,
\end{equation*}
where $\tau\colon \R^d_\limplus \to \R^d$ is the natural reflection map through $\partial \R^d_\limplus$, $$\tau((x_1, \ldots, x_d))= (x_1, \ldots, x_{d-1}, -x_d)\,.$$
In particular, we have that
\begin{equation*}
    k_{\R^d_\limplus}(t, x, x) = (4\pi t)^{-d/2}\bigl(1+ e^{-x_d^2/t}\bigr)\,.
\end{equation*}

As the precise geometric assumptions needed for our proof that $k_\Omega$ can be approximated with a half-space kernel get rather technical, here we only state a qualitative result. The result captures that, as $t \to 0^\limplus$, the heat kernel is well-approximated by that in a half-space in an arbitrarily large portion of the boundary region $d_\Omega(x) \lesssim \sqrt{t}$.

\begin{theorem}\label{thm: diagonal bounds good set Lipschitz}
	Let $\Omega \subset \R^d, d\geq2,$ be a bounded open set with Lipschitz boundary. There exist a constant $C_\Omega$, a function $\mathcal{E}\colon (0, \infty)\to (0, 1]$, and a family of measurable sets $\{\mathcal{G}_r\}_{r>0}$, $\mathcal{G}_r\subset \Omega$, with the properties that
    \begin{equation*}   
        \lim_{r\to 0^\limplus}\mathcal{E}(r)=0\,, \qquad  \lim_{r \to 0^\limplus} \sup_{0<s\leq r} \frac{|\{x\in \Omega: d_\Omega(x)<s\}\setminus\mathcal{G}_r|}{|\{x\in \Omega: d_\Omega(x)<s\}|} =0 \,,
    \end{equation*}
    and
	\begin{equation*}
			\Bigl|k_\Omega(t, x, x) - (4\pi t)^{-d/2}\bigl(1 + e^{-d_\Omega(x)^2/t}\bigr)\Bigr|\leq C_\Omega t^{-d/2}e^{-cd_\Omega(x)^2/t} \Bigl(\mathcal{E}(r) + e^{- cr^2/t}\Bigr)\,,
	\end{equation*}
    for all $(t, x)\in(0, r^2]\times \mathcal{G}_r$ where $c>0$ is a universal constant.
\end{theorem}

The sets $\mathcal{G}_r$ will be described in detail in Section \ref{sec:constructiongood} below. Essentially, they consist of those points $x\in \Omega$ for which there exists $y \in \partial\Omega$ with $|x-y|<r$ and such that $\partial\Omega \cap B_{2r}(y)$ can be closely approximated by a hyperplane. 

Just as Theorem \ref{thm: diagonal bulk convex} strengthened the conclusion of Theorem \ref{thm: diagonal bulk Lipschitz} under the additional assumption that $\Omega$ was convex, we shall prove a stronger version of Theorem~\ref{thm: diagonal bounds good set Lipschitz} for convex sets. 

\begin{theorem}\label{thm: diagonal kernel bounds convex}
	Let $\Omega \subset \R^d$ be a bounded, open, and convex set. For any $\epsilon \in (0, \frac{1}{4}]$ and $r\in (0, \epsilon r_{\rm in}(\Omega)]$ there is a measurable set $\mathcal{G}_{\epsilon, r}\subset\Omega$ such that
	\begin{equation}\label{eq: Good set is big0}
    	\frac{|\{x\in \Omega: d_\Omega(x)<s\}\setminus \mathcal{G}_{\epsilon, r}|}{|\{x\in \Omega: d_\Omega(x)<s\}|} \leq C_d \frac{r}{\epsilon \, r_{\rm in}(\Omega)}
    \end{equation}
    for all $s \in (0, r/2]$, and
    \begin{equation*}
			\Bigl|k_\Omega(t, x, x) - (4\pi t)^{-d/2}\bigl(1 + e^{-d_\Omega(x)^2/t}\bigr)\Bigr|\leq C_d  t^{-d/2}e^{-cd_\Omega(x)^2/t}\bigl(\epsilon + e^{- cr^2/t}\bigr)
	\end{equation*}
    for all $(t, x)\in (0, r^2/2]\times \mathcal{G}_{\epsilon, r}$ where $c>0$ is a universal constant and $C_d$ depends only on $d$.
\end{theorem}

This theorem improves upon Theorem~\ref{thm: diagonal bounds good set Lipschitz} in that (i) the quantities $\mathcal{E}$ and $|\{x\in \Omega: d_\Omega(x)<s\}\setminus \mathcal{G}_r|$ satisfy more explicit bounds in terms of simple geometric characteristics of $\Omega$, and (ii) the constants appearing in the bounds depend only on the dimension and not the specific convex set.



\section{The trace of the Neumann heat kernel}

Our main objective in this section is to prove Theorems~\ref{thm: Brown trace asymptotics} and~\ref{thm: main thm convex} under the assumption that we have accurate estimates for the heat kernel at our disposal. Given this result, all remaining sections of this paper will be concerned with heat kernel bounds and asymptotics.


\subsection{A semi-abstract result on Neumann heat traces}

We begin by showing, in a rather general setting, how approximation results for $k_\Omega$ lead to bounds and asymptotics for the trace. We assume three types of bounds. First, a universal order sharp bound valid everywhere; second, a sufficiently good leading order approximation in the bulk; and third, a sufficiently good leading plus subleading approximation in part of the boundary region. These bounds correspond to item (1), (2) and (3) in the following theorem.

In addition to the accuracy of the various approximations of $k_\Omega$, the estimate we obtain will depend on the quantity
\begin{equation*}
    \vartheta_\Omega(s) := \frac{|\{x\in \Omega: d_\Omega(x)<s\}|}{s \, \Haus^{d-1}(\partial\Omega)}-1
    \qquad\text{for}\ s>0 \,.
\end{equation*}

\begin{proposition}\label{thm: trace of N heatkernel general}
	Let $\Omega \subset \R^d, d\geq2,$ be an open set and assume that:
    \begin{enumerate}
        \item There exists a function $\mathcal{E}_0\colon (0, \infty)\times \Omega \to \R$ such that
        \begin{equation*}
        k_\Omega(t, x, x) \leq t^{-d/2}\mathcal{E}_0(t, x)\quad \mbox{for all }x\in \Omega, t>0\,.
        \end{equation*}
    
        \item There exist constants $c_1, \eta >0$ and a function $\mathcal{E}_1\colon (0, \infty)\times \Omega \to \R$ such that
        \begin{equation*}
        |k_\Omega(t, x, x)-(4\pi t)^{-d/2}| \leq t^{-d/2} \mathcal{E}_1(t, x)e^{-c_1 d_\Omega(x)^2/t}
        \end{equation*}
        for all $t>0$ and $x\in \Omega$ with $d_\Omega(x)\geq \eta \sqrt{t}$.

        \item There exist a constant $c_2>0$, a function $\mathcal{E}_2\colon (0, \infty)^2\times \Omega \to \R$, and measurable sets $\{\mathcal{G}_r\}_{r>0}$, $\mathcal{G}_r\subset\Omega$ such that
	    \begin{equation*}
			\Bigl|k_\Omega(t, x, x) - (4\pi t)^{-d/2}\bigl(1 + e^{-d_\Omega(x)^2/t}\bigr)\Bigr|\leq t^{-d/2} \mathcal{E}_2(r, t, x)e^{-c_2d_\Omega(x)^2/t} \,,
	    \end{equation*}
        for all $(t, x)\in(0, r^2]\times \mathcal{G}_r$.
    \end{enumerate}

    Then, for any $r\geq \max\{1, \eta\}\sqrt{t}>0$,\footnote{Here and in what follows we use the notation $\lesssim / \gtrsim$ to mean that the left side is bounded from above/below by a positive constant times the right side, where the implied constant is independent of the relevant parameters. Sometimes, as here, we put a subscript under the symbols, emphasizing that the implied constant only depends on the parameters appearing in this subscript. }
    \begin{equation*}
    \begin{aligned}
	   \biggl|&(4\pi t)^{d/2}\Tr(e^{t\Delta_\Omega}) - |\Omega| - \frac{\sqrt{\pi t}}{2}\Haus^{d-1}(\partial\Omega)\biggr|\\
        & \lesssim_{\eta,c_1,c_2}
        \sqrt{t}\Haus^{d-1}(\partial\Omega)\Biggl(\int_0^{r/\sqrt{t}} s^2 e^{-c s^2} |\vartheta_\Omega(s\sqrt{t})|\,ds+ e^{-c r^2/t}\\
        &\quad
        +
        \|\mathcal{E}_2(r, t, \cdot)\|_{L^\infty(\mathcal{G}_r)} \biggl[\int_0^{2r/\sqrt{t}} s^2 e^{-cs^2} |\vartheta_\Omega(s\sqrt{t})|\,ds+1\biggr]\\
        &\quad
        +
        \|\mathcal{E}_1(t, \cdot)\|_{L^\infty(\{x\in \Omega: d_\Omega(x)\geq r\})}\biggl[\int_{r/\sqrt{t}}^\infty s^2 e^{-c s^2}|\vartheta_\Omega(s\sqrt{t})|\,ds 
        +
        e^{-cr^2/t}\biggr]\\
        &\quad
        +\frac{1}{\sqrt{t}\Haus^{d-1}(\partial\Omega)}\int_{\{x\in \Omega: d_\Omega(x)<r\}\setminus \mathcal{G}_r} (1+\mathcal{E}_0(t, x))\,dx\Biggr)\,,
    \end{aligned}
    \end{equation*}
    where $c = \min\{1/8, c_2/8, c_1/2\}$.
\end{proposition}

\begin{proof}
	\emph{Step 1. Decomposition of the trace.}
	
    Define
	\begin{align*}
		\Omega_{bulk} & := \{x \in \Omega: d_\Omega(x)\geq r\}\,,\\
		\Omega_{bad} & := \{x \in \Omega: d_\Omega(x)< r\} \setminus \mathcal{G}_{r}\,,\\
		\Omega_{good} & := \{x \in \Omega: d_\Omega(x)< r\} \cap \mathcal{G}_{r}\,.
	\end{align*}
	Splitting the integral of $k_\Omega$ into these disjoint sets, we find
	\begin{equation}\label{eq: integral trace decomposition}
	\begin{aligned}
		(4\pi t)^{d/2}\Tr(e^{t\Delta_\Omega}) &=(4\pi t)^{d/2} \int_\Omega k_\Omega(t, x, x)\,dx\\
		&=
		(4\pi t)^{d/2}\Biggl[\int_{\Omega_{bulk}} k_\Omega(t, x, x)\,dx+ \int_{\Omega_{bad}} k_\Omega(t, x, x)\,dx\\
		&\qquad + \int_{\Omega_{good}} k_\Omega(t, x, x)\,dx\Biggr]\\
		&=
		|\Omega| + \int_{\{x\in \Omega:d_\Omega(x) <r\}} e^{-d_\Omega(x)^2/t}\,dx\\
		&\quad + \int_{\Omega_{bulk}} \Bigl((4\pi t)^{d/2}k_\Omega(t, x, x)-1\Bigr)\,dx\\
		&\quad + \int_{\Omega_{bad}} \Bigl((4\pi t)^{d/2}k_\Omega(t, x, x)-1- e^{-d_\Omega(x)^2/t}\Bigr)\,dx\\
		&\quad + \int_{\Omega_{good}} \Bigl((4\pi t)^{d/2}k_\Omega(t, x, x)-1 - e^{-d_\Omega(x)^2/t}\Bigr)\,dx\,.
	\end{aligned}
	\end{equation}
	
\medskip

\noindent \emph{Step 2. Extracting the main terms.}

We first show that the first two terms on the right side of \eqref{eq: integral trace decomposition} correspond to the terms in the asymptotic expansion, provided $r \gg \sqrt{t}$. To see this, we note that by Fubini's theorem and writing
\begin{equation*}
    e^{-y^2} = 2\int_y^\infty se^{-s^2}\,ds = 2\int_0^\infty \1_{\{s>y\}}se^{-s^2}\,ds \,,
\end{equation*}
we have
\begin{align*}
	\int_{\{x\in \Omega:d_\Omega(x) <r\}} e^{-\frac{d_\Omega(x)^2}{t}}\,dx
	&=
	2\int_0^\infty\int_\Omega  \1_{\{y\in \Omega:d_\Omega(y) <r\}}(x)\1_{\{y\in \Omega: d_\Omega(y)<s \sqrt{t}\}}(x) se^{-s^2}\,ds\,dx\\
    &=
	2\int_0^\infty|\{x\in \Omega:d_\Omega(x) <\min\{r, \sqrt{t} s\}\}| se^{-s^2}\,ds\\
    &=
	2\int_0^{r/\sqrt{t}}|\{x\in \Omega:d_\Omega(x) <\sqrt{t} s\}| se^{-s^2}\,ds\\
    &\quad 
    +2|\{x\in \Omega:d_\Omega(x) <r\}|\int_{r/\sqrt{t}}^\infty se^{-s^2}\,ds\,.
\end{align*}
Since for any $s>0$
\begin{equation*}
    |\{x\in \Omega: d_\Omega(x)<s\}| = s \Haus^{d-1}(\partial\Omega)(1+\vartheta_\Omega(s)) \,,
\end{equation*}
we can write the above equality as
\begin{align*}
	\int_{\{x\in \Omega:d_\Omega(x) <r\}} e^{-\frac{d_\Omega(x)^2}{t}}\,dx
    &=
	2\sqrt{t}\Haus^{d-1}(\partial\Omega)\biggl[\int_0^{\infty}s^2e^{-s^2}\,ds\\
    &\qquad - \int_{r/\sqrt{t}}^\infty s^2e^{-s^2}\,ds\\
    &\qquad+\int_0^{r/\sqrt{t}}\vartheta_\Omega(s\sqrt{t}) s^2e^{-s^2}\,ds\\
    &\qquad 
    +\frac{r}{\sqrt{t}}(1+\vartheta_\Omega(r))\int_{r/\sqrt{t}}^\infty se^{-s^2}\,ds\biggr]\,.
\end{align*}

The first integral now gives us the correct boundary term as
\begin{equation*}
	\int_0^{\infty} s^2 e^{-s^2}\,ds = \frac{\sqrt{\pi}}{4}\,.
\end{equation*}

The third integral we leave as it is. For the second and fourth we observe that
\begin{equation*}
    \int_{r/\sqrt{t}}^\infty s e^{-s^2}\,ds = \frac{e^{-r^2/t}}{2}\,,
\end{equation*}
and, since $se^{-s^2/2}\lesssim 1$,
\begin{equation*}
    \int_{r/\sqrt{t}}^\infty s^2 e^{-s^2}\,ds \lesssim \int_{r/\sqrt{t}}^\infty s e^{-s^2/2}\,ds = e^{-r^2/(2t)}\,.
\end{equation*}

Using $(r/\sqrt{t}) e^{-r^2/(2t)}\lesssim 1$ we have thus arrived at
\begin{equation}\label{eq: estimate boundary integral Lip}
\begin{aligned}
	\biggl|&\int_{\{x\in \Omega:d_\Omega(x) <r\}} e^{-d_\Omega(x)^2/t}\,dx- \frac{\sqrt{\pi t}}{2}\Haus^{d-1}(\partial\Omega)\biggr|\\
    &\quad \lesssim \sqrt{t}\Haus^{d-1}(\partial\Omega)\biggl[\int_0^{r/\sqrt{t}} s^2 e^{-s^2} |\vartheta_\Omega(s\sqrt{t})|\,ds+ (1+|\vartheta_\Omega(r)|)e^{-r^2/(2t)}\biggr]\,.
\end{aligned}
\end{equation}

\medskip

\noindent\emph{Step 3. Bounding the remainder terms.}

To bound the integrals over $\Omega_{bulk}, \Omega_{bad}, \Omega_{good}$ in~\eqref{eq: integral trace decomposition} we use the assumed bounds, obtaining that, if $r\geq \max\{1, \eta\}\sqrt{t}$, then
\begin{align*}
		\Biggl|\int_{\Omega_{bulk}} \Bigl((4\pi t)^{d/2}k_\Omega(t, x, x)-1\Bigr)\,dx\Biggr|&\leq \int_{\Omega_{bulk}} \mathcal{E}_1(t, x) e^{-c_1 d_\Omega(x)^2/t}\,dx\,,\\
		\Biggl|\int_{\Omega_{bad}} \Bigl((4\pi t)^{d/2}k_\Omega(t, x, x)-1- e^{-d_\Omega(x)^2/t}\Bigr)\,dx\Biggr|&\lesssim \int_{\Omega_{bad}} (1+\mathcal{E}_0(t, x))\,dx \,,\\
		\Biggl|\int_{\Omega_{good}} \Bigl((4\pi t)^{d/2}k_\Omega(t, x, x)-1 - e^{-d_\Omega(x)^2/t}\Bigr)\,dx\Biggr|&\leq \int_{\Omega_{good}} \mathcal{E}_2(r, t, x)e^{-c_2d_\Omega(x)^2/t}\,dx\,.
\end{align*}
Here, for the set $\Omega_{bad}$ we used the fact that
$$
1 + e^{-d_\Omega(x)^2/t} \leq 2 \,.
$$

From the inclusion $\Omega_{good}\subset \{x\in \Omega: d_\Omega(x)<r\}$ and \eqref{eq: estimate boundary integral Lip} applied with $t/c_2$ in place of $t$, it follows that
\begin{align*}
	\int_{\Omega_{good}} &e^{- c_2d_\Omega(x)^2/t}\,dx\\
    &\leq \int_{\{x\in \Omega: d_\Omega(x)<r\}} e^{- c_2 d_\Omega(x)^2/t}\,dx\\
    &\lesssim \frac{1}{\sqrt{c_2}}\sqrt{t}\Haus^{d-1}(\partial\Omega) \Biggl[1+\int_0^{\sqrt{c_2}r/\sqrt{t}} s^2 e^{-s^2} |\vartheta_\Omega(s\sqrt{t/c_2})|\,ds + |\vartheta_\Omega(r)|e^{-c_2r^2/(2t)}\Biggr]\,.
\end{align*}

Thus,
\begin{align*}
		\Biggl|\int_{\Omega_{good}} &\Bigl((4\pi t)^{d/2}k_\Omega(t, x, x)-1 - e^{-d_\Omega(x)^2/t}\Bigr)\,dx\Biggr|\\
        &\lesssim
        \|\mathcal{E}_2(r, t, \cdot)\|_{L^\infty(\Omega_{good})} \sqrt{t}\Haus^{d-1}(\partial\Omega)\\
        &\qquad \times \Biggl[\frac{1}{\sqrt{c_2}}+c_2\int_0^{r/\sqrt{t}} s^2 e^{-c_2s^2} |\vartheta_\Omega(s\sqrt{t})|\,ds+ \frac{1}{\sqrt{c_2}}|\vartheta_\Omega(r)|e^{-c_2r^2/(2t)}\Biggr]
\end{align*}

To bound the integral over $\Omega_{bulk}$ we first use the trivial bound 
\begin{align*}
    \int_{\Omega_{bulk}} \mathcal{E}_1(t, x) e^{- c_1d_\Omega(x)^2/t}\,dx
    &\leq\|\mathcal{E}_1(t, \cdot)\|_{L^\infty(\Omega_{bulk})}\int_{\Omega_{bulk}} e^{- c_1d_\Omega(x)^2/t}\,dx
\end{align*}
and then apply the layer cake formula in a similar manner as before to find
\begin{align*}
    \int_{\Omega_{bulk}}  e^{- c_1d_\Omega(x)^2/t}\,dx
    &=
    2c_1\int_{\Omega_{bulk}} \int_0^\infty \1_{\{y\in \Omega:d_\Omega(y)<s\sqrt{t}\}}(x) se^{-c_1s^2}\,dsdx\\
    &= 2c_1\int_{r/\sqrt{t}}^\infty |\{x \in \Omega: r\leq d_\Omega(x) \leq s\sqrt{t}\}|s e^{-c_1s^2}\,ds\,.
\end{align*}
The integral we bound as follows
\begin{align*}
    0&\leq \int_{r/\sqrt{t}}^\infty |\{x \in \Omega: r\leq d_\Omega(x) \leq s\sqrt{t}\}|s e^{-c_1s^2}\,ds\\
    &\leq
    \int_{r/\sqrt{t}}^\infty |\{x \in \Omega: d_\Omega(x) \leq s\sqrt{t}\}|s e^{-c_1s^2}\,ds\\
    &=
    \sqrt{t}\Haus^{d-1}(\partial\Omega)\int_{r/\sqrt{t}}^\infty \vartheta_\Omega(s\sqrt{t})s^2 e^{-c_1s^2}\,ds\\
    &\quad
    +
    \sqrt{t}\Haus^{d-1}(\partial\Omega)\int_{r/\sqrt{t}}^\infty s^2 e^{-c_1s^2}\,ds\,.
\end{align*}

Using 
\begin{equation*}
    s e^{-c s^2}\lesssim 1/\sqrt{c} \quad \forall s\geq 0\,,
\end{equation*}
we have
\begin{align*}
    \int_{r/\sqrt{t}}^\infty s^2 e^{-c_1s^2}\,ds &\lesssim \frac{1}{\sqrt{c_1}} \int_{r/\sqrt{t}}^\infty s e^{-c_1s^2/2}\,ds = \frac{1}{4c_1^{3/2}} e^{-c_1r^2/(2t)}\,.
\end{align*}
Therefore, we have arrived at
\begin{align*}
    \int_{\Omega_{bulk}}  e^{-c_1d_\Omega(x)^2/t}\,dx
    &\lesssim
    \sqrt{t}\Haus^{d-1}(\partial\Omega)\Biggl[c_1\int_{r/\sqrt{t}}^\infty |\vartheta_\Omega(s\sqrt{t})|s^2 e^{-c_1 s^2}\,ds
    +
     \frac{1}{\sqrt{c_1}}e^{-c_1r^2/(2t)}\Biggr]\,.
\end{align*}

\medskip

\noindent\emph{Step 4. Completing the proof.}

Gathering the estimates, we have proved that, if $0<\max\{1, \eta\}\sqrt t \leq r$, then
\begin{equation*}
\begin{aligned}
	\biggl|&(4\pi t)^{d/2}\Tr(e^{t\Delta_\Omega}) - |\Omega| - \frac{\sqrt{\pi t}}{2}\Haus^{d-1}(\partial\Omega)\biggr|\\
 &  \lesssim 
    \sqrt{t}\Haus^{d-1}(\partial\Omega)\Biggl(\int_0^{r/\sqrt{t}} s^2 e^{-s^2} |\vartheta_\Omega(s\sqrt{t})|\,ds+ (1+|\vartheta_\Omega(r)|)e^{-r^2/(2t)}\\
    &\quad
    +
    \|\mathcal{E}_2(r, t, \cdot)\|_{L^\infty(\Omega_{good})} \biggl[\frac{1}{\sqrt{c_2}}+c_2\int_0^{r/\sqrt{t}} s^2 e^{-c_2s^2} |\vartheta_\Omega(s\sqrt{t})|\,ds+ \frac{1}{\sqrt{c_2}}|\vartheta_\Omega(r)|e^{-c_2r^2/(2t)}\biggr]\\
    &\quad
    +
    \|\mathcal{E}_1(t, \cdot)\|_{L^\infty(\Omega_{bulk})}\biggl[c_1\int_{r/\sqrt{t}}^\infty |\vartheta_\Omega(s\sqrt{t})|s^2 e^{-c_1 s^2}\,ds 
    +
    \frac{1}{\sqrt{c_1}}e^{-c_1r^2/(2t)}\biggr]\\
    &\quad
    +\frac{1}{\sqrt{t}\Haus^{d-1}(\partial\Omega)}\int_{\Omega_{bad}} (1+\mathcal{E}_0(t, x))\,dx\Biggr)
\end{aligned}
\end{equation*}
For any $s \in (r/\sqrt{t}, 2r/\sqrt{t})$
\begin{equation*}
    \vartheta_\Omega(s\sqrt{t}) = \frac{|\{d_\Omega<s\sqrt{t}\}|}{s\sqrt{t}\Haus^{d-1}(\partial\Omega)}-1
    \geq \frac{|\{d_\Omega<r\}|}{2r\Haus^{d-1}(\partial\Omega)}-1
    = \vartheta_\Omega(r)/2-1/2\,.
\end{equation*}
Therefore, since $r\geq \max\{1, \eta\}\sqrt{t}$, we have
\begin{align*}
    \int_{r/\sqrt{t}}^{2r/\sqrt{t}}s^2 e^{-cs^2}|\vartheta_\Omega(s\sqrt{t})|\,ds  
    &\geq 
    \frac{1}{2}(|\vartheta_\Omega(r)|-1)_\limplus\int_{r/\sqrt{t}}^{2r/\sqrt{t}}s^2 e^{-cs^2}\,ds\\
    &\geq
    \frac{1}{2}(|\vartheta_\Omega(r)|-1)_\limplus \Bigl(\frac{r}{\sqrt{t}}\Bigr)^3 e^{-4cr^2/t}\\
    &\geq
    \frac{1}{2}(|\vartheta_\Omega(r)|-1)_\limplus \max\{1,\eta\}^{-3} e^{-4cr^2/t}\,.
\end{align*}
Consequently, for any $c>0$
\begin{align*}
\frac{1}{\sqrt{c}} |\vartheta_\Omega(r)|e^{-cr^2/(2t)} \lesssim \frac{1}{\sqrt{c}}\Bigl(\max\{1, \eta\}^3\int_{r/\sqrt{t}}^{2r/\sqrt{t}}s^2 e^{-cs^2/8}|\vartheta_\Omega(s\sqrt{t})|\,ds  +e^{-cr^2/(2t)}\Bigr)\,,
\end{align*}
and thus
\begin{equation*}
\begin{aligned}
	\biggl|&(4\pi t)^{d/2}\Tr(e^{t\Delta_\Omega}) - |\Omega| - \frac{\sqrt{\pi t}}{2}\Haus^{d-1}(\partial\Omega)\biggr|\\
 &  \lesssim 
    \sqrt{t}\Haus^{d-1}(\partial\Omega)\Biggl(\max\{1,\eta\}^3\int_0^{r/\sqrt{t}} s^2 e^{-s^2/8} |\vartheta_\Omega(s\sqrt{t})|\,ds+ e^{-r^2/(2t)}\\
    &\quad
    +
    \|\mathcal{E}_2(r, t, \cdot)\|_{L^\infty(\Omega_{good})} \biggl[\frac{1}{\sqrt{c_2}}+\Bigl(c_2+\frac{\max\{1,\eta\}^3}{\sqrt{c_2}}\Bigr)\int_0^{2r/\sqrt{t}} s^2 e^{-c_2s^2/8} |\vartheta_\Omega(s\sqrt{t})|\,ds\biggr]\\
    &\quad
    +
    \|\mathcal{E}_1(t, \cdot)\|_{L^\infty(\Omega_{bulk})}\biggl[c_1\int_{r/\sqrt{t}}^\infty s^2 e^{-c_1 s^2}|\vartheta_\Omega(s\sqrt{t})|\,ds 
    +
    \frac{1}{\sqrt{c_1}}e^{-c_1r^2/(2t)}\biggr]\\
    &\quad
    +\frac{1}{\sqrt{t}\Haus^{d-1}(\partial\Omega)}\int_{\Omega_{bad}} (1+\mathcal{E}_0(t, x))\,dx\Biggr)\,.
\end{aligned}
\end{equation*}
This completes the proof of the desired bound.
\end{proof}


\subsection{Proof of Theorem~\ref{thm: Brown trace asymptotics}}

In this subsection, we show how Brown's theorem (Theorem~\ref{thm: Brown trace asymptotics}) can be deduced from the heat kernel bounds in Theorems \ref{thm: diagonal bulk Lipschitz} and \ref{thm: diagonal bounds good set Lipschitz} by means of Proposition \ref{thm: trace of N heatkernel general}.

By the bounds in Theorems \ref{thm: diagonal bulk Lipschitz} and \ref{thm: diagonal bounds good set Lipschitz}, together with the bound 
\begin{equation}\label{eq: Lipschitz diagonal bound0}
    k_\Omega(t, x, x) \leq C_{\Omega,t_0} t^{-d/2}\max\bigl\{1,(t/t_0)^{d/2}\bigr\}\,,
\end{equation}
valid for any bounded, open, Lipschitz set $\Omega \subset \R^d$ and $t_0>0$ (see Lemma~\ref{heatnapriori} below), we see that the assumptions (1), (2) and (3) in Proposition \ref{thm: trace of N heatkernel general} are satisfied with $\eta=1$,
\begin{align*}
    & \mathcal E_0(t,x) := C_{\Omega,t_0} \max\bigl\{1,(t/t_0)^{d/2}\bigr\} \,,\\
    & \mathcal E_1(t,x) := C_{\Omega} \,,\\
    & \mathcal E_2(r,t,x) := C_\Omega \bigl( \mathcal E(r) + e^{-cr^2/r} \bigr)
\end{align*}
and universal constants $c_1,c_2>0$. Here $t_0>0$ is a fixed positive number and $\mathcal E$ is as in Theorem \ref{thm: diagonal bounds good set Lipschitz}. We deduce that, for $0<t\leq\min\{t_0, r^2/2\}$,
\begin{equation*}
    \begin{aligned}
	   \biggl|(4\pi t)^{d/2}&\Tr(e^{t\Delta_\Omega}) - |\Omega| - \frac{\sqrt{\pi t}}{2}\Haus^{d-1}(\partial\Omega)\biggr|\\
        & \lesssim_{\Omega,d,t_0}
        \sqrt{t}\Haus^{d-1}(\partial\Omega)\Biggl(\int_0^{r/\sqrt{t}} s^2 e^{-c s^2} |\vartheta_\Omega(s\sqrt{t})|\,ds\\
        &\quad
        +
        \Bigl(\mathcal{E}(r)+e^{-cr^2/t}\Bigr) \biggl[\int_0^{2r/\sqrt{t}} s^2 e^{-cs^2} |\vartheta_\Omega(s\sqrt{t})|\,ds+1\biggr]\\
        &\quad
        +
        \int_{r/\sqrt{t}}^\infty s^2 e^{-c s^2}|\vartheta_\Omega(s\sqrt{t})|\,ds 
        +\frac{|\{x\in \Omega: d_\Omega(x)<r\}\setminus \mathcal{G}_r|}{\sqrt{t} \, \Haus^{d-1}(\partial\Omega)}\Biggr)\,.
    \end{aligned}
\end{equation*}
The implicit dependence on $\Omega$ is only through the constants in the bounds on $k_\Omega$ in~\eqref{eq: Lipschitz diagonal bound0}, Theorem \ref{thm: diagonal bulk Lipschitz}, and Theorem \ref{thm: diagonal bounds good set Lipschitz}. 

We bound
\begin{align*}
    \int_0^{2r/\sqrt{t}} s^2 e^{-c s^2} |\vartheta_\Omega(s\sqrt{t})|\,ds &\leq \sup_{s<2r}|\vartheta_\Omega(s)|\int_0^{\infty} s^2 e^{-c s^2} \,ds \lesssim \sup_{s<2r}|\vartheta_\Omega(s)|
\end{align*}
and
\begin{align*}
    \int_{r/\sqrt{t}}^\infty s^2 e^{-c s^2}|\vartheta_\Omega(s\sqrt{t})|\,ds&\leq \sup_{s>0}|\vartheta_\Omega(s)|\int_{r/\sqrt{t}}^\infty s^2 e^{-c s^2}\,ds \lesssim\sup_{s>0}|\vartheta_\Omega(s)| e^{-c r^2/(2t)} \,.
\end{align*}
We note that
\begin{equation}
    \label{eq:thetadiscussion}
    \lim_{r\to 0^\limplus} \vartheta_\Omega(r)=0
    \qquad\text{and}\qquad
    \sup_{r>0}|\vartheta_\Omega(r)|<\infty \,.
\end{equation}
Indeed, the first relation follows from the fact that the $\Haus^{d-1}$-measure of $\partial\Omega$ coincides with its one-sided $(d-1)$-dimensional Minkowski content (see, for instance, \cite{Ambrosio_etal_08}), and the second relation is a consequence of the first combined with $|\Omega|<\infty$.

Choosing $r = M\sqrt{t}$ with $M^2\geq 2$, we have proved that
\begin{equation*}
    \begin{aligned}
	  &\frac{\bigl|(4\pi t)^{d/2}\Tr(e^{t\Delta_\Omega}) - |\Omega| - \frac{\sqrt{\pi t}}{2}\Haus^{d-1}(\partial\Omega)\bigr|}{\Haus^{d-1}(\partial\Omega) \, \sqrt{t}}\\
        &\qquad \lesssim_{\Omega,d,t_0}
        \sup_{s<M\sqrt{t}}|\vartheta_\Omega(s)|
        +
        \Bigl(\mathcal{E}(M\sqrt{t})+e^{-cM^2}\Bigr) \biggl[\sup_{s<M\sqrt{t}}|\vartheta_\Omega(s)|+1\biggr]\\
        &\qquad\quad
        +
        e^{-c M^2/2}\sup_{s>0}|\vartheta_\Omega(s)|
        +M(1+\vartheta_\Omega(M\sqrt{t}))\frac{|\{x\in \Omega: d_\Omega(x)<M\sqrt{t}\}\setminus \mathcal{G}_{M\sqrt{t}}|}{|\{x\in \Omega: d_\Omega(x)<M\sqrt{t}\}|}\,.
    \end{aligned}
\end{equation*}
We claim that the right side tends to zero, when sending first $t\to 0^\limplus$ and then $M \to \infty$. This is a consequence of \eqref{eq:thetadiscussion} and the properties of $\mathcal{E}$ and $\mathcal{G}_r$ listed in Theorem \ref{thm: diagonal bounds good set Lipschitz}, namely, that
\begin{align*}
    \lim_{r\to 0^\limplus}\mathcal{E}(r) =0 \quad \mbox{and} \quad \lim_{r\to 0^\limplus}\frac{|\{x\in \Omega: d_\Omega(x)<r\}\setminus \mathcal{G}_{r}|}{|\{x\in \Omega: d_\Omega(x)<r\}|}=0 \,.
\end{align*}
This completes the proof of Theorem \ref{thm: Brown trace asymptotics}.


\subsection{Proof of Theorem \ref{thm: main thm convex}}
\label{sec: Neumann heat trace on convex sets}

In this subsection, we prove the heat trace bound for convex sets (Theorem \ref{thm: main thm convex}). It is convenient to split the result into two assertions, depending on whether $t$ is smaller or larger than $r_{\rm in}(\Omega)^2$; see Propositions \ref{thm:heatkernelboundneumannconvex} and \ref{thm: bound large times}.

We begin with the case where $t$ is bounded by $r_{\rm in}(\Omega)^2$. The proof in this case has similarities with that of Theorem \ref{thm: Brown trace asymptotics}, namely, we will combine heat kernel bounds with the general result from Proposition~\ref{thm: trace of N heatkernel general}. The relevant heat kernel bounds appear in Theorems \ref{thm: diagonal bulk convex} and \ref{thm: diagonal kernel bounds convex}, and also in Lemma~\ref{lem: off-diagonal kernel bound convex}. In order to achieve the claimed uniformity, we will need, however, some additional geometric considerations.

\begin{proposition}\label{thm:heatkernelboundneumannconvex}
	For any $d\geq 2$ there exist constants $c_1, c_2>0$ with the following property. If $\Omega \subset \R^d$ is open, convex, and bounded and if $0<\sqrt{t}\leq c_1 r_{\rm in}(\Omega)$, then 
	\begin{equation*}
		\biggl|(4\pi t)^{d/2}\Tr(e^{t\Delta_\Omega}) - |\Omega| - \frac{\sqrt{\pi t}}{2}\Haus^{d-1}(\partial\Omega)\biggr|
   \leq c_2 \Haus^{d-1}(\partial\Omega)\sqrt{t}\Bigl(\frac{\sqrt{t}}{r_{\rm in}(\Omega)}\Bigr)^{1/2}\ln\Bigl(\frac{r_{\rm in}(\Omega)^2}{t}\Bigr)^{1/2}\,.
	\end{equation*}
\end{proposition}

In the proof, and in the rest of this paper, we use the notation
$$
V_\Omega(x,r):=|\Omega\cap B_r(x)| \,.
$$

\begin{proof}
    For any $\epsilon \in (0, 1/4]$ and $r<\epsilon r_{\rm in}(\Omega)/2$, Theorems \ref{thm: diagonal bulk convex} and \ref{thm: diagonal kernel bounds convex}, as well as Lemma~\ref{lem: off-diagonal kernel bound convex}, allow us to apply Proposition \ref{thm: trace of N heatkernel general} with $\eta=1$, $\mathcal{G}_r = \mathcal{G}_{\epsilon, 2r}$, $c_1=1/2$, $c_2$ being a universal positive constant,
\begin{equation*}
    \mathcal{E}_0(t, x) \lesssim_d\frac{t^{d/2}}{V_\Omega(x, \sqrt{t})}\,, \quad \mathcal{E}_1(t, x) \lesssim_d 1\,, \quad \mbox{and}\quad \mathcal{E}_2(r, t, x) \lesssim_d \epsilon + e^{-cr^2/t}\,.
\end{equation*}
The conclusion is that there is a universal constant $c>0$ so that, for any $0<\sqrt{t} \leq r$,
\begin{equation}\label{eq: initial error bound convex}
\begin{aligned}
	   \biggl|(4\pi t)^{d/2}&\Tr(e^{t\Delta_\Omega}) - |\Omega| - \frac{\sqrt{\pi t}}{2}\Haus^{d-1}(\partial\Omega)\biggr|\\
        & \lesssim_{d}
        \sqrt{t}\Haus^{d-1}(\partial\Omega)\Biggl(\int_0^{r/\sqrt{t}} s^2 e^{-c s^2} |\vartheta_\Omega(s\sqrt{t})|\,ds \\
        &\quad
        +
        \bigl(\epsilon +e^{-cr^2/t}\bigr)\biggl[\int_0^{2r/\sqrt{t}} s^2 e^{-cs^2} |\vartheta_\Omega(s\sqrt{t})|\,ds+1\biggr]\\
        &\quad
        +
        \int_{r/\sqrt{t}}^\infty s^2 e^{-c s^2}|\vartheta_\Omega(s\sqrt{t})|\,ds \\
        &\quad
        +\frac{1}{\sqrt{t}\Haus^{d-1}(\partial\Omega)}\int_{\{x\in \Omega: d_\Omega(x)<r\}\setminus \mathcal{G}_{\epsilon,2r}} \Bigl(1+\frac{t^{d/2}}{V_\Omega(x, \sqrt{t})}\Bigr)\,dx\Biggr)\\
        & \lesssim_{d}
        \sqrt{t}\Haus^{d-1}(\partial\Omega)\Biggl(\epsilon + e^{-cr^2/t}+\int_0^{\infty} s^2 e^{-c s^2} |\vartheta_\Omega(s\sqrt{t})|\,ds  \\
        &\quad
        +\frac{1}{\sqrt{t}\Haus^{d-1}(\partial\Omega)}\int_{\{x\in \Omega: d_\Omega(x)<r\}\setminus \mathcal{G}_{\epsilon,2r}} \frac{t^{d/2}}{V_\Omega(x, \sqrt{t})}\,dx\Biggr)\,,
\end{aligned}
\end{equation}
where we used the fact that $V(x, \sqrt{t})\leq |B_1|t^{d/2}$ for all $x\in \Omega$. 

To complete the proof, it remains to first bound the two remaining integrals and then choose the parameters $\epsilon$ and $r$ appropriately.

\medskip
\noindent\emph{Step 1. Bounding the integral of $\vartheta_\Omega$.}

Let us begin with the integral involving $\vartheta_\Omega$. By Lemma \ref{lem: volume bdry neighborhood bound} it holds that
\begin{equation*}
    \Bigl(1- \frac{s}{r_{\rm in}(\Omega)}\Bigr)^{d-1}_\limplus-1 \leq \vartheta_\Omega(s)\leq 0\,.    
\end{equation*}
Since $d\geq 2$, the function $[0, 1]\ni x \mapsto (1-x)^{d-1}-1$ is convex. It follows that
\begin{equation*}
    - \frac{(d-1)s}{r_{\rm in}(\Omega)} \leq \vartheta_\Omega(s)\leq 0\,.    
\end{equation*}
Consequently,
\begin{equation}\label{eq: vartheta integral bound}
    \int_0^{\infty} s^2 e^{-c s^2} |\vartheta_\Omega(s\sqrt{t})|\,ds \lesssim_d \frac{\sqrt{t}}{r_{\rm in}(\Omega)}\int_0^{\infty} s^3 e^{-c s^2}\,ds \lesssim_d \frac{\sqrt{t}}{r_{\rm in}(\Omega)}\,.
\end{equation}

\medskip

\noindent\emph{Step 2. Bounding the integral of $1/V_\Omega$.}

Let us consider the remaining integral in \eqref{eq: initial error bound convex}. For notational simplicity, we again write
\begin{equation*}
    \Omega_{bad}:= \{x\in \Omega: d_\Omega(x) <r\}\setminus \mathcal{G}_{\epsilon, 2r} \,.
\end{equation*}
By \eqref{eq: Good set is big0} and Lemma \ref{lem: volume bdry neighborhood bound}, for any $\epsilon \in (0, 1/4], r \in (0, \epsilon r_{\rm in}(\Omega)/2]$ we have
\begin{equation}\label{eq: bad set volume bound}
 	|\Omega_{bad}|\lesssim_d \Haus^{d-1}(\partial\Omega) \frac{r^2}{\epsilon r_{\rm in}(\Omega)}\,.
\end{equation} 
The second bound in Proposition \ref{prop: bounds integral of local volume} gives
\begin{equation*}
    \frac{1}{\sqrt{t}\Haus^{d-1}(\partial\Omega)}\int_{\Omega_{bad}}\frac{t^{d/2}}{V_\Omega(x, \sqrt{t})}\,dx
    \lesssim_d \frac{|\Omega_{bad}|}{\sqrt{t}\Haus^{d-1}(\partial\Omega)}+ \frac{\sqrt{t}}{r_{\rm in}(\Omega)}\Bigl(1+\Bigl(\frac{\sqrt{t}}{r_{\rm in}(\Omega)}\Bigr)^{d-2}\Bigr)\,.
\end{equation*}
By \eqref{eq: bad set volume bound} and the assumption $\sqrt{t}\leq r \leq \epsilon r_{\rm in}(\Omega)/2$, we arrive at the estimate
\begin{equation}\label{eq: final bound remaining integral error term}
	\frac{1}{\sqrt{t}\Haus^{d-1}(\partial\Omega)} \int_{\Omega_{bad}}\frac{t^{d/2}}{V_\Omega(x, \sqrt{t})}\,dx \lesssim_d \frac{r^2}{\epsilon \sqrt{t}r_{\rm in}(\Omega)}  + \frac{\sqrt{t}}{r_{\rm in}(\Omega)} \lesssim_d \frac{r^2}{\epsilon \sqrt{t} \,r_{\rm in}(\Omega)}\,.
\end{equation}

\medskip

\noindent\emph{Step 3. Completing the proof.}

Inserting~\eqref{eq: vartheta integral bound} and \eqref{eq: final bound remaining integral error term} into \eqref{eq: initial error bound convex} yields, for $\epsilon \in (0, 1/4]$, $r\in (0, \epsilon r_{\rm in}(\Omega)/2]$ and $0<\sqrt{t}\leq  r$,
\begin{align*}
	\biggl|(4\pi t)^{d/2}\Tr(e^{t\Delta_\Omega})& - |\Omega| - \frac{\sqrt{\pi t}}{2}\Haus^{d-1}(\partial\Omega)\biggr|\\
 &  \lesssim_{d} \Haus^{d-1}(\partial\Omega)\sqrt{t}\biggl(\epsilon+ e^{-cr^2/t} + \frac{\sqrt{t}}{r_{\rm in}(\Omega)}+ \frac{r^2}{\epsilon \sqrt{t} \,r_{\rm in}(\Omega)}\biggr)\\
 &\lesssim_d
    \Haus^{d-1}(\partial\Omega)\sqrt{t}\biggl(\epsilon+ e^{-cr^2/t} + \frac{r^2}{\epsilon \sqrt{t} \,r_{\rm in}(\Omega)}\biggr)\,,
\end{align*}
where in the second step we used $\sqrt t/r_{\rm in}(\Omega) \leq r^2/(\sqrt t r_{\rm in}(\Omega)) \leq r^2/(4\epsilon \sqrt t \, r_{\rm in}(\Omega))$.

We now want to choose our parameters to minimize the right-hand side. 

Setting $r= M\sqrt{t}$ for some $1\leq M\leq \epsilon r_{\rm in}(\Omega)/(2\sqrt{t})$, we find that
\begin{align*}
	\biggl|(4\pi t)^{d/2}\Tr(e^{t\Delta_\Omega})& - |\Omega| - \frac{\sqrt{\pi t}}{2}\Haus^{d-1}(\partial\Omega)\biggr|\\
 &  \lesssim_{d} \Haus^{d-1}(\partial\Omega)\sqrt{t}\biggl(\epsilon+ e^{-cM^2} + \frac{\sqrt{t}}{r_{\rm in}(\Omega)}\frac{M^2}{\epsilon}\biggr)\,.
\end{align*}
Choosing $\epsilon = M t^{1/4}r_{\rm in}(\Omega)^{-1/2}$ results in the bound
\begin{align*}
	\biggl|(4\pi t)^{d/2}\Tr(e^{t\Delta_\Omega})& - |\Omega| - \frac{\sqrt{\pi t}}{2}\Haus^{d-1}(\partial\Omega)\biggr|\\
 &  \lesssim_{d} \Haus^{d-1}(\partial\Omega)\sqrt{t}\biggl(e^{-cM^2} + M\Bigl(\frac{\sqrt{t}}{r_{\rm in}(\Omega)}\Bigr)^{1/2}\biggr)\,,
\end{align*}
valid for $t> 0$ and $M\geq 1$ such that
\begin{align*}
    M\leq \frac{1}{2} M t^{-1/4}r_{\rm in}(\Omega)^{1/2} \quad \mbox{and} \quad  M t^{1/4}r_{\rm in}(\Omega)^{-1/2}\leq 1/4
\end{align*}
or, equivalently, if
\begin{equation*}
    \frac{\sqrt{t}}{r_{\rm in}(\Omega)} \leq \min\Bigl\{ \frac{1}{4}, \frac{1}{16M^2}\Bigr\}\,.
\end{equation*}
To optimize the choice of $M$, we need to choose it so that
\begin{equation*}
    e^{-cM^2} \sim M\Bigl(\frac{\sqrt{t}}{r_{\rm in}(\Omega)}\Bigr)^{1/2}\,.
\end{equation*}
By taking logarithms of both sides, this suggests
\begin{equation*}
    M^2 \sim \log\Bigl(\frac{r_{\rm in}(\Omega)^2}{t}\Bigr)- \log(M)\,,
\end{equation*}
which as $t \to 0$ yields the choice $M = C \sqrt{\log(r_{\rm in}(\Omega)^2/t)}$ for some $C>0$. With this choice of $M$ we have proved that
\begin{align*}
	\biggl|(4\pi t)^{d/2}\Tr(e^{t\Delta_\Omega})& - |\Omega| - \frac{\sqrt{\pi t}}{2}\Haus^{d-1}(\partial\Omega)\biggr|\\
 &  \lesssim_{d} \Haus^{d-1}(\partial\Omega)\sqrt{t}\biggl(\Bigl(\frac{t}{r_{\rm in}(\Omega)^2}\Bigr)^{cC^2} + C \Bigl(\frac{\sqrt{t}}{r_{\rm in}(\Omega)}\Bigr)^{1/2} \Bigl(\log\Bigl(\frac{r_{\rm in}(\Omega)^2}{t}\Bigr)\Bigr)^{1/2}\biggr)\,,
\end{align*}
valid for $t> 0$ satisfying
\begin{equation*}
    \frac{\sqrt{t}}{r_{\rm in}(\Omega)} \leq \frac{1}{4} \quad \mbox{and}\quad  \frac{\sqrt{t}}{r_{\rm in}(\Omega)}\log\Bigl(\frac{r_{\rm in}(\Omega)^2}{t}\Bigr) \leq\frac{1}{16C^2 }\,.
\end{equation*}
Choosing $C= 1/\sqrt{2c}$ and using the fact that $s\log(1/s^2)$ tends to zero as $s\to 0^\limplus$, completes the proof of the proposition.
\end{proof}

Next, we prove a bound on the heat trace for $t$ that are bounded from below by some multiple of $r_{\rm in}(\Omega)^2$.

\begin{proposition}\label{thm: bound large times}
	Let $d\geq 2$. For every constant $c_1>0$ there is a constant $c_2>0$ with the following property. If $\Omega \subset \R^d$ is open, convex, and bounded and $\sqrt{t}\geq c_1 r_{\rm in}(\Omega)$, then 
	\begin{equation*}
		\biggl|(4\pi t)^{d/2}\Tr(e^{t\Delta_\Omega}) - |\Omega| - \frac{\sqrt{\pi t}}{2}\Haus^{d-1}(\partial\Omega)\biggr|
   \leq c_2 \Haus^{d-1}(\partial\Omega)\sqrt{t} \Bigl(\frac{\sqrt{t}}{r_{\rm in}(\Omega)}\Bigr)^{d-1}\,.
	\end{equation*}
\end{proposition}

The power of $t$ in the right-hand side of the estimate is sharp since $\Tr(e^{t\Delta_\Omega})\geq 1$ for every $t>0$, so when $t$ becomes large the first term on the left side dominates the entire left-hand side.

\begin{proof}
	Let us begin by bounding the quantity in the absolute values from below. We use the bound $(4\pi t)^{d/2}\Tr(e^{t\Delta_\Omega})\geq |\Omega|$, which can be deduced from Kr\"oger's inequality~\cite{Kroger92} in the form of \cite[Theorem 3.37]{LTbook} via the integral representation formula
    \begin{equation*}
    \Tr(e^{-tH}) = t^2 \int_{\R} e^{-\lambda t} \Tr(H-\lambda)_\limminus\,d\lambda \,.
    \end{equation*}
    Because of this lower bound on the heat trace and the assumption $\sqrt{t}\geq c_1 r_{\rm in}(\Omega)$, we find
	\begin{align*}
		(4\pi t)^{d/2}\Tr(e^{t\Delta_\Omega}) - |\Omega| - \frac{\sqrt{\pi t}}{2}\Haus^{d-1}(\partial\Omega)
		&>
		 - \frac{\sqrt{\pi t}}{2}\Haus^{d-1}(\partial\Omega)\\
		 &\geq
		  - \frac{\sqrt{\pi}}{c_1^{d-1}2}\sqrt{t}\Haus^{d-1}(\partial\Omega) \Bigl(\frac{\sqrt{t}}{r_{\rm in}(\Omega)}\Bigr)^{d-1}\,.
	\end{align*}

	For the upper bound, we drop the negative terms and use the bound in Lemma \ref{lem: off-diagonal kernel bound convex} to estimate
	\begin{equation*}
		(4\pi t)^{d/2}\Tr(e^{t\Delta_\Omega}) - |\Omega| - \frac{\sqrt{\pi t}}{2}\Haus^{d-1}(\partial\Omega) \lesssim_d \int_\Omega \frac{t^{d/2}}{V_\Omega(x, \sqrt{t})}\,dx\,.
	\end{equation*}

	By Propositions~\ref{prop: Minkowski sum bounds} and \ref{prop: bounds integral of local volume}, there exists $c_2$ depending only on $d, c_1$ so that 
\begin{align*}
	\int_\Omega \frac{t^{d/2}}{V_\Omega(x, \sqrt{t})}\,dx \leq \frac{4^d}{|B_1(0)|} |\Omega + B_{\sqrt{t}}| 
	&\leq
	c_2 \sqrt{t} \, \Haus^{d-1}(\partial\Omega) \Bigl(\frac{\sqrt{t}}{r_{\rm in}(\Omega)}\Bigr)^{d-1}
\end{align*}
for all $\sqrt{t}\geq c_1 r_{\rm in}(\Omega)$.
Therefore, 
\begin{equation*}
	(4\pi t)^{d/2}\Tr(e^{t\Delta_\Omega}) - |\Omega| - \frac{\sqrt{\pi t}}{2}\Haus^{d-1}(\partial\Omega)\leq c_2\sqrt{t}\Haus^{d-1}(\partial\Omega) \Bigl(\frac{\sqrt{t}}{r_{\rm in}(\Omega)}\Bigr)^{d-1}\,.
\end{equation*}
 Combining the upper and lower bounds completes the proof.
\end{proof}

\begin{proof}[Proof of Theorem~\ref{thm: main thm convex}]
	The claimed bound follows upon combining Propositions~\ref{thm:heatkernelboundneumannconvex} and~\ref{thm: bound large times} and the estimate
	\begin{equation*}
		\ln\Bigl(\frac{r_{\rm in}(\Omega)^2}{t}\Bigr) \leq \frac{2}{\epsilon e} \Bigl(\frac{\sqrt{t}}{r_{\rm in}(\Omega)}\Bigr)^{-\epsilon}\quad \mbox{for all }t>0\,.\qedhere
	\end{equation*}
\end{proof}

\section{Initial bounds on the heat kernel}

The proofs of our approximation results for the heat kernel will be based on using integral representations for solutions of the heat equation in terms of the heat kernel combined with an initial bound on this kernel. Depending on the assumptions on $\Omega$ these initial bounds will differ, but they will in each case quantify how large the kernel $k_\Omega$ can be in comparison to the heat kernel in the full space 
\begin{equation*}
    k_{\R^d}(t, x, y) = \frac{1}{(4\pi t)^{d/2}} \, e^{- \frac{|x-y|^2}{4t}}\,.
\end{equation*}

To this end, we find it convenient to introduce the following family of weighted norms. For $f\colon (0, T]\times \Omega \times \Omega\to \R$ define for $x\in \Omega$ and $P\subseteq (0, T]\times \R^d$
\begin{equation}\label{eq:gnorm}
    \|f\|_{G^\infty_{\delta,x}(P)} := \sup_{(s, y) \in P \cap ((0, T]\times \overline{\Omega})}\limsup_{\Omega\ni y' \to y} |f(s, x, y')|s^{d/2}e^{\frac{|x-y|^2}{4(1+\delta)s}}\,.
\end{equation}
Note that this is a function of $x$ and, if $f$ extends continuously to $\partial\Omega$, then the limsup in the definition can be replaced by the limit.

In other words, these norms are defined so that they correspond to the best possible Gaussian upper bound on $k_\Omega$ that is uniform for $(s, y) \in P$:
\begin{equation*}
    k_\Omega(t, x, y) \leq \|k_\Omega\|_{G^\infty_{\delta, x}(P)}t^{-d/2} e^{-\frac{|x-y|^2}{4(1+\delta)t}}\,.
\end{equation*}

A well-known result of Davies' states that $k_\Omega$ is essentially bounded by a multiple of $k_{\R^d}$, provided that $\partial\Omega$ is regular enough. In fact, Davies' result is more general, covering heat kernels associated to uniformly elliptic operators. 
\begin{lemma}[{\cite[Theorem 2.4.4]{Davies_heatkernels}}]\label{heatnapriori}
    Let $d\geq 1$ and let $\Omega\subset\R^d$ be an open set with the extension property. For any $t_0,\delta>0$ there is a constant $C_{\Omega,t_0,\delta}$ such that, for all $x, y \in\Omega$ and $t>0$,
    $$
    0\leq k_\Omega(t, x, y) \leq C_{\Omega,t_0,\delta} t^{-d/2}\max\bigl\{1,(t/t_0)^{d/2}\bigr\} \, e^{- \frac{|x-y|^2}{4(1+\delta)t}} \,.
    $$
\end{lemma}

In Theorem 2.4.4 of~\cite{Davies_heatkernels} the result is stated for heat kernels associated to general uniformly elliptic operators on a connected set $\Omega$ with the extension property and for $0<t<1$. For a connected set $\Omega$, the bound claimed above follows by applying Davies' result with the operator $-t_0\Delta_\Omega$ and then extend the obtained bound to all $t$ by using the monotonicity argument in the proof of~\cite[Theorem 3.2.9]{Davies_heatkernels}. The extension to non-connected sets follows by applying Davies' bound in each of the connected components and using the fact that the implied constant appearing in Davies' proof only depends on the norm of the extension operator.

An immediate consequence of Lemma \ref{heatnapriori} that will be useful here is the following corollary.
\begin{corollary}\label{cor: unif kernel bound Lip}
    Let $d\geq 1$ and let $\Omega\subset\R^d$ be an open set with the extension property. For any $t_0,\delta>0$ there is a constant $C_{\Omega,t_0,\delta}$ such that, for all $x \in\Omega$, $t>0$ and $P \subseteq (0, t]\times \overline{\Omega}$,
    $$
    \|k_\Omega\|_{G^\infty_{\delta, x}(P)} \leq C_{\Omega,t_0,\delta} \max\bigl\{1,(t/t_0)^{d/2}\bigr\} \,.
    $$
\end{corollary}

Since bounded convex sets are uniformly Lipschitz, Corollary \ref{cor: unif kernel bound Lip} is applicable in this setting. However, to obtain the estimate in Theorem~\ref{thm: main thm convex}, this uniform bound in $x$ is not good enough (at least not with our methods). Indeed, as we shall see below, the relative error,
$$
    \frac{k_\Omega(t, x, y)}{k_{\R^d}(t, x, y)}\,, 
$$
is large when $x, y$ are close to `sharp corners' but small close to more regular parts of the boundary. To prove our main result, we need to capture this better behavior close to regular parts of the boundary and combine this with the fact that `sharp corners' cannot cover a too large portion of $\partial\Omega$. The precise behavior of $k_\Omega$ close to the boundary of a convex set was studied by Li and Yau in \cite{LiYau86} where they obtained matching upper and lower bounds in terms of the quantity $$V_\Omega(x, r) = |\Omega \cap B_r(x)|\,,$$ see also \cite[Theorem 5.5.6]{Davies_heatkernels}. For our purposes it suffices to recall the upper bound of Li and Yau.

\begin{lemma}[{\cite{LiYau86}}]\label{lem: off-diagonal kernel bound convex}
	Let $d\geq2$ and $\delta >0$. There exists positive constants $c_{d,\delta},C_{d,\delta}$ so that, for any open and convex set $\Omega\subset \R^d$ and any $x,y\in\Omega$ and $t>0$,
 \begin{equation*}
	0
    \leq 
	k_\Omega(t, x, y)
	\leq C_{d,\delta} \,
	\frac{e^{-\frac{|x-y|^2}{4(1+\delta)t}}}{V_\Omega(x, \sqrt{t})}\,.
\end{equation*}
\end{lemma}

In \cite{LiYau86} the results are stated and proved for bounded and smooth convex sets. Both of these additional assumptions can be removed by an approximation argument using that the heat kernel is the unique weak solution of a heat equation.

In terms of the $G^\infty_{\delta,x}$-norms, the bound of Li and Yau has the following consequence.
\begin{corollary}\label{cor: unif kernel bound convex}
    Let $d\geq 1$ and let $\Omega\subset\R^d$ be open and convex. For any $\delta>0$ there is a constant $C_{d,\delta}$ such that, for all $x \in\Omega$, $t>0$ and $P \subseteq (0, t]\times \overline{\Omega}$,
    $$
    \|k_\Omega\|_{G^\infty_{\delta, x}(P)} \leq C_{d,\delta} \, \frac{t^{d/2}}{V_\Omega(x, \sqrt{t})}\,.
    $$
\end{corollary}


\section{An improved heat kernel bound in the bulk}

In this section, we show that an a priori Gaussian-like upper bound on the heat kernel $k_\Omega$ implies precise estimates for $k_\Omega(t, x, x)$ as $t\to 0$ as long as the distance from $x$ to the boundary is $\gg \sqrt{t}$. The proof provided does not depend on us considering the heat kernel of the Neumann realization of the Laplace operator but works for the heat kernel associated to any nonnegative self-adjoint extension $H$ of $-\Delta$ on $L^2(\Omega)$ with $C_0^\infty(\Omega)\subset \dom(H)$. We denote the corresponding kernel by $k_H$ whose existence follows from, for instance, \cite[Proposition 2.12]{FrankLarson_Inventiones25}.

We recall the definition of the $G^\infty_{\delta,x}$-norm from \eqref{eq:gnorm}.

\begin{proposition}\label{prop: improved bound bulk}
    Let $\kappa_d=0$ for $d\geq 3$ and $\kappa_2 =1$. For all $x\in \Omega, t>0$, $\delta'>\delta>0$ and $R \in (0, d_\Omega(x)]$, it holds that
    \begin{equation*}
		\bigl|(4\pi t)^{d/2}k_H(t, x, x)- 1\bigr| \lesssim_{d,\delta,\delta'} \Bigl(1+\log_\limplus\Bigl(\frac{\sqrt{t}}{R}\Bigr) \Bigr)^{\kappa_d} \|k_H\|_{G^\infty_{\delta,x}((0, t]\times B_{R}(x))} e^{- \frac{R^2}{(1+\delta') t}}\,.
	\end{equation*}
\end{proposition}
We emphasize that in Proposition~\ref{prop: improved bound bulk} all the dependence on the details of the operator $H$ and the geometry of $\Omega$ is captured by the weighted norm in the right-hand side of the bound. We also note that the bound is only interesting in the regime when $\sqrt{t}\ll R\leq  d_\Omega(x)$, as otherwise the bound is a direct consequence of the pointwise Gaussian bound that holds by the definition of the weighted norm.

For the case of the Neumann Laplace operator in a bounded Lipschitz set $\Omega \subset \R^d$ a bound corresponding to that in Proposition~\ref{prop: improved bound bulk} with $R \propto d_\Omega(x)$ appears in \cite[Proposition 2.1 (2.4+)]{Brown93}. Apart from giving a bound for heat kernels associated to a much larger class of Laplace operators, the estimate in Proposition~\ref{prop: improved bound bulk} improves that of Brown in several regards; (i) all the dependence on the operator $H$ (and thus on the geometry $\Omega$) is in the right-hand side contained in the a priori estimate for $k_H$ given by the weighted norm, (ii) the bound obtained by Brown degenerates close to the boundary as an arbitrary positive power of $\sqrt{t}/d_\Omega(x)$ whilst our bound degenerates logarithmically when $d=2$ and does not degenerate when $d\geq 3$. 

In \cite[Proposition 2.1 (2.4+)]{Brown93} the factor $\|k_\Omega\|_{G_{\delta,x}^\infty((0,t]\times B_{d_\Omega(x)}(x))}$ appearing in Proposition~\ref{prop: improved bound bulk} is replaced by 1. When $|\Omega|<\infty$ such a bound cannot be true for all $t>0$, since it can be deduced from the spectral theorem that $k_\Omega(t, x, x)\geq 1/|\Omega|$. However, since the main result in \cite{Brown93} concerns
small $t > 0$, this omission does not influence the validity of the main result of \cite{Brown93}.

\begin{proof}
    Fix $x\in \Omega$ and $t>0$ as in the statement of the proposition. Let $\chi \in C_0^\infty(\Omega)$ with $\chi \equiv 1$ in a neighborhood of $x$. Define $w\colon (0, \infty)\times \Omega \to \R$ by 
	\begin{equation*}
		w(t', x') := k_H(t', x', x) - k_{\R^d}(t', x', x)\chi(x')\,.
	\end{equation*}
	Note that $w$ solves
	\begin{equation*}
		\begin{cases}
			(\partial_{t'}-\Delta)w(t', x') = -2\nabla k_{\R^d}(t', x', x) \cdot \nabla \chi(x')\\
            \hphantom{(\partial_{t'}-\Delta)w(t', x') =}-k_{\R^d}(t', x', x)\Delta \chi(x'),  & \mbox{for }(t', x')\in (0, \infty)\times \Omega\,,\\
			\lim_{t'\to 0^\limplus }w(t', x') =0 &\mbox{for }x'\in \Omega\,,\\
		\end{cases}
	\end{equation*}
    and, since $w$ agrees with $k_H(\cdot,\cdot,x)$ in a neighborhood of $\partial\Omega$, it satisfies the same boundary conditions as $k_H(\cdot,\cdot,x)$.

	By Duhamel's principle, the function $w$ can be expressed as an integral, namely
	\begin{equation*}
		w(t', x') = -\int_0^{t'}\int_{\Omega}k_H(t'-s, x', z)\bigr(2\nabla k_{\R^d}(s, z, x) \cdot \nabla \chi(z)+k_{\R^d}(s, z, x)\Delta \chi(z)\bigr)\,dzds\,.
	\end{equation*}
	Therefore, by H\"older's inequality,
	\begin{equation*}
		|w(t', x')| \leq \|k_H\|_{G^\infty_{\delta,x'}(P)} \int_0^{t'}\int_{\Omega}\frac{e^{-\frac{|x'-z|^2}{4(1+\delta)(t'-s)}}}{(t'-s)^{d/2}}\frac{e^{-\frac{|z-x|^2}{4s}}}{s^{d/2}}\Bigr(2 \frac{|z-x|}{s}|\nabla \chi(z)|+|\Delta \chi(z)|\Bigr)\,dzds\,,
	\end{equation*}
    with $P := (0, t']\times \supp(|\nabla \chi|+|\Delta \chi|)$.
    Since $w(x)=1$ it holds that
    $$
        k_H(t,x,x)-k_{\R^d}(t, x, x)= w(t, x) \,.
    $$ 
    Therefore, we are done if we show that the right-hand side is appropriately bounded for a suitable choice of $\chi$. 

    Fix $\epsilon \in (0, 1/2)$ and $\chi_0 \in C_0^\infty(\R^d)$ with $\supp\chi_0\subset B_{1}(0)$, $\chi_0\equiv 1$ in $B_{1-\epsilon}(0)$, and $0\leq \chi_0\leq 1$. Set $\chi(z):=\chi_0((x-z)/R)$ and note that
    \begin{align*}
        |\nabla \chi(z)| &\lesssim_{\epsilon} R^{-1}\1_{A}(z)\,,\\
        |\Delta \chi(z)| &\lesssim_{\epsilon} R^{-2}\1_{A}(z)\,,
    \end{align*}
    where $A:= \{y\in \R^d: (1-\epsilon)R< |x-y|< R\}\subset \Omega$.

    To bound $w(t, x)$, we estimate the integral
    \begin{align*}
		 \int_0^t\int_{\Omega}&\frac{e^{-\frac{|x-z|^2}{4(1+\delta)(t-s)}}}{(t-s)^{d/2}}\frac{e^{-\frac{|z-x|^2}{4s}}}{s^{d/2}}\Bigr(2 \frac{|z-x|}{s}|\nabla \chi(z)|+|\Delta \chi(z)|\Bigr)\,dzds\\
         &\lesssim_{\epsilon}\int_0^t\int_{A}\frac{e^{-\frac{|x-z|^2}{4(1+\delta)(t-s)}}}{(t-s)^{d/2}}\frac{e^{-\frac{|z-x|^2}{4s}}}{s^{d/2}}\Bigr( \frac{|z-x|}{sR}+\frac{1}{R^2}\Bigr)\,dzds\,.
	\end{align*}

    From the elementary inequality
	\begin{equation*}
		h^\beta e^{- \gamma h^2} \leq \Bigl(\frac{\beta}{\gamma \sqrt{2e}}\Bigr)^{\beta} \quad \mbox{for all }h\geq 0 \mbox{ and } \beta, \gamma>0\,,
	\end{equation*}
	it follows that for any $s>0, z \in A$,
	\begin{equation*}
		\frac{|z-x|e^{-\frac{|z-x|^2}{4s}}}{sR} \leq \frac{e^{-\frac{|z-x|^2}{4(1+\delta)s}}}{(1-\epsilon)R^2}\frac{|z-x|^2e^{-\frac{\delta|z-x|^2}{4(1+\delta) s}}}{s} \lesssim_\delta \frac{e^{-\frac{|z-x|^2}{4(1+\delta)s}}}{(1-\epsilon)R^2}\,,
	\end{equation*}
    and thus, for any $\epsilon_0\in (0, 1)$,
    \begin{equation}\label{eq: bulk relative error}
    \begin{aligned}
		 \int_0^t&\int_{\Omega}\frac{e^{-\frac{|x-z|^2}{4(1+\delta)(t-s)}}}{(t-s)^{d/2}}\frac{e^{-\frac{|z-x|^2}{4s}}}{s^{d/2}}\Bigr(2 \frac{|z-x|}{s}|\nabla \chi(z)|+|\Delta \chi(z)|\Bigr)\,dzds\\
         &\lesssim_{\epsilon,\delta}\frac{1}{R^{2}}\int_0^t\int_{A}\frac{e^{-\frac{|x-z|^2}{4(1+\delta)(t-s)}-\frac{|z-x|^2}{4(1+\delta)s}}}{(t-s)^{d/2}s^{d/2}}\,dzds\\
         &=
         \frac{1}{R^{2}t^{d-1}}\int_0^1\int_{\S^{d-1}}\int_{(1-\epsilon)R}^{R}\frac{e^{-\frac{r^2}{4(1+\delta)t}(\frac{1}{1-\tau}+\frac{1}{\tau})}}{(1-\tau)^{d/2}\tau^{d/2}}r^{d-1}\,drd\Haus^{d-1}(\theta)d\tau\\
         &\lesssim_{\epsilon,\delta}
         \frac{R^{d-2}}{t^{d-1}}e^{-\frac{(1-\epsilon)^2(1-\epsilon_0)R^2}{(1+\delta)t}}\int_0^1\frac{e^{-\frac{(1-\epsilon)^2\epsilon_0 R^2}{4(1+\delta)t}(\frac{1}{1-\tau}+\frac{1}{\tau})}}{(1-\tau)^{d/2}\tau^{d/2}}\,d\tau\,.
    \end{aligned}
    \end{equation}

    By symmetry and the change of variables $\eta = \frac{1}{1-\tau}+\frac{1}{\tau}$ it holds that, for any $M>0$,
    \begin{align*}
        \int_0^1 \frac{e^{-M(\frac{1}{1-\tau}+\frac{1}{\tau})}}{(1-\tau)^{d/2}\tau^{d/2}}\,d\tau
        &=
        2\int_0^{1/2} \frac{e^{-M(\frac{1}{1-\tau}+\frac{1}{\tau})}}{(1-\tau)^{d/2}\tau^{d/2}}\,d\tau\\
        &\leq
        2^{1+d/2}e^{-2M}\int_0^{1/2} \frac{e^{-M/\tau}}{\tau^{d/2}}\,d\tau\,.
    \end{align*}
    By a change of variables, we have
    \begin{align*}
        \int_0^{1/2} \frac{e^{-M/\tau}}{\tau^{d/2}}\,d\tau 
        &= M^{1-d/2}\int_0^{1/(2M)} \frac{e^{-1/s}}{s^{d/2}}\,ds\,,
    \end{align*}
    which leads to the bound
    \begin{equation*}
        \int_0^{1/2} \frac{e^{- M/\tau}}{\tau^{d/2}}\,d\tau 
        \lesssim_d 
        \begin{cases}
            M^{1-d/2} & \mbox{for }d\geq 3\,,\\
            1+\log_\limplus(1/M) & \mbox{for }d=2\,.
        \end{cases}
    \end{equation*}%

    Putting these bounds back in \eqref{eq: bulk relative error} yields
    \begin{equation*}
    \begin{aligned}
        \int_0^t&\int_{\Omega}\frac{e^{-\frac{|x-z|^2}{4(1+\delta)(t-s)}}}{(t-s)^{d/2}}\frac{e^{-\frac{|z-x|^2}{4s}}}{s^{d/2}}\Bigr(2 \frac{|z-x|}{s}|\nabla \chi(z)|+|\Delta \chi(z)|\Bigr)\,dzds\\
        &\lesssim_{\epsilon,\delta,d}
         t^{-d/2}
         \Bigl(1+\log_\limplus\Bigl(\frac{\sqrt{t}}{R}\Bigr)\Bigr)^{\kappa_d}
         e^{-\frac{(1-\epsilon)^2(1-\epsilon_0)R^2}{(1+\delta)t}}
    \end{aligned}
    \end{equation*}
    where for $\kappa_d = 0$ when $d\geq 3$ and $\kappa_2=1$.

    We have thus arrived at
    \begin{align*}
        |k_H(t&, x, x)-k_{\R^d}(t, x, x)| \\
        &\lesssim_{d, \epsilon, \delta, \epsilon_0} t^{-d/2}\|k_H\|_{G^{\infty}_{\delta,x}((0, t]\times A)} \Bigl(1+\log_\limplus\Bigl(\frac{\sqrt{t}}{R}\Bigr)\Bigr)^{\kappa_d} e^{-\frac{(1-\epsilon)^2(1-\epsilon_0)R^2}{(1+\delta)t}}\\
        &\lesssim_{d, \epsilon, \delta, \epsilon_0} t^{-d/2}\|k_H\|_{G^{\infty}_{\delta,x}((0, t]\times B_R(x))} \Bigl(1+\log_\limplus\Bigl(\frac{\sqrt{t}}{R}\Bigr)\Bigr)^{\kappa_d} e^{-\frac{(1-\epsilon)^2(1-\epsilon_0)R^2}{(1+\delta)t}}\,.
    \end{align*}
    Choosing $\epsilon, \epsilon_0>0$ so small that $(1-\epsilon)^2(1-\epsilon_0)(1+\delta')=(1+\delta)$, leads to the desired estimate and completes the proof of the proposition.
\end{proof}


\section{Estimates for the Neumann boundary value problem}\label{sec: pde theorems}

Our next step towards proving our desired estimates for the heat kernel $k_\Omega$ are some a priori estimates for solutions of the heat equation. They will be useful in the next section when analyzing the behavior of $k_\Omega$ close to $\partial\Omega$. More specifically, we prove two estimates for the solutions $u$ of the Neumann problem 
\begin{equation}\label{eq: Neumann bvp}
		\begin{cases}
			\partial_t u(t, x)-\Delta u(t, x) =0\quad &\mbox{for }(t, x) \in (0, \infty)\times \Omega\,,\\
			u(0, x) = 0\quad &\mbox{for } x \in \Omega \,,\\
			\frac{\partial}{\partial \nu}u(t, x) = a(t, x)\quad &\mbox{for } (t, x)\in (0, \infty)\times \partial\Omega\,,
		\end{cases}
\end{equation}
in terms of the integrability of the Neumann data $a\colon (0, \infty)\times \partial\Omega \to \R$. 


\subsection{Preliminary remarks}
\label{sec: existence uniqueness}

Let us give a precise meaning to the boundary value/ initial value problem \eqref{eq: Neumann bvp}. 

We shall assume that $\Omega\subset\R^d$, $d\geq 2$, is open and that there is an $\Haus^{d-1}$-measurable set $\omega\subset\partial\Omega$ such that $\partial\Omega$ is Lipschitz in a neighborhood of $\omega$. We also assume that there is a measurable function $a\colon[0,\infty)\times\partial\Omega\to\R$ such that either $a\in L^2((0,T);L^p(\partial\Omega))$ or $a\in L^p(\partial\Omega;L^2(0,T))$, where $p>1$ if $d=2$ and $p\geq 2(d-1)/d$ if $d\geq 3$.

We say that $u$ is a weak solution of \eqref{eq: Neumann bvp} if, for any $T>0$, we have
$$
u\in L^2((0,T);H^1(\Omega)),
\qquad \dot u\in L^2((0,T);(H^1(\Omega))^*), \qquad u(0)=0
$$ 
and
$$
\langle \dot u(t),v\rangle + \int_\Omega \nabla u(t)\cdot\nabla v\,dx = \int_{\partial\Omega} a(t) v\,d\Haus^{d-1}
$$
for all $v\in H^1(\Omega)$ and a.e.~$t\in(0,T)$, where $\dot u$ denotes the weak derivative of $u$ with respect to the time variable. Here $(H^1(\Omega))^*$ denotes the dual of $H^1(\Omega)$ and $\langle\cdot,\cdot\rangle$ denotes the corresponding duality pairing. Moreover, as is well known, $u\in L^2((0,T);H^1(\Omega))$, $\dot u\in L^2((0,T);(H^1(\Omega))^*)$ implies $u\in C([0,T];L^2(\Omega))$, so the initial condition $u(0)=0$ is well defined. We shall prove momentarily that the boundary integral $\int_{\partial\Omega} a(t) v\,d\Haus^{d-1}$ is well defined under our assumptions on $\omega$ and $a$.

We claim that there is a unique weak solution of \eqref{eq: Neumann bvp}. Indeed, define $\langle f(t),v(t)\rangle:= \int_{\partial\Omega} a(t) v(t)\,d\Haus^{d-1}$ for $v\in L^2((0,T);H^1(\Omega))$. We claim that $f\in L^2((0,T),(H^1(\Omega))^*)$. Once this is shown, the existence and uniqueness of a weak solution follows from the arguments in \cite[Section~7.1]{Evans_PDE}.

To show that $f\in L^2((0,T),(H^1(\Omega))^*)$, we let $v\in L^2((0,T),H^1(\Omega))$. We recall that by the Sobolev trace embedding theorem, we have a bounded trace operator $H^1(\Omega) \to L^{p'}(\omega)$ for $p$ as in our assumption on $a$ and with the notation $p'=\frac{p}{p-1}$ for the dual H\"older exponent of $p$. First, assume that $a\in L^2((0,T);L^p(\partial\Omega))$. Then, by the Sobolev trace theorem,
\begin{align*}
    \left| \int_0^T \langle f(t),v(t)\rangle \,dt \right| 
    & \leq \int_0^T \|a(t)\|_{L^p(\partial\Omega)} \|v(t)\|_{L^{p'}(\partial\Omega)} \,dt \\
    & \lesssim \int_0^T \|a(t)\|_{L^p(\partial\Omega)} \|v(t)\|_{H^1(\Omega)} \,dt \\
    & \leq \| a \|_{L^2((0,T);L^p(\partial\Omega))} \|v\|_{L^2((0,T);H^1(\Omega))} \,,
\end{align*}
proving the claimed inclusion. Meanwhile, if $a\in L^p(\partial\Omega;L^2(0,T))$, we set $w(x):= ( \int_0^T |v(t,x)|^2,dt)^{1/2}$ and bound
\begin{align*}
    \left| \int_0^T \langle f(t),v(t)\rangle \,dt \right| 
    & \leq \int_{\partial\Omega} \left( \int_0^T |a(t,x)|^2\,dt\right)^{1/2} w(x)\,d\Haus^{d-1} \\
    & \leq \| a \|_{L^p(\partial\Omega;L^2(0,T))} \|w\|_{L^{p'}(\partial\Omega)} \\
    & \lesssim \| a \|_{L^p(\partial\Omega;L^2(0,T))} \|w\|_{H^1(\Omega)}  \\
    & \leq \| a \|_{L^p(\partial\Omega;L^2(0,T))} \|v\|_{L^2((0,T);H^1(\Omega))} \,,
\end{align*}
where the last inequality follows by an argument as in \cite[Theorem 7.8]{LiLo}. This proves again the claimed inclusion.

We claim that the weak solution of \eqref{eq: Neumann bvp} satisfies
$$
u\in C^\infty((0,\infty)\times\Omega) \,.
$$
This follows by standard parabolic regularity; see, for instance, \cite[Corollary 7.6]{Shubin}. (This corollary is a consequence of \cite[Theorem 5.11 and Theorem 7.5]{Shubin}.)

We now prove a Duhamel representation of the solution of \eqref{eq: Neumann bvp}.

\begin{lemma}\label{lem: duhamel}
    The solution of \eqref{eq: Neumann bvp} is given by
    $$
    u(t,x)=\int_0^{t}\int_{\partial\Omega}a(t)k_\Omega(t-s, x, y)\,d\Haus^{d-1}(y)\,ds\,.
    $$
\end{lemma}

Our argument uses the fact that $k_\Omega$ is locally bounded and jointly continuous in $x, y$, uniformly for $t$ in compact intervals of $(0,\infty)$. This can be deduced from properties of the spectral function discussed in \cite[Lemma 2.14]{FrankLarson_Inventiones25}.

\begin{proof}
    \emph{Step 1.} We shall show that for each $x\in \Omega$ and almost every $t>0$, we have
    $$
    k_\Omega(t,x,\cdot) \in H^1(\Omega) \,,
    \qquad
    \dot k_\Omega(t,x,\cdot) \in (H^{1}(\Omega))^*\,,
    $$
    with norms that are uniformly bounded for $t$ from compact subsets of $(0,\infty)$ and $x$ from compact subsets of $\Omega$, 
    and that, for $v\in H^1(\Omega)$ and almost every $t>0$, we have 
    \begin{equation}\label{eq: weak equation k}
        \langle \dot k_\Omega(t,x,\cdot),v \rangle + \int_\Omega \nabla_y k_\Omega(t,x,y)\cdot\nabla v\,dy = 0 \,.
    \end{equation}
    
    Indeed, by the spectral theorem, we have
    \begin{equation}\label{eq: H1 bound}
    \begin{aligned}
    \|\nabla k_\Omega(t,x,\cdot) \|^2_{L^2(\Omega)} &= \int_{[0, \infty)} \lambda e^{-2t\lambda} d\1(-\Delta_\Omega<\lambda)(x,x) \\
    &\leq e^{-1} t^{-1} \int_{[0, \infty)} e^{-t\lambda} d\1(-\Delta_\Omega<\lambda)(x,x) \\
    &= e^{-1} t^{-1} k_\Omega(t,x,x) \,,
    \end{aligned}
    \end{equation}
    so the first assertion follows from the boundedness of the heat kernel.

    Fix a compact set $K\subset\Omega$ and let $f\in L^2(K)$ (extended by zero to $\Omega$) and $v\in H^1(\Omega)$. Then, by the spectral theorem, $t\mapsto (f,e^{t\Delta}v)$ is differentiable with
    \begin{equation}
        \label{eq:heatkernelequation}
            \partial_t (f,e^{t\Delta} v) = ((-\Delta)^{1/2} e^{t\Delta} f,(-\Delta)^{1/2} v) \,.
    \end{equation}
    It follows that
    \begin{equation}\label{eq: time derivative bound}
    \begin{aligned}
        \left| \partial_t (f,e^{t\Delta} v) \right| & \leq \| (-\Delta)^{1/2} \, e^{(t/2)\Delta}\|_{L^2(\Omega)\to L^2(\Omega)} \| e^{(t/2)\Delta} f\|_{L^2(\Omega)} \|v\|_{H^1(\Omega)} \\
        & \leq e^{-1/2} t^{-1/2} \| e^{(t/2)\Delta} f\|_{L^2(\Omega)} \|v\|_{H^1(\Omega)} \\
        & \leq e^{-1/2} t^{-1/2} \| e^{(t/2)\Delta} \1_K \|_{L^1(\Omega)\to L^2(\Omega)} \|f\|_{L^1(K)} \|v\|_{H^1(\Omega)} \,.    
    \end{aligned}
    \end{equation}
    Note that
    $$
    \| e^{(t/2)\Delta} \1_K \|_{L^1(\Omega)\to L^2(\Omega)} = \| \1_K e^{t\Delta} \1_K \|_{L^1(K)\to L^\infty(K)}^{1/2} = \sup_{x\in K} k_\Omega(t,x,x)^{1/2}
    $$
    and that the latter is finite, uniformly for $t$ from compact subsets of $(0, \infty)$. Thus $\partial_t (f, e^{t\Delta}v)$ is uniformly bounded for $f$ in bounded subsets of $L^2(K)$ and $t$ in compact subsets of $(0, \infty)$. We now let $f$ approach a delta function supported in $K$. The continuity of $k_\Omega$ with respect to $x$ implies that $t\mapsto (f, e^{t\Delta}v)$ converges to $t\mapsto (k_\Omega(t, x, \cdot), v)$. Since $t\mapsto(f, e^{t\Delta}v)$ are uniformly Lipschitz on compact subsets of $(0, \infty)$, the limit is again Lipschitz with Lipschitz norm bounded by the right-hand side of \eqref{eq: time derivative bound}. It follows that $\dot k_\Omega(t,x, \cdot)\in (H^{1}(\Omega))^*$ for almost every $t\in (0, \infty)$ with norm bounded by the right-hand side of \eqref{eq: time derivative bound}. 

    In view of \eqref{eq:heatkernelequation} and since
    $$
    ((-\Delta)^{1/2} e^{t\Delta} f,(-\Delta)^{1/2} v)
    =\int_\Omega  \nabla(e^{t\Delta} \overline f)\cdot \nabla v\,dy \,,
    $$
    we have
    \begin{equation*}
        -\int_0^\infty \psi'(t)(f,e^{t\Delta} v)\,dt = \int_0^\infty \psi(t)\int_\Omega \nabla (e^{t\Delta}\bar f)\cdot \nabla v\,dy\,dt
    \end{equation*}
    for any $\psi \in C_c^\infty((0, \infty))$. Letting again $f$ approach a delta function and noting that by \eqref{eq: H1 bound} we can extract a subsequence along which $e^{t\Delta}\bar f \rightharpoonup \nabla_y k_\Omega(t, x, \cdot)$ in $H^1(\Omega)$, we deduce that
    \begin{equation*}
        -\int_0^\infty \psi'(t)(k_\Omega(t, x, \cdot), v)\,dt = \int_0^\infty \psi(t)\int_\Omega \nabla_y k_\Omega(t, x, y)\cdot \nabla v\,dy\,dt\,,
    \end{equation*}
    this proves \eqref{eq: weak equation k}.
    
    \medskip

    \noindent\emph{Step 2.} Assume that $u$ is a weak solution of \eqref{eq: Neumann bvp} in the sense above. By Step 1, $k_\Omega(s', x, \cdot) \in H^1(\Omega)$ for almost every $s'>0$ and all $x\in \Omega$ so we can take $v= k_\Omega(s', x, \cdot)$ in the weak formulation of the equation. This yields that
\begin{align*}
    \langle \dot u(s), k_\Omega(s', x, \cdot)\rangle+\int_\Omega \nabla k_\Omega(s', x, y)\cdot \nabla u(s) \,dy = \int_{\partial\Omega}a(s)k_\Omega(s', x, y)\,d\Haus^{d-1}(y)
\end{align*}
for almost every $s>0, s'>0$ and every $x\in \Omega$. Correspondingly, setting $v= u(s)$ in \eqref{eq: weak equation k} yields that
\begin{align*}
    \int_\Omega \nabla k_\Omega(s', x, y)\cdot \nabla u(s) \,dy = -\langle \dot k_\Omega(s', x, \cdot), u(s)\rangle
\end{align*}
for almost every $s>0, s'>0$ and all $x\in \Omega$. 
Thus
\begin{equation*}
    \int_{\partial\Omega}a(s)k_\Omega(s', x, y)\,d\Haus^{d-1}(y) = \langle \dot u(s), k_\Omega(s', x, \cdot)\rangle-\langle \dot k_\Omega(s', x, \cdot), u(s)\rangle\,,
\end{equation*}
for almost every $s>0, s'>0$ and all $x\in \Omega$. In particular, it follows from Fubini's theorem that for almost every $t>0$,
\begin{equation*}
    \int_{\partial\Omega}a(s)k_\Omega(t-s, x, y)\,d\Haus^{d-1}(y) = \langle \dot u(s), k_\Omega(t-s, x, \cdot)\rangle-\langle  \dot k_\Omega(t-s, x, \cdot), u(s)\rangle\,,
\end{equation*}
for almost every $s\in (0, t)$.

By integrating the above equality with respect to $s$ over the interval $(\epsilon, t-\epsilon)$, it follows that 
\begin{align*}
    \int_\epsilon^{t-\epsilon}&\int_{\partial\Omega}a(s)k_\Omega(t-s, x, y)\,d\Haus^{d-1}(y)\,ds\\
    &=
    \int_\epsilon^{t-\epsilon}\bigl(\langle \dot u(s), k_\Omega(t-s, x, \cdot)\rangle-\langle \dot k_\Omega(t-s, x, \cdot), u(s)\rangle\bigr)\,ds\\
    &=
    \int_\epsilon^{t-\epsilon}\partial_s( u(s), k_\Omega(t-s, x, \cdot))\,ds\\
    &=
    ( u(t-\epsilon), k_\Omega(\epsilon, x, \cdot))-(u(\epsilon), k_\Omega(t-\epsilon, x, \cdot))\\
    &= (e^{\epsilon\Delta}u(t-\epsilon))(x)-(e^{(t-\epsilon)\Delta}u(\epsilon))(x)\,,
\end{align*}
where the use of the product rule is justified by \cite[Theorem 5.9.3 (ii)]{Evans_PDE} applied to $\mathbf{u}(s)=u(s)\pm k_\Omega(t-s, x, \cdot)$. Sending $\epsilon \to 0^\limplus$ and using the small time asymptotic behavior of $u$ and and $e^{t\Delta}$ shows the claimed formula.
\end{proof}


\subsection{A first a priori bound}

Our aim in this subsection is to prove an a priori estimate for solutions of \eqref{eq: Neumann bvp} when the boundary data $a$ belongs to suitable Lebesgue spaces. The bound proved here refines and generalizes that of \cite[Lemma 3.4]{Brown93}. In particular, we explicitly track how the geometry of $\Omega$ enters in the estimate. Making the dependence on the geometry explicit is crucial in proving Theorem~\ref{thm: main thm convex}.

In order to state our estimate we for an $\Haus^{d-1}$-measurable set $\omega\subset\R^d$ define
\begin{equation*}
        \Theta[\omega] := \sup_{r>0, z\in \R^d} \frac{\Haus^{d-1}(\omega \cap B_r(z))}{r^{d-1}}\,.
\end{equation*}

\begin{theorem}\label{thm: neumann bvp Lp estimates}
    Let $\Omega \subset \R^d, d\geq2,$ be an open set. Assume that $\omega \subseteq \partial\Omega$ is $\Haus^{d-1}$-measurable and that $\partial\Omega$ is Lipschitz regular in a neighborhood of $\omega$.
    
    Fix $p\in [1, \infty], q \in (2, \infty]$ satisfying $\frac{q'd}2>1$ and $p> \frac{(d-1)q}{q-2}$. Suppose $a\colon [0, \infty) \times \partial\Omega\to \R$ is supported in $[0, \infty) \times \omega$ and let $u \colon (0, \infty)\times \Omega \to \R$ be the corresponding solution of the Neumann boundary value problem \eqref{eq: Neumann bvp}.

    For any $t>0$,
	\begin{align*}
		|u(t, x)| 
		&\lesssim_{d,p,q,\delta} \|k_\Omega\|_{G^\infty_{\delta,x}((0, t]\times \omega)}e^{-\frac{\dist(x, \omega)^2}{8(1+\delta)t}} \min\bigl\{\|a\|_{L^p(\omega; L^q((0, t)))}, \|a\|_{L^q((0, t); L^p(\omega))}\bigr\}\\
        &\quad \times
         \min\biggl\{\Theta[\omega]^{\frac{2}{d-1}(\frac{d}2-\frac{1}{q'})}\Haus^{d-1}(\omega)^{\frac{1}{p'}-\frac{2}{d-1}(\frac{d}2-\frac{1}{q'})}, \Theta[\omega]^{\frac{1}{p'}}t^{\frac{1}{q'}+\frac{d-1}{2p'}-\frac{d}{2}}\biggr\}\,.
	\end{align*}
\end{theorem}

We note that, by Minkowski's inequality, $\|a\|_{L^p(\omega; L^q((0, t)))} \leq \|a\|_{L^q((0, t); L^p(\omega))}$ if $p\geq q$ and $\|a\|_{L^p(\omega; L^q((0, t)))} \geq \|a\|_{L^q((0, t); L^p(\omega))}$ if instead $p\leq q$. Therefore, which one of the terms in the first minimum that is the smallest only depends on the relation between $p, q$ and not on the function $a$.

\begin{proof}
    \emph{Step 1.}
    We may assume that $\min\bigl\{\|a\|_{L^p(\omega; L^q((0, t)))}, \|a\|_{L^q((0, t); L^p(\omega))}\bigr\}<\infty$. This assumption implies the assumption in the previous subsection and, in particular, implies that the solution $u$ exists and the Duhamel representation in Lemma \ref{lem: duhamel} is valid. It follows from this formula that
	\begin{align*}
		|u(t, x)| 
		&\leq \|k_\Omega\|_{G^\infty_{\delta,x}((0, t]\times \omega)}e^{-\frac{\dist(x, \omega)^2}{8(1+\delta)t}}\int_0^t \int_{\omega} |a(s, y)|\frac{e^{-\frac{|x-y|^2}{8(1+\delta)(t-s)}}}{(t-s)^{d/2}}\,d\Haus^{d-1}(y)ds\,.
	\end{align*}
    By H\"older's inequality, we have
    \begin{equation}\label{eq: pq Holder}
    \begin{aligned}
		|u(t, x)| 
		&\leq \|k_\Omega\|_{G^\infty_{\delta,x}((0, t]\times \omega)}e^{-\frac{\dist(x, \omega)^2}{8(1+\delta)t}}\|a\|_{L^p(\omega; L^q((0, t)))}\\
        &\qquad \times\biggl(\int_\omega\biggl(\int_0^t\frac{e^{-q'\frac{|x-y|^2}{8(1+\delta)(t-s)}}}{(t-s)^{q'd/2}}\,ds\biggr)^{p'/q'}d\Haus^{d-1}(y)\biggr)^{1/p'}
	\end{aligned}
    \end{equation}
    and
    \begin{equation}\label{eq: qp Holder}
    \begin{aligned}
		|u(t, x)| 
		&\leq \|k_\Omega\|_{G^\infty_{\delta,x}((0, t]\times \omega)}e^{-\frac{\dist(x, \omega)^2}{8(1+\delta)t}}\|a\|_{L^q((0, t); L^p(\omega))}\\
        &\qquad \times\biggl(\int_0^t(t-s)^{-q'd/2}\biggl(\int_\omega e^{-p'\frac{|x-y|^2}{8(1+\delta)(t-s)}}\,d\Haus^{d-1}(y)\biggr)^{q'/p'}ds\biggr)^{1/q'}\,.
	\end{aligned}
    \end{equation}
    In the remainder of this proof, we will show that
    \begin{equation}\label{eq: pq bound}
    \begin{aligned}
		& \biggl(\int_\omega\biggl(\int_0^t\frac{e^{-q'\frac{|x-y|^2}{8(1+\delta)(t-s)}}}{(t-s)^{q'd/2}}\,ds\biggr)^{p'/q'}d\Haus^{d-1}(y)\biggr)^{1/p'} \\
        & \lesssim_{d,p,q,\delta} \min\biggl\{\Theta[\omega]^{\frac{2}{d-1}(\frac{d}2-\frac{1}{q'})}\Haus^{d-1}(\omega)^{\frac{1}{p'}-\frac{2}{d-1}(\frac{d}2-\frac{1}{q'})}, \Theta[\omega]^{\frac{1}{p'}}t^{\frac{1}{q'}+\frac{d-1}{2p'}-\frac{d}{2}}\biggr\}
	\end{aligned}
    \end{equation}    
    and
    \begin{equation}\label{eq: qp bound}
    \begin{aligned}
		& \biggl(\int_0^t(t-s)^{-q'd/2}\biggl(\int_\omega e^{-p'\frac{|x-y|^2}{8(1+\delta)(t-s)}}\,d\Haus^{d-1}(y)\biggr)^{q'/p'}ds\biggr)^{1/q'} \\
        & \lesssim_{d,p,q,\delta} \min\biggl\{\Theta[\omega]^{\frac{2}{d-1}(\frac{d}2-\frac{1}{q'})}\Haus^{d-1}(\omega)^{\frac{1}{p'}-\frac{2}{d-1}(\frac{d}2-\frac{1}{q'})}, \Theta[\omega]^{\frac{1}{p'}}t^{\frac{1}{q'}+\frac{d-1}{2p'}-\frac{d}{2}}\biggr\} \,.
	\end{aligned}
    \end{equation}
    Once we have shown \eqref{eq: pq bound} and \eqref{eq: qp bound}, the theorem will follow from \eqref{eq: pq Holder} and \eqref{eq: qp Holder}.
    Clearly, in the proof of \eqref{eq: pq bound} and \eqref{eq: qp bound} we may assume $\Theta[\omega]<\infty$, for otherwise there is nothing to prove. Here and in what follows we allow for the possibility that $\Haus^{d-1}(\omega)$ is infinite.

    \medskip

    \noindent\emph{Step 2.} Our proof of \eqref{eq: pq bound} and \eqref{eq: qp bound} relies on a more general bound, which we derive in this step. Define, for $(\beta, \gamma)\in [0, \infty)^2\setminus \{(0,0)\}$ and $\mu > 0$,
    \begin{equation*}
        h_{\beta, \gamma}(\mu) := \mu^{-\beta} e^{-\gamma \mu}\,.
    \end{equation*}
    We will show that, if $\beta<\frac{d-1}{2}$, then
    \begin{equation}\label{eq: general h integral bound}
    \begin{aligned}
        \int_\omega h_{\beta, \gamma}\Bigl(\frac{|x-y|^2}{\tau}\Bigr)\,d\Haus^{d-1}(y)
        \lesssim_{d,\beta,\gamma} 
        \Theta[\omega] \min\biggl\{\Bigl(\frac{\Haus^{d-1}(\omega)}{\Theta[\omega]}\Bigr)^{1-2\beta/(d-1)}\tau^{\beta}, \tau^{(d-1)/2}\biggr\}\,.
    \end{aligned}
    \end{equation}

    For the proof of \eqref{eq: general h integral bound}, we note that $h_{\beta, \gamma}$ is smooth, decreasing, and tends to zero as $\mu \to \infty$. This allows us to write
    \begin{equation*}
        h_{\beta, \gamma}(s) = \int_s^\infty |h'_{\gamma,\beta}(\mu)|\,d\mu= \int_0^\infty \1_{\{\mu>s\}}|h'_{\gamma,\beta}(\mu)|\,d\mu\,.
    \end{equation*}
    Therefore, by Fubini's theorem,
    \begin{align*}
        \int_\omega h_{\beta, \gamma}\Bigl(\frac{|x-y|^2}{\tau}\Bigr)\,d\Haus^{d-1}(y)
        &=
        \int_\omega \int_0^\infty \1_{\{\mu>|x-y|^2/\tau\}}|h'_{\gamma,\beta}(\mu)|\,d\mu\,d\Haus^{d-1}(y)\\
        &=
        \int_0^\infty \biggl(\int_\omega  \1_{\{\sqrt{\tau \mu}>|x-y|\}}\,d\Haus^{d-1}(y)\biggr)|h'_{\gamma,\beta}(\mu)|\,d\mu\\
        &=
        \int_0^\infty \Haus^{d-1}(\omega \cap B_{\sqrt{\tau \mu}(x)}) |h'_{\gamma,\beta}(\mu)|\,d\mu \,.
    \end{align*}
    Consequently, we have the bound
    \begin{align*}
        \int_\omega h_{\beta, \gamma}\Bigl(\frac{|x-y|^2}{\tau}\Bigr)\,d\Haus^{d-1}(y)
        &\leq
        \int_0^{\infty} \min\{\Haus^{d-1}(\omega), \Theta[\omega](\mu \tau)^{(d-1)/2}\}|h'_{\beta,\gamma}(\mu)|\,d\mu\,.
    \end{align*}
    Since $|h'_{\beta,\gamma}(\mu)|\sim \mu^{-1-\beta}$ as $\mu \to 0$, the integral on the right side is finite only if $\beta <\frac{d-1}{2}$. If this is the case, an integration by parts yields
    \begin{equation*}
    \begin{aligned}
        \int_\omega h_{\beta, \gamma}\Bigl(\frac{|x-y|^2}{\tau}\Bigr)\,d\Haus^{d-1}(y)
        &\lesssim_d
        \Theta[\omega] \tau^{(d-1)/2}\int_0^{\tau^{-1}(\Haus^{d-1}(\omega)/\Theta[\omega])^{2/(d-1)}}\!\!\mu^{(d-3)/2}h_{\beta,\gamma}(\mu)\,d\mu\\
        &=
        \Theta[\omega] \tau^{(d-1)/2}\int_0^{\tau^{-1}(\Haus^{d-1}(\omega)/\Theta[\omega])^{2/(d-1)}}\!\!\mu^{(d-3)/2-\beta}e^{-\gamma\mu}\,d\mu\\
        &\lesssim_{d,\beta,\gamma} 
        \Theta[\omega] \min\biggl\{\Bigl(\frac{\Haus^{d-1}(\omega)}{\Theta[\omega]}\Bigr)^{1-2\beta/(d-1)}\tau^{\beta}, \tau^{(d-1)/2}\biggr\}\,.
    \end{aligned}
    \end{equation*}
    This proves the desired bound \eqref{eq: general h integral bound}.
    
    \medskip

    \noindent\emph{Step 3.}
    We turn to the proof of \eqref{eq: pq bound}. By a change of variables and using $q'd>2$, we find
    \begin{align*}
        \int_0^t\frac{e^{-q'\frac{|x-y|^2}{8(1+\delta)(t-s)}}}{(t-s)^{q'd/2}}\,ds
        &=
        \Bigl(\frac{q'|x-y|^2}{8(1+\delta)}\Bigr)^{1-q'd/2} \int_{q'\frac{|x-y|^2}{8(1+\delta)t}}^\infty e^{-\tau}\tau^{q'd/2-2}\,d\tau\\
        &\leq
        \Bigl(\frac{q'|x-y|^2}{8(1+\delta)}\Bigr)^{1-q'd/2} e^{-q'\frac{|x-y|^2}{16(1+\delta)t}}\int_{0}^\infty e^{-\tau/2}\tau^{q'd/2-2}\,d\tau\\
        &=
        \Bigl(\frac{q'|x-y|^2}{8(1+\delta)}\Bigr)^{1-q'd/2} e^{-q'\frac{|x-y|^2}{16(1+\delta)t}}2^{q'd/2-1}\Gamma\Bigl(\frac{q'd}2-1\Bigr)\,.
    \end{align*}
    Therefore,
    \begin{equation*}
    \begin{aligned}
        \int_\omega\biggl(\int_0^t&\frac{e^{-q'\frac{|x-y|^2}{8(1+\delta)(t-s)}}}{(t-s)^{q'd/2}}\,ds\biggr)^{p'/q'}d\Haus^{d-1}(y)\\
        &\lesssim_{p,q,d,\delta} 
        t^{p'/q'-p'd/2}\int_\omega \Bigl(\frac{|x-y|^2}{t}\Bigr)^{p'/q'-p'd/2} e^{-p'\frac{|x-y|^2}{16(1+\delta)t}}d\Haus^{d-1}(y)\,.
    \end{aligned}
    \end{equation*}
    The claimed bound \eqref{eq: pq bound} now follows from \eqref{eq: general h integral bound} with the choice $\beta = \frac{p'}{q'}\bigl(\frac{q'd}2-1\bigr)$ and $\gamma = \frac{p'}{16(1+\delta)}$. Note that our assumptions on $p, q$ imply that $0<\beta<\frac{d-1}{2}$. 

    Next, we turn to the proof of \eqref{eq: qp bound}. It follows from \eqref{eq: general h integral bound} with $\beta =0$ and $\gamma = \frac{p'}{8(1+\delta)}$ that
    \begin{equation*}
    \begin{aligned}
		& \biggl(\int_0^t(t-s)^{-q'd/2}\biggl(\int_\omega e^{-p'\frac{|x-y|^2}{8(1+\delta)(t-s)}}\,d\Haus^{d-1}(y)\biggr)^{q'/p'}ds\biggr)^{1/q'} \\
        & \lesssim_{d,p,q,\delta} \biggl(\int_0^t(t-s)^{-q'd/2}  \min\Bigl\{\Haus^{d-1}(\omega), \Theta[\omega](t-s)^{(d-1)/2}\Bigr\}^{q'/p'}ds\biggr)^{1/q'} \,.
	\end{aligned}
    \end{equation*}
    Note that by our assumptions on $p, q$ it holds that $\frac{q'd}2>1$ and $\frac{q'd}{2}<1+\frac{q'(d-1)}{2p'}$.
    The claimed bound \eqref{eq: qp bound} now follows from the fact that, if $1< \kappa<1+\alpha, \alpha>0, t>0, A>0$, then
    \begin{align*}
        \int_0^t (t-&s)^{-\kappa}\min\{A, (t-s)^\alpha\}\,ds \\
        &=
        \begin{cases}
            \int_0^t (t-s)^{-\kappa+\alpha}\,ds & \mbox{if }A^{1/\alpha}\geq t\\
            A\int_0^{t-A^{1/\alpha}} (t-s)^{-\kappa}\,ds+\int_{t-A^{1/\alpha}}^t (t-s)^{-\kappa+\alpha}\,ds & \mbox{if }A^{1/\alpha}< t
        \end{cases}\\
        &=
        \begin{cases}
            \frac{t^{1-\kappa+\alpha}}{1-\kappa+\alpha} & \mbox{if }A^{1/\alpha}\geq t\\
            \frac{At^{1-\kappa}-A^{(1-\kappa)/\alpha+1}}{1-\kappa}+\frac{A^{(1-\kappa)/\alpha+1}}{1-\kappa+\alpha} & \mbox{if }A^{1/\alpha}< t
        \end{cases}\\
        &\leq
        \begin{cases}
            \frac{t^{1-\kappa+\alpha}}{1-\kappa+\alpha} & \mbox{if }A^{1/\alpha}\geq t\\
            \bigl(\frac{1}{1-\kappa}+\frac{1}{1-\kappa+\alpha}\bigr)A^{(1-\kappa)/\alpha+1} & \mbox{if }A^{1/\alpha}< t
        \end{cases}\\
        &\lesssim_{\alpha, \kappa} \min\{A^{(1-\kappa)/\alpha+1}, t^{1-\kappa+\alpha}\}\,.
    \end{align*}
    This completes the proof of the theorem.
\end{proof}


\subsection{A second a priori bound}

In addition to Theorem \ref{thm: neumann bvp Lp estimates}, we shall require an estimate under the assumption that the boundary data $a$ has support in a parabolic ball $[0, \rho^2]\times (\overline{B_\rho(x_0)}\cap \partial\Omega)$ and zero average. The estimate that we prove is a version of \cite[Lemma 3.3]{Brown93}. Here, as in our previous bounds, it is necessary that the track more carefully the dependence of the bound on the geometry of $\Omega$.

To state the result, we need to make precise in which sense and, importantly, at what scale the boundary is assumed to be locally a graph.

\begin{definition}
    Let $\Omega \subset \R^d, d\geq 2$, be an open set. Fix $x_0\in \partial\Omega$. We say that $\partial\Omega$ is locally parametrized by a function $\varphi \colon \R^{d-1}\to \R$ in $B_R(x_0)$ if there exists $Q\in \mathbb{SO}(d)$ such that
    \begin{equation*}
        (\Omega-x_0)\cap B_R(0) = Q(\{y=(y', y_d)\in \R^d: y_d>\varphi(y')\})\cap B_R(0)\,.
    \end{equation*}
\end{definition}
The notion that $\partial\Omega$ is Lipschitz regular means that for each $x_0 \in \partial\Omega$ there exists a $R>0$ small enough so that $\partial\Omega$ can be locally parametrized in $B_R(x_0)$ by a Lipschitz function $\varphi$. For the results that follow, not only the existence of such $R$ and $\varphi$ is of importance but also the relative size of $R$ and $\|\nabla\varphi\|_{L^\infty}$ compared to the other quantities that are involved.

The result that we aim to prove in this subsection is the following.
\begin{lemma}\label{lem: a priori bound zero average}
	Let $\Omega \subset \R^d, d\geq 2,$ be an open set and fix $x_0\in \partial\Omega, \delta>0$. Assume that there exist $M, R>0$, and $\varphi \colon \R^{d-1}\to \R$ with $\|\nabla \varphi\|_{L^\infty}\leq M$ such that $\partial\Omega$ is locally parametrized by $\varphi$ in $B_R(x_0)$. Assume that $\rho_0 \in (0, R)$ satisfies that
    \begin{equation*}
        \sup_{y \in B_{2\rho_0}(0)\setminus\{y_d>\varphi(y')\}}\bigl(|y|^2+4\varphi(y')(\varphi(y')-y_d)\bigr)<R^2\,.
    \end{equation*}
    Assume furthermore that $a \in L^2((0, \infty); L^p(\partial\Omega))+L^p(\partial\Omega; L^2(0,\infty))$ for some $p>1$ if $d=2$ and $p\geq 2(d-1)/d$ if $d\geq 3$ and satisfies $\supp(a) \subseteq [0, \rho^2] \times (\partial \Omega\cap B_{\rho}(x_0))$ for some $\rho\geq 0$ and $\iint a\,ds\,d\Haus^{d-1}(y) =0$. 
    
    If $u\colon (0, \infty)\times \Omega\to \R$ denotes the solution of the Neumann boundary value problem~\eqref{eq: Neumann bvp}, then
	\begin{equation*}
		|u(t, x)| \lesssim_{d, \delta, M} \|a\|_{L^1((0, \rho^2)\times \partial\Omega)} \|k_\Omega\|_{G^\infty_{\delta, x}((t/2, 3t/2)\times B_{\sqrt{t}}(x_0))} t^{-d/2}e^{-\frac{d_\Omega(x)^2}{12(1+\delta)t}} \min \Bigl\{1, \Bigl(\frac{\rho}{\rho_0}\Bigr)^\alpha\Bigr\}
	\end{equation*}
	for all $(t, x) \in [4\max\{\rho_0,\rho\}^2, \infty)\times \Omega$ and where $\alpha>0$ only depends on $d, M$.
\end{lemma}

\begin{remark}
    We note that the assumptions of the lemma imply that $|\varphi(y')|\leq M |y'|$ (note that $\varphi(0)=0$ as $x_0 \in \partial\Omega$). Therefore, the Cauchy--Schwarz inequality yields
    \begin{align*}
        \sup_{y'\in B_{2{\rho_0}}(0)\setminus \{y_d>\varphi(y')\}} &\left( |y|^2+4\varphi(y')(\varphi(y')-y_d) \right) \\
        &\leq \sup_{y'\in B_{2{\rho_0}}(0)\setminus \{y_d>\varphi(y')\}} \left( |y|^2+4M^2|y'|^2+4M|y'||y_d| \right)\\
        &\leq \sup_{y'\in B_{2{\rho_0}}(0)\setminus \{y_d>\varphi(y')\}} \left( |y|^2+4M^2|y'|^2+2M|y|^2 \right)\\
        &\leq 4(1+4M^2+2M)\rho_0^2 \,.
    \end{align*}
    It follows that the smallness assumption on $\rho_0$ is valid if $\rho_0 < (4(1+4M^2+2M))^{-1/2}R$. In our applications, we shall use this only close to points where the boundary lies within a small neighborhood of a hyperplane.
\end{remark}

In order to prove Lemma \ref{lem: a priori bound zero average}, we shall utilize the following result due to Nash~\cite{Nash58} and Moser~\cite{Moser64}, which implies that $k_\Omega$ is H\"older continuous in both space and time up to the spatial boundary.

\begin{lemma}\label{lem: Nash Moser Holder regularity}
	Let $A \in L^\infty((0, R^2)\times B_R(0); \R^{d\times d})$ satisfy $A= A^\top$ and assume that there exists a constant $\Lambda\geq 1$ such that
	\begin{equation*}
			\Lambda^{-1}|\xi|^2\leq \xi \cdot A(t, x)\xi \leq \Lambda |\xi|^2 \quad \mbox{for all }\xi \in \R^d, (t, x) \in (0, R^2)\times B_R(0)\,.
		\end{equation*}	
	If $u \colon (0, R^2) \times B_R(0)\to \R$ is a bounded function satisfying
	\begin{equation*}
		\partial_t u(t, x) = \Div(A(t, x)\nabla u(t, x)) \quad \mbox{for all } (t, x)\in (0, R^2)\times B_R(0)\,,
	\end{equation*}
	then there exists $\alpha >0$ so that
	\begin{equation*}
		|u(t, x)-u(s, y)| \lesssim_{d, \Lambda} \|u\|_{L^\infty((0, R^2)\times B_R(0))}\biggl(\frac{|x-y|+|t-s|^{1/2}}{R}\biggr)^\alpha
	\end{equation*}
	whenever $(t, x), (s, y) \in (R^2/4, 3R^2/4)\times B_{R/2}(0)$. Furthermore, $\alpha$ can be chosen depending only on $d, \Lambda$. 
\end{lemma}

\begin{corollary}\label{cor: Holder regularity up to boundary}
	Assume that $\Omega, x_0, \delta, M, R, \varphi$, and $\rho_0$ are as in Lemma \ref{lem: a priori bound zero average}. Then, for all $(t, x) \in [4\rho_0^2, \infty)\times \Omega$, $y_1, y_2 \in \overline{\Omega} \cap B_{\rho_0}(x_0)$ and $s_1, s_2 \in (t-\rho_0^2, t+\rho_0^2)$, we have
	\begin{align*}
		|k_\Omega(s_1&, x, y_1)-k_\Omega(s_2, x, y_2)|\\
        &\lesssim_{d, \delta, M} \|k_\Omega\|_{G^\infty_{\delta, x}((t/2, 3t/2)\times B_{2\rho_0}(x_0))} \, t^{-d/2} \, e^{-\frac{|x-x_0|^2}{12(1+\delta)t}} \, \biggl(\frac{|y_1-y_2|+|s_1-s_2|^{1/2}}{\rho_0}\biggr)^\alpha \,.
	\end{align*}
	Here, $\alpha>0$ can be chosen depending only on $d, M$.
\end{corollary}

\begin{proof}
	After a rotation and a translation we may assume without loss of generality that $x_0=0, Q= \I$. Fix $(t, x) \in (0, \infty) \times \Omega$. The function $v(s, y) := k_\Omega(s, x, y)$ solves
	\begin{equation*}
			(\partial_s-\Delta_y)v(s, y) =0 \quad \mbox{in } (0, \infty) \times \Omega\,.
	\end{equation*}
	For $y \in \R^d$ we define the transformation
	\begin{equation*}
		\tilde y(y) := \begin{cases}
			y & \mbox{if }y_d> \varphi(y') \,, \\
			(y', 2\varphi(y')-y_d) & \mbox{if } y_d\leq \varphi(y')\,.
		\end{cases}
	\end{equation*}
	Note that since $$\sup_{y \in B_{2\rho_0}(0)\setminus \Omega}(|y|^2+4\varphi(y')(\varphi(y')-y_d))<R^2\,,$$
	we have $\tilde y(y) \in \Omega \cap B_R(0)$ for all $y \in B_{2\rho_0}(0) \setminus \partial\Omega$. Using the fact that $v$ solves the heat equation in $\Omega \cap B_R(0)$, we see that the function $\tilde v \colon (0, \infty) \times B_{2\rho_0}(0)\to \R$ defined by $\tilde v(s, y) := v(s, \tilde y(y))$ solves
	\begin{equation*}
		\partial_s \tilde v(s, y) = \Div_y(A(y)\nabla_y \tilde v(s, y)) \quad \mbox{in }(0, \infty) \times \left( B_{2\rho_0}(0) \setminus\partial\Omega\right),
	\end{equation*}
	with
	\begin{equation*}
		A(y) := \I \mbox{ for } y \in B_{2\rho_0}(0)\cap \Omega
	\end{equation*}
	and
	\begin{equation*}
	A(y) := \left(\begin{matrix}
			\I' & 2\nabla \varphi(y')\\ 2(\nabla \varphi(y'))^\top & 1+4|\nabla \varphi(y')|^2 
		\end{matrix}\right)\mbox{ for }y \in B_{2\rho_0}(0)\setminus \overline{\Omega}\,,
	\end{equation*}
	where $\I'$ denotes the $(d-1)\times (d-1)$ identity matrix. 
	In fact, it can be verified that $\tilde v$ is a weak solution of $\partial_s \tilde v = \Div(A\nabla_y \tilde v)$ in $(0, \infty)\times B_{2\rho_0}(0)$.

	Clearly, $A$ is uniformly elliptic in $B_{2\rho_0}(0)\cap \Omega$. That this holds in all of $B_{2\rho_0}(0)$ follows from the fact that the eigenvalues of $A(y)$ for $y \in B_{2\rho_0}(0)\setminus\overline\Omega$ are $1$ and
	\begin{align*}
		1+ 2|\nabla \varphi(y')|^2 \pm|\nabla \varphi(y')|\sqrt{4+4|\nabla \varphi(y')|^2}\,,
	\end{align*}
	where $1$ has multiplicity (at least) $d-2$ and the other two are simple (unless they are both $1$). In particular, $A$ is uniformly elliptic with the bound 
	\begin{equation*}
		\frac{1}{1+\|\nabla \varphi\|^2_{L^\infty(B_{2\rho}(0))}}|\xi|^2 \lesssim \xi\cdot A(y)\xi \lesssim (1+\|\nabla \varphi\|_{L^\infty(B_{\rho_0}(0))}^2)|\xi|^2\,,
	\end{equation*}
	for all $y \in B_{2\rho_0}(0)$.

	From Lemma~\ref{lem: Nash Moser Holder regularity} we deduce that, for any $t_0>0$ and all $(s_1, y_1), (s_2, y_2)\in (t_0+\rho_0^2, t_0+3\rho_0^2)\times (\overline{\Omega} \cap B_{\rho_0}(0))$,
	\begin{equation*}
			|v(s_1, y_1)-v(s_2, y_2)| \leq C \|v\|_{L^\infty((t_0, t_0+4\rho_0^2)\times (\Omega \cap B_{2\rho_0}(0)))} \biggl(\frac{|y_1-y_2|+|s_1-s_2|^{1/2}}{\rho_0}\biggr)^\alpha\,,
	\end{equation*}
	where $C$ and $\alpha$ only depend on $d$ and $\|\nabla \varphi\|_{L^\infty(B_{\rho_0}(0))}$. We now choose $t_0 := t-2\rho_0^2$ and note that
	\begin{align*}
		\|v&\|_{L^\infty((t-2\rho_0^2, t+2\rho_0^2)\times (\Omega \cap B_{2\rho_0}(0)=)} \\
        &\qquad= \sup_{y \in B_{2\rho_0}(0)\cap \Omega, s\in (t-2\rho_0^2, t+2\rho_0^2)} k_\Omega(s, x, y)\\
		&\qquad\leq \|k_\Omega\|_{G^\infty_{\delta, x}((t-2\rho_0^2, t+2\rho_0^2)\times B_{2\rho_0}(0))}\sup_{y \in B_{2\rho_0}(0)\cap \Omega, s\in (t-2\rho_0^2, t+2\rho_0^2)} s^{-d/2} e^{-\frac{|x-y|^2}{4(1+\delta)s}}\,.
	\end{align*}
    Since $t\geq 4\rho_0^2$ we have that $(t-2\rho_0^2, t+2\rho_0^2)\subseteq (t/2, 3t/2)\cap [2\rho_0^2, \infty)$, and thus for any $y \in B_{2\rho_0}(0)$ we have
    \begin{equation*}
        \frac{|x-y|^2}{4(1+\delta)s}  \geq \frac{|x|^2}{8(1+\delta)s}- \frac{|y|^2}{4(1+\delta)s} \geq \frac{|x|^2}{12t}- \frac{1}{2(1+\delta)}\,.
    \end{equation*}
    Thus we conclude that
    \begin{align*}
		\|v\|_{L^\infty((t-2\rho_0^2, t+2\rho_0^2)\times (\Omega\cap B_{2\rho_0}(0)))} 
        &\lesssim_{d} \|k_\Omega\|_{G^\infty_{\delta, x}((t/2, 3t/2)\times B_{2\rho_0}(0))}t^{-d/2} e^{-\frac{|x|^2}{12(1+\delta)t}}\,.
	\end{align*}
    This completes the proof.
\end{proof}

With these preliminaries in hand we are ready to prove Lemma \ref{lem: a priori bound zero average}.

\begin{proof}[Proof of Lemma \ref{lem: a priori bound zero average}]
    By the discussing in Section~\ref{sec: existence uniqueness}, the integrability assumptions on $a$ implies that the solution $u$ exists and that we can apply Lemma \ref{lem: duhamel}.
	Using the representation formula in Lemma \ref{lem: duhamel} along with the assumption that $\iint a\,ds\,d\Haus^{d-1}(y) =0$, we find
	\begin{align*}
		\left| u(t, x) \right|
		&= 
			\left| \int_0^{\rho^2} \int_{\partial\Omega\cap B_\rho(x_0)} k_\Omega(t-s, x, y)a(s, y)\,d\Haus^{d-1}(y)ds \right| \\
		&=
			\left| \int_0^{\rho^2} \int_{\partial\Omega\cap B_\rho(x_0)} \Bigl(k_\Omega(t-s, x, y)-k_\Omega(t, x, x_0)\Bigr)a(s, y)\,d\Haus^{d-1}(y)ds \right| \\
            &\leq \|a\|_{L^1((0, \rho^2)\times \partial\Omega)}\sup_{s \in (0, \rho^2), y\in \partial\Omega \cap B_\rho(x_0)}|k_\Omega(t- s, x, y)-k_\Omega(t, x, x_0)|\,.
	\end{align*}
    Note that the assumptions on $a$ together with H\"older's inequality imply that the $L^1$-norm is finite.

    First assume that $\rho \leq \rho_0$ and $t \geq 4\rho_0^2$. Then Corollary~\ref{cor: Holder regularity up to boundary} implies
	\begin{align*}
		\sup_{s \in (0, \rho^2), y\in \partial\Omega \cap B_\rho(x_0)}&|k_\Omega(t- s, x, y)-k_\Omega(t, x, x_0)|\\
		&\lesssim_{d, \delta, M} \|k_\Omega\|_{G^\infty_{\delta, x}((t/2, 3t/2)\times B_{2\rho_0}(x_0))} t^{-d/2}e^{-\frac{|x-x_0|^2}{12(1+\delta)t}} \biggl(\frac{\rho}{\rho_0}\biggr)^\alpha\\
        &\lesssim_{d, \delta, M} \|k_\Omega\|_{G^\infty_{\delta, x}((t/2, 3t/2)\times B_{\sqrt{t}}(x_0))} t^{-d/2}e^{-\frac{d_\Omega(x)^2}{12(1+\delta)t}} \biggl(\frac{\rho}{\rho_0}\biggr)^\alpha\,,
	\end{align*}
    which gives the claimed bound.

    Now assume that $\rho \in [\rho_0, \sqrt{t}/2)$. Then for all $s \in (0, \rho^2)$ we have $t-s \in [t/2, t]$, and thus
    \begin{align*}
		\sup_{s \in (0, \rho^2), y\in \partial\Omega \cap B_\rho(x_0)}&|k_\Omega(t- s, x, y)-k_\Omega(t, x, x_0)|\\
		&\lesssim_d 
        \|k_\Omega\|_{G^\infty_{\delta,x}((t/2,t)\times B_{\rho}(x_0))}t^{-d/2}e^{-\frac{|x-y|^2}{4(1+\delta)t}}\\
        &\quad 
        +\|k_\Omega\|_{G^\infty_{\delta,x}((t/2,t)\times B_{\rho}(x_0))}t^{-d/2}e^{-\frac{|x-x_0|^2}{4(1+\delta)t}}\\
        &\lesssim_d
        \|k_\Omega\|_{G^\infty_{\delta,x}((t/2,t)\times  B_{\sqrt{t}}(x_0))}t^{-d/2}e^{-\frac{d_\Omega(x)^2}{4(1+\delta)t}}\,,
	\end{align*}   
	which again gives the claimed bound.
\end{proof}


\section{Estimates close to flat parts of the boundary}

Our goal in this section is to prove that the on-diagonal heat kernel $k_\Omega(t, x, x)$ is closely approximated by the corresponding kernel in a suitably chosen half-space, as long as the distance from $x$ to the boundary is at most of the order $\sqrt{t}$ for $t$ sufficiently small and the nearby boundary is sufficiently flat in a specific sense. The precise notion of local flatness of the boundary matches almost precisely that of Brown \cite{Brown93}, and is captured in terms of two quantities that we define next.

\subsection{Notions of flatness}

\begin{definition}\label{def: good point}
	For $\epsilon\in (0, 1/2], r>0$ we say that $x_0 \in \partial\Omega$ is $(\epsilon, r)$-good if
    \begin{enumerate}[label=\textup{(}\roman*\textup{)}]
    \item\label{itm: good point local parametrization} after a suitable rotation $\partial\Omega$ can in a neighborhood of $x_0$ be parametrized as the graph of a function that is differentiable at $x_0$, and 
    \item\label{itm: good point cone condition} it holds that
	\begin{equation*}
		\partial\Omega \cap B_r(x_0) \subset \{x\in \R^d: |(x-x_0)\cdot \nu(x_0)|<\epsilon \, |x-x_0|\}\,,
	\end{equation*}
    where $\nu(x_0)$ denotes the outwards pointing unit normal to $\partial\Omega$. 
    \end{enumerate}
\end{definition}
\begin{remark}
    It is a consequence of \ref{itm: good point local parametrization} in the definition that the outwards pointing normal $\nu(x_0)$ in \ref{itm: good point cone condition} is well defined. For much of what we do it would suffice that there exists a unit vector $\nu$ such that
    \begin{align*}
        \Omega \cap B_r(x_0) &\subset \{x\in \R^d: (x-x_0)\cdot \nu<\epsilon \, |x-x_0|\}\,,\\
        \Omega^c \cap B_r(x_0) &\subset \{x\in \R^d: (x-x_0)\cdot \nu>-\epsilon \, |x-x_0|\}\,,
    \end{align*}
    but since our argument at certain places requires that $\partial\Omega$ be (locally) Lipschitz the stronger assumption that the condition in \ref{itm: good point cone condition} holds with respect to the outward pointing unit normal is a rather small restriction for our purposes.
\end{remark}

\begin{definition}\label{def: gammarxeps}
	If $x_0 \in \partial\Omega$ is an $(\epsilon, r)$-good point, we define
	\begin{equation*}
		\Gamma_r(x_0, \epsilon) := \{x \in \R^d: (x-x_0)\cdot \nu(x_0)<-\sqrt{1-\epsilon^2} \, |x-x_0|\}\cap B_{r/2}(x_0)\,.
	\end{equation*}
\end{definition} 

If $\partial\Omega \cap B_r(x_0)$ is Lipschitz regular, then Rademacher's theorem implies that the outwards pointing unit normal to $\partial\Omega$ is well defined for $\Haus^{d-1}$-almost every $x\in \partial\Omega \cap B_r(x_0)$. Let $\nu \colon \partial\Omega \to \S^{d-1}$ be this normal field, defined wherever it exists. For our estimates for the heat kernel close to the boundary we shall, in addition to assuming that the boundary is flat in the sense of the definition above, need to control the local oscillations of the normal vector.

\begin{definition}
Fix $p\geq 1$. For $x_0\in \partial\Omega$ and $0<s\leq r$ such that $\partial\Omega \cap B_r(x_0)$ is Lipschitz and the inward pointing unit normal $\nu(x_0)$ to $\partial\Omega$ at $x_0$ is defined, let
\begin{equation*}
    \bar\nu_p(x_0, s):=\sup_{0<\eta<s}\biggl(\frac{1}{\Haus^{d-1}(B_\eta(x_0)\cap \partial\Omega)}\int_{B_\eta(x_0)\cap \partial\Omega} |\nu(x_0)-\nu(y)|^p \,d\Haus^{d-1}(y)\biggr)^{1/p}\,.
\end{equation*}
\end{definition}

The points $x\in\Omega$ where we shall prove more precise results about the behavior of $k_\Omega(t, x, x)$ will be those in $\Gamma_r(x_0, \epsilon)$ for some $(\epsilon, r)$-good point $x_0$ with $\epsilon$ and $r$ suitably related to $t$. The error in our approximation will depend on the oscillations of the normal in terms of the quantity $\bar\nu_p(x_0,r)$ for suitably chosen $p\geq 1$.

An important aspect of the above notion of local flatness is that if $x \in \Gamma_r(x_0,\epsilon)$ and $y \in \partial\Omega \cap B_r(x_0)$ for some $(\epsilon, r)$-good $x_0\in \partial\Omega$, then the vectors $x-x_0$ and $y-x_0$ are almost orthogonal. As this property will be used frequently in what follows, we record it in the following lemma.

\begin{lemma}\label{lem: distance to boundary in good set}
	Let $x_0 \in \partial \Omega$ be $(\epsilon, r)$-good. Then, for all $x \in \Gamma_r(x_0, \epsilon)$ and $y \in \partial\Omega \cap B_r(x_0)$ we have, setting $x^*:= 2x_0-x$,
	\begin{align*}
        |(x-x_0)\cdot(y-x_0)|&\leq 2\epsilon |x-x_0||y-x_0|\,,\\ 
		\Bigl||x-y|^2 -|x-x_0|^2-|y-x_0|^2\Bigr|&\leq  2\epsilon |x-x_0|^2+2\epsilon |y-x_0|^2\,,\\
		\Bigl||x^*-y|^2 -|x-x_0|^2-|y-x_0|^2\Bigr|&\leq  2\epsilon |x-x_0|^2+2\epsilon |y-x_0|^2\,.
	\end{align*}
	In particular,
	\begin{equation*}
		d_\Omega(x) \leq |x-x_0| = |x^*-x_0| \leq \frac{d_\Omega(x)}{\sqrt{1-2\epsilon}} \,.
	\end{equation*}
\end{lemma}

\begin{proof}
	Define the orthogonal projections $P$ and $P^\perp$ by $P z := (z\cdot\nu(x_0))\nu(x_0)$ and $P^\perp := \1-P$. Then
    \begin{align*}
        (x-x_0)\cdot (y-x_0)= P(x-x_0)\cdot P(y-x_0)+P^\perp(x-x_0)\cdot P^\perp(y-x_0)
    \end{align*}
    and
\begin{align*}
	|x-y|^2 &= |(x-x_0)-(y-x_0)|^2 \\
	&= |x-x_0|^2 + |y-x_0|^2 - 2(x-x_0)\cdot(y-x_0)\\
	&=
	|x-x_0|^2 + |y-x_0|^2 - 2P(x-x_0)\cdot P(y-x_0) \\
    &\quad - 2P^\perp(x-x_0)\cdot P^\perp(y-x_0)\,.
\end{align*}
 Since $x_0$ is $(\epsilon, r)$-good, $y \in \partial \Omega \cap B_{r}(x_0)$, and $x \in \Gamma_{r}(x_0, \epsilon)$, it holds that $|P(y-x_0)|<\epsilon |y-x_0|$ and $|P^\perp(x-x_0)|<\epsilon |x-x_0|$. Therefore,
 \begin{equation*}
    |(x-x_0)\cdot (y-x_0)| \leq 2\epsilon|x-x_0||y-x_0|
 \end{equation*}
 and
\begin{equation*}
	\Bigl||x-y|^2 -|x- x_0|^2-|y- x_0|^2\Bigr|\leq 4\epsilon |x- x_0||y- x_0| \leq 2\epsilon |x- x_0|^2+2\epsilon |y- x_0|^2\,,
\end{equation*} 
which proves the first two of the claimed inequalities. The third follows from the same computation, using $x^*-x_0 = 2x_0-x-x_0 = -(x-x_0)$.

The remaining lower bound follows by noting that $x_0\in \partial\Omega$, so $d_\Omega(x) \leq |x-x_0|$, and the remaining upper bound follows by applying the first set of bounds with $y \in \partial \Omega \cap B_r(x_0)$ chosen to satisfy $|x-y|=d_\Omega(x)$.
\end{proof}


\subsection{The heat kernel near flat parts of the boundary}
Our aim in what remains of this section is to prove the following result. We recall that the quantity $\Theta[\omega]$ for subsets $\omega\subseteq\partial\Omega$ was defined before Theorem \ref{thm: neumann bvp Lp estimates}.

\begin{proposition}\label{prop: Estimate close the good boundary}
	Let $\Omega \subset \R^d$, $d\geq 2$, be open. Fix $r>0, \epsilon\in (0, 1/4], R>0$. Assume that $\partial\Omega \cap B_{2r}(x_0)$ is Lipschitz regular, that $x_0 \in \partial\Omega$ is $(\epsilon, r)$-good, and that $\partial\Omega$ is locally parametrized by a function $\varphi$ in $B_R(x_0)$ with $\|\nabla \varphi\|_{L^\infty}\leq M$.

    Fix $p>d-1, \delta>0, \kappa>0$. Then, for all $t\in(0,r^2/2]\cap(0, \kappa R^2]$ and $x\in \Gamma_r(x_0, \epsilon)$, we have
	\begin{align*}
    	\Bigl|(4\pi t)^{d/2}k_\Omega(t, x, x) - 1 - e^{-\frac{d_\Omega(x)^2}{t}}\Bigr|
        &\lesssim_{d,\delta, M, \kappa, p}\|k_\Omega\|_{G^\infty_{\delta,x}((0, 3t/2]\times B_{2r}(x_0))}\Theta[\partial\Omega \cap B_{2r}(x_0)]\\
        &\qquad \times\Bigl((\epsilon +\bar \nu_p(x_0, r))e^{- \frac{d_\Omega(x)^2}{c(1+\delta)t}}+ e^{- \frac{r^2}{c(1+\delta)t}}\Bigr)
    \end{align*}
    with some universal constant $c>0$.
\end{proposition}

The idea of the proof is similar to that of Proposition~\ref{prop: improved bound bulk}, but the local approximation by the free heat kernel is replaced by the sum of the free heat kernel and a reflected copy of it (in other words, the heat kernel in an appropriately chosen half-space). Before proving the proposition, we make some preliminary observations about the boundary value problem that we shall study and prove some initial bounds on the relevant boundary data.  

For $x, x_0$ as in the proposition we define
$$
x^* := 2 x_0- x \,,
$$
which should be understood as an `approximate reflection' of $x$ through $\partial\Omega$. 

Moreover, with the parameter $r$ as in the proposition, we choose $\chi \in C_0^\infty(\R^d)$ satisfying 
\begin{equation*}
    \1_{B_r(x_0)}\leq \chi\leq \1_{B_{2r}(x_0)}\,, \quad |\nabla \chi|\lesssim r^{-1}\,,\quad |\Delta \chi|\lesssim r^{-2} \,.
\end{equation*}

For $y \in \Omega$ and $s>0$ we define
\begin{equation*}
	u(s, y) := k_\Omega(s, x, y) - (k_{\R^d}(s, x, y)+k_{\R^d}(s, x^*, y))\chi(y)\,.
\end{equation*}
Our aim is to prove that $u(s, y)$ is small when $y=x$ and $s=t$. Once such a bound is established, the bound in Proposition \ref{prop: Estimate close the good boundary} will follow rather easily.

Define $f \colon (0, \infty)\times \partial\Omega \to \R$ and $g \colon (0, \infty)\times \Omega \to \R$ by 
\begin{align}
    \label{eq:deff}
	f(s, y) &:=-\frac{\partial}{\partial\nu_y} ((k_{\R^d}(s, x, y)+ k_{\R^d}(s, x^*, y))\chi(y))\,,\\
    g(s, y) &:=(k_{\R^d}(s, x, y)+k_{\R^d}(s, x^*, y))\Delta\chi(y)\notag \\
    &\quad +2\nabla_y (k_{\R^d}(s, x, y) + k_{\R^d}(s, x^*, y))\cdot \nabla \chi(y)\,,\label{eq:defg}
\end{align}
and note that $u$ solves
\begin{align*}
\begin{cases}
		(\partial_s-\Delta_y)u(s, y) =   g(s, y) & \mbox{for } (s, y)\in (0, \infty)\times \Omega\,,\\
		u(0, y) = 0 & \mbox{for } y\in \Omega\,,\\
		\frac{\partial}{\partial \nu_y} u(s, y) = f(s, y)  &\mbox{for } (s, y) \in (0, \infty) \times \partial\Omega\,.
	\end{cases}	
\end{align*}
We will be able to prove that $u(x,t)$ is small by exploiting the results from Section~\ref{sec: pde theorems}. Those will rely on precise bounds on $f$, which we will derive in the following two subsections. In contrast, for $g$ rather crude bounds are enough.

 
\subsection{Local bounds on \texorpdfstring{$f$}{f}}
We work in the setting of Proposition \ref{prop: Estimate close the good boundary}. More precisely, we fix an $(\epsilon,r)$-good point $x_0\in\partial\Omega$ and a point $x\in\Gamma_r(x_0,\epsilon)$, and we define $f$ by \eqref{eq:deff} (depending on $x$ and $x_0$). Our goal in this subsection will be to gather various bounds on the function $f$.

Note that, by the explicit form of $k$ and since $\chi=1$ in $B_r(x_0)$, it holds that
\begin{align}
    \label{eq:fexplicit}
	 f(s, y) &=\frac{1}{2s(4\pi s)^{d/2}} \, \nu(y)\cdot \Bigl[(y-x)e^{- \frac{|x-y|^2}{4s}}+(y-x^*)e^{- \frac{|x^*-y|^2}{4s}}\Bigr]
\end{align}
for all $(s, y)\in (0, \infty)\times B_r(x_0)$.
Similarly, by the assumptions on $\chi$ and since $x, x^*\in B_{r/2}(x_0)$, we have
\begin{equation}\label{eq: bound bdry data faraway}
	 |f(s, y)| \lesssim s^{-d/2}(rs^{-1}+r^{-1})e^{- \frac{r^2}{16s}}\1_{B_{2r}(x_0)}(y) \quad \mbox{for }(s, y)\in (0, \infty)\times B_r(x_0)^c\,.
\end{equation}

For $\rho>0$, we shall use the notation
$$
I(\rho) := (0,\rho^2) \times (B_\rho(x_0) \cap \partial\Omega) \,.
$$

\begin{lemma}\label{lem: f bound 1}
    If $2d_\Omega(x)\leq \rho \leq r$, then
\begin{equation*}
	\biggl|\iint_{I(\rho)}f(s,y)\,ds\,d\Haus^{d-1}(y)\biggr|\lesssim_d \epsilon\,.
\end{equation*}
\end{lemma}

\begin{proof}
Introduce the sets 
\begin{align*}
    \Omega_\rho^\limplus & := \Omega \cap B_\rho(x_0) \,,\\
    \Omega_\rho^\limminus & := ({\overline\Omega})^c \cap B_\rho(x_0) \,,\\
    \Sigma_\rho^\limplus & := (\partial \Omega_\rho^\limplus) \cap \Omega \,,\\ 
    \Sigma_\rho^\limminus & := (\partial \Omega_\rho^\limminus) \cap (\overline\Omega)^c \,.
\end{align*}
Let also
\begin{equation*}
	C_{\epsilon} := \{z \in B_\rho(x_0):  |(z-x_0)\cdot \nu(x_0)| <\epsilon |z-x_0|\}\,.
\end{equation*}
By Green's formula and the fact that $k$ solves the heat equation, we can rewrite the integral of $f$ as follows
\begin{align*}
 	\iint_{I(\rho)}f(s,y)\,ds\,d\Haus^{d-1}(y) &=  
 	-\int_{\Omega^\limplus_\rho}k_{\R^d}(\rho^2, x, y)\,dy+\int_{\Omega^\limminus_\rho}k_{\R^d}(\rho^2, x^*, y)\,dy\\
 	&\quad
 	+ \int_0^{\rho^2}\int_{\Sigma_\rho^\limplus} \frac{\partial}{\partial \nu_y}k_{\R^d}(s, x, y)\,d\Haus^{d-1}(y)ds\\
 	&\quad - \int_0^{\rho^2}\int_{\Sigma_\rho^\limminus} \frac{\partial}{\partial \nu_y}k_{\R^d}(s, x^*, y)\,d\Haus^{d-1}(y)ds\\
 	&=
 	-\int_{\Omega^\limplus_\rho \cap C_{\epsilon}}k_{\R^d}(\rho^2, x, y)\,dy+\int_{\Omega^\limminus_\rho \cap C_{\epsilon}}k_{\R^d}(\rho^2, x^*, y)\,dy\\
 	&\quad
 	+ \int_0^{\rho^2}\int_{\Sigma_\rho^\limplus\cap C_{\epsilon}} \frac{\partial}{\partial \nu_y}k_{\R^d}(s, x, y)\,d\Haus^{d-1}(y)ds\\
 	&\quad - \int_0^{\rho^2}\int_{\Sigma_\rho^\limminus\cap C_{\epsilon}} \frac{\partial}{\partial \nu_y}k_{\R^d}(s, x^*, y)\,d\Haus^{d-1}(y)ds\,,
\end{align*} 
where, in the second step, we used the reflection symmetry between $k_{\R^d}(s, x, y)$ and $k_{\R^d}(s, x^*, y)$ to cancel large pieces of the integrals. (Note that all the non-trivial geometry happens in $C_{\epsilon}$ as $x_0$ is $(\epsilon, r)$ good.)

For all $y \in C_\epsilon \cap B_\rho(x_0)$ we can bound
\begin{equation*}
	|k_{\R^d}(\rho^2, x, y)| + |k_{\R^d}(\rho^2, x^*, y)| \lesssim_d \rho^{-d}\,.
\end{equation*}
Similarly, for all $y \in C_\epsilon \cap \partial B_\rho(x_0)$,
\begin{equation*}
	|\nabla_y k_{\R^d}(s, x, y)|+|\nabla_y k_{\R^d}(s, x^*, y)| \lesssim_{d} s^{-(d+1)/2}e^{-c\rho^2/s}
\end{equation*}
for some universal constant $c$ by Lemma~\ref{lem: distance to boundary in good set} using the fact that $x \in \Gamma_r(x_0, \epsilon)$, $y \in C_\epsilon \cap \partial B_\rho(x_0)$, $\epsilon \leq \frac{1}{4}$, and the assumption that $\rho\geq 2d_\Omega(x)$. We have thus arrived at the bound
\begin{align*}
 	\biggl|\iint_{I(\rho)}&f(s,y)\,ds\,d\Haus^{d-1}(y)\biggr|\\ &\lesssim_d \frac{|C_\epsilon \cap B_\rho(x_0)|}{\rho^d}+ \Haus^{d-1}(\partial B_\rho(x_0)\cap C_\epsilon)\int_0^{\rho^2}s^{-(d+1)/2}e^{-c\rho^2/s}\,ds\\
 	&\lesssim_d \frac{|C_\epsilon \cap B_\rho(x_0)|}{\rho^d}+ \rho^{-d+1}\Haus^{d-1}(\partial B_\rho(x_0)\cap C_\epsilon)\\
 	&\lesssim_d \epsilon \,.
\end{align*} 
This completes the proof of Lemma~\ref{lem: f bound 1}.
\end{proof}

In the following two lemmas, we prove bounds for $f$ in terms of the oscillation of the normal vector.

\begin{lemma}\label{lem: f bound 2}
    If $\rho\in (0, r/2]$ and if $(s, y) \in I(2\rho)\setminus I(\rho)$ is such that $\nu(y)$ exists, then
\begin{equation*}
	|f(s, y)|\lesssim_{d}  (\epsilon+ |\nu(y)-\nu(x_0)|) \, \rho^{-(d+1)} \,.
\end{equation*}
\end{lemma}

\begin{proof}
We write
\begin{equation*}
	I(2\rho)\setminus I(\rho) = U_1 \cup U_2
\end{equation*}
with
\begin{equation*}
	U_1 := [\rho^2, 4\rho^2) \times (\partial \Omega \cap B_{2\rho}(x_0)) \quad \mbox{and}\quad U_2 := (0, 4\rho^2)\times (\partial\Omega \cap B_{2\rho}(x_0)\setminus B_\rho(x_0))\,.
\end{equation*}
Since $\epsilon \leq 1/4$, Lemma~\ref{lem: distance to boundary in good set} implies that 
$$
	\min\{|x-y|^2, |x^*-y|^2\} \geq \rho^2/2 \quad \mbox{for}\quad (s, y)\in U_2\,.
$$
Consequently, for any $\beta>0, \alpha>0$,
\begin{equation}\label{eq: fundamental bounds in U1 and U2}
	s^{-\alpha}\frac{|x-y|^{2\beta}}{s^{\beta}}e^{-\frac{|x-y|^2}{4s}}+s^{-\alpha}\frac{|x^*-y|^{2\beta}}{s^{\beta}}e^{-\frac{|x^*-y|^2}{4s}} \lesssim_{\beta, \alpha} \rho^{-2\alpha} \quad \mbox{for }(s, y)\in U_1 \cup U_2\,.
\end{equation}

Using the explicit expression \eqref{eq:fexplicit}, we split $f$ into three parts and bound, using $2y-x-x^* = 2y-2x_0$,
\begin{equation}\label{eq: f local bound split}
\begin{aligned}
	\hspace{-10pt}|f(s, y)|
	&\lesssim_{d}
	s^{-(d+1)/2}\Bigl|\nu(y)\cdot\Bigl(\frac{y-x}{\sqrt{s}}e^{-\frac{|x-y|^2}{4s}}+\frac{y-x^*}{\sqrt{s}}e^{-\frac{|x^*-y|^2}{4s}}\Bigr)\Bigr|\\
	&\lesssim_{d}
	s^{-(d+1)/2}|\nu(y)-\nu(x_0)|\biggl(\frac{|y-x|}{\sqrt{s}}e^{-\frac{|x-y|^2}{4s}}+\frac{|y-x^*|}{\sqrt{s}}e^{-\frac{|x^*-y|^2}{4s}}\biggr)\\
	&\quad +
	s^{-(d+1)/2}\frac{|\nu(x_0)\cdot(y- x_0)|}{\sqrt{s}}e^{-\frac{|x-y|^2}{4s}} +
	s^{-(d+1)/2}\frac{|y-x^*|}{\sqrt{s}}\Bigl|e^{-\frac{|x^*-y|^2}{4s}}-e^{-\frac{|x-y|^2}{4s}}\Bigr|.\hspace{-30pt}
\end{aligned}
\end{equation}
We bound each of the three terms in the right-hand side of \eqref{eq: f local bound split} separately.

For the first term in the right-hand side of \eqref{eq: f local bound split}, the bound~\eqref{eq: fundamental bounds in U1 and U2} yields
\begin{equation*}
	s^{-(d+1)/2}|\nu(y)-\nu(x_0)|\biggl(\frac{|y-x|}{\sqrt{s}}e^{-\frac{|x-y|^2}{4s}}+\frac{|y-x^*|}{\sqrt{s}}e^{-\frac{|x^*-y|^2}{4s}}\biggr)
	\lesssim_{d}
	 |\nu(y)-\nu(x_0)|\rho^{-(d+1)}\,,
\end{equation*}
for all $(s, y) \in U_1 \cup U_2$.

To bound the second term in the right-hand side of \eqref{eq: f local bound split}, we note that since $x_0$ is $(\epsilon, r)$-good it holds that 
$$
    |\nu(x_0)\cdot(y-x_0)|<\epsilon |y-x_0|< 2\epsilon \rho
$$ 
for all $y \in \partial\Omega \cap B_{2\rho}(x_0)$. When combined with~\eqref{eq: fundamental bounds in U1 and U2} we find
\begin{align*}
	s^{-(d+1)/2}\frac{|\nu(x_0)\cdot(y- x_0)|}{\sqrt{s}}e^{-\frac{|x-y|^2}{4s}} \leq 2\epsilon \rho s^{-(d+2)/2}e^{-\frac{|x-y|^2}{4s}}\lesssim_{d} \epsilon \rho^{-(d+1)}
\end{align*}
for all $(s, y) \in U_1 \cup U_2$.

We finally bound the third term in \eqref{eq: f local bound split}. In this case we need to capture cancellation between the terms in the absolute values. By expanding squares
\begin{equation*}
	e^{-\frac{|x^*-y|^2}{4s}}-e^{-\frac{|x-y|^2}{4s}} = e^{-\frac{|x^*-y|^2}{4s}}\Bigl(1-e^{-\frac{|x-y|^2-|x^*-y|^2}{4s}}\Bigr) = e^{-\frac{|x^*-y|^2}{4s}}\Bigl(1-e^{\frac{(x-x_0)\cdot (y-x_0)}{s}}\Bigr)\,.
\end{equation*}
Since $x_0$ is $(\epsilon, r)$-good and $y\in \partial\Omega \cap B_r(x_0), x\in \Gamma_r(x_0, \epsilon)$, Lemma~\ref{lem: distance to boundary in good set} yields that
\begin{equation*}
	|(x-x_0)\cdot (y-x_0)| \leq 2 \epsilon |x-x_0| |y-x_0|\,.
\end{equation*}
Therefore, whenever $\frac{\epsilon|x-x_0||y-x_0|}s \leq 1$ a Taylor expansion implies that
\begin{equation*}
	\Bigl|e^{-\frac{|x^*-y|^2}{4s}}-e^{-\frac{|x-y|^2}{4s}}\Bigr| \leq  e^{-\frac{|x^*-y|^2}{4s}}\bigl|1-e^{\frac{(x-x_0)\cdot (y-x_0)}{s}}\bigr| \lesssim \epsilon \frac{|x-x_0||y-x_0|}{s} e^{-\frac{|x^*-y|^2}{4s}}\,.
\end{equation*}
Since $|y-x_0|\leq 2\rho$ and, by Lemma~\ref{lem: distance to boundary in good set}, $|x-x_0|\lesssim |x^*-y|$ we conclude that
\begin{equation*}
	\Bigl|e^{-\frac{|x^*-y|^2}{4s}}-e^{-\frac{|x-y|^2}{4s}}\Bigr| \lesssim \epsilon \frac{\rho|x^*-y|}{s} e^{-\frac{|x^*-y|^2}{4s}}\quad \mbox{if } \epsilon|x-x_0||y-x_0| \leq s\,.
\end{equation*}
Therefore, by~\eqref{eq: fundamental bounds in U1 and U2} 
\begin{equation*}
	s^{-(d+1)/2}\frac{|y-x^*|}{\sqrt{s}}\Bigl|e^{-\frac{|x^*-y|^2}{4s}}-e^{-\frac{|x-y|^2}{4s}}\Bigr| 
	\lesssim \epsilon \rho^{-(d+1)}
\end{equation*}
if $(s, y) \in U_1\cup U_2$ and $\epsilon|x-x_0||y-x_0| \leq s$.

Furthermore, as $2|x-x_0||y-x_0|\leq |x-x_0|^2+|y-x_0|^2$ it follows from Lemma~\ref{lem: distance to boundary in good set} that
\begin{equation*}
	\epsilon|x-x_0||y-x_0| > s \quad \Longrightarrow \quad \min\{|x-y|^2, |x^*-y|^2\} \gtrsim \frac{s}{\epsilon}\,. 
\end{equation*}
Therefore, using that by Lemma~\ref{lem: distance to boundary in good set} $|x-y| \lesssim |x^*-y|$, we have that if $(s, y) \in U_1\cup U_2,\ \epsilon|x-x_0||y-x_0| > s$ then
\begin{equation*}
	s^{-(d+1)/2}\frac{|y-x^*|}{\sqrt{s}}\Bigl|e^{-\frac{|x^*-y|^2}{4s}}-e^{-\frac{|x-y|^2}{4s}}\Bigr| \lesssim e^{- c/\epsilon}s^{-(d+1)/2}\frac{|y-x^*|}{\sqrt{s}}e^{-c\frac{|x^*-y|^2}{s}} \lesssim \epsilon \rho^{-(d+1)}\,,
\end{equation*}
for some absolute constant $c>0$.

Combining the three bounds shown above completes the proof of Lemma~\ref{lem: f bound 2}. 
\end{proof}

Finally, we will need to estimate $f$ in $I(\rho)$ for $\rho$ relatively small compared to both $d_\Omega(x)$ and $r$.

\begin{lemma}\label{lem: f bound 3}
    If $\rho \in (0, 4d_\Omega(x)]\cap (0, r]$ and if $(s, y) \in I(\rho)$ is such that $\nu(y)$ exists, then
    \begin{equation*}
	|f(s, y)| \lesssim_{d} (\epsilon+|\nu(y)-\nu(x_0)|) \, d_\Omega(x)^{-(d+1)} \,.
    \end{equation*}
\end{lemma}

\begin{proof}
Since $x \in \Omega$ and $(s, y) \in I(\rho) \subset (0, \infty)\times \partial\Omega$, we have
\begin{equation*}
	|x-y|^2\geq  d_\Omega(x)^2\,.
\end{equation*}
Moreover, using Lemma~\ref{lem: distance to boundary in good set} and that $x_0 \in \partial \Omega$ is $(\epsilon, r)$-good and $x\in \Gamma_r(x_0,\epsilon)$ we have
\begin{equation*}
	|x^*-y|^2 \geq (1-2\epsilon) |x-x_0|^2 \geq (1-2\epsilon) d_\Omega(x)^2 \,.
\end{equation*}
By Lemma \ref{lem: distance to boundary in good set}, $|x-x_0|=|x^*-x_0|\lesssim d_\Omega(x)$ and thus since $y \in B_{\rho}(x_0)$ with $\rho \leq 4 d_\Omega(x)$ we have
$$
    \max\{|x-y|,|x^*-y|\} \lesssim  d_\Omega(x)\,.
$$
Therefore, the assumptions of the lemma imply that
\begin{equation*}
    d_\Omega(x)\lesssim |x-y|\lesssim d_\Omega(x) \quad \mbox{and}\quad d_\Omega(x)\lesssim |x^*-y|\lesssim d_\Omega(x)\,.
\end{equation*}

Consequently, for $(s, y)$ as in the statement the above bounds for $|x-y|, |x^*-y|$ together with $x_0$ being $(\epsilon, r)$-good imply that
\begin{equation*}
\begin{aligned}
	|f(s, y)|
	&\lesssim_{d}
	s^{-(d+1)/2}\Bigl|\nu(y)\cdot\Bigl(\frac{y-x}{\sqrt{s}}e^{-\frac{|x-y|^2}{4s}}+\frac{y-x^*}{\sqrt{s}}e^{-\frac{|x^*-y|^2}{4s}}\Bigr)\Bigr|\\
	&\lesssim_{d}
	s^{-(d+1)/2}|\nu(y)-\nu(x_0)|\biggl(\frac{|y-x|}{\sqrt{s}}e^{-\frac{|x-y|^2}{4s}}+\frac{|y-x^*|}{\sqrt{s}}e^{-\frac{|x^*-y|^2}{4s}}\biggr)\\
	&\quad +
	s^{-(d+1)/2}\frac{|\nu(x_0)\cdot(y- x_0)|}{\sqrt{s}}e^{-\frac{|x-y|^2}{4s}} +
	s^{-(d+1)/2}\frac{|y-x^*|}{\sqrt{s}}\Bigl|e^{-\frac{|x^*-y|^2}{4s}}-e^{-\frac{|x-y|^2}{4s}}\Bigr|\\
    &\lesssim_{d}
	s^{-(d+1)/2}|\nu(y)-\nu(x_0)|\frac{d_\Omega(x)}{\sqrt{s}}e^{-\frac{d_\Omega(x)^2}{8s}}\\
	&\quad +
	s^{-(d+1)/2}\frac{\epsilon d_\Omega(x)}{\sqrt{s}}e^{-\frac{d_\Omega(x)^2}{8s}} +
	s^{-(d+1)/2}\frac{|y-x^*|}{\sqrt{s}}\Bigl|e^{-\frac{|x^*-y|^2}{4s}}-e^{-\frac{|x-y|^2}{4s}}\Bigr|\\
    &\lesssim_{d}
	(\epsilon + |\nu(y)-\nu(x_0)|)d_\Omega(x)^{-(d+1)} +
	s^{-(d+1)/2}\frac{|y-x^*|}{\sqrt{s}}\Bigl|e^{-\frac{|x^*-y|^2}{4s}}-e^{-\frac{|x-y|^2}{4s}}\Bigr|\,.
\end{aligned}
\end{equation*}

To bound the remaining term we argue as in the proof of Lemma \ref{lem: f bound 2}. As in that proof one arrives at the conclusion that 
\begin{equation*}
	\Bigl|e^{-\frac{|x^*-y|^2}{4s}}-e^{-\frac{|x-y|^2}{4s}}\Bigr| \leq  e^{-\frac{|x^*-y|^2}{4s}}\bigl|1-e^{\frac{(x-x_0)\cdot (y-x_0)}{s}}\bigr| \lesssim \epsilon \frac{|x-x_0||y-x_0|}{s} e^{-\frac{|x^*-y|^2}{4s}}\,,
\end{equation*}
if $\epsilon|x-x_0||y-x_0| \leq s$.
Since now $|y-x_0|\leq \rho \leq 4d_\Omega(x)$, $|x-x_0|\lesssim  d_\Omega(x)$, and $|x^*-y|\gtrsim d_\Omega(x)$ we conclude that
\begin{equation*}
	\Bigl|e^{-\frac{|x^*-y|^2}{4s}}-e^{-\frac{|x-y|^2}{4s}}\Bigr| \lesssim \epsilon \frac{d_\Omega(x)^2}{s} e^{-\frac{d_\Omega(x)^2}{8s}}
\end{equation*}
if $\epsilon|x-x_0||y-x_0| \leq s$. Therefore, by maximizing over $s>0$,
\begin{equation*}
	s^{-(d+1)/2}\frac{|y-x^*|}{\sqrt{s}}\Bigl|e^{-\frac{|x^*-y|^2}{4s}}-e^{-\frac{|x-y|^2}{4s}}\Bigr| 
	\lesssim \epsilon d_\Omega(x)^{-(d+1)}
\end{equation*}
for all $(s, y) \in I(\rho)$ with $\epsilon|x-x_0||y-x_0| \leq s$.

In the setting of the current lemma, $\epsilon |x-x_0||y-x_0|>s$ implies that $\epsilon d_\Omega(x)^2\gtrsim s$ and so
\begin{equation*}
	s^{-(d+1)/2}\frac{|y-x^*|}{\sqrt{s}}\Bigl|e^{-\frac{|x^*-y|^2}{4s}}-e^{-\frac{|x-y|^2}{4s}}\Bigr| \lesssim e^{- c/\epsilon}s^{-(d+1)/2}\frac{d_\Omega(x)^2}{\sqrt{s}}e^{-c\frac{d_\Omega(x)^2}{s}} \lesssim \epsilon d_\Omega(x)^{-(d+1)}\,,
\end{equation*}
for some absolute constant $c>0$.

This completes the proof of Lemma~\ref{lem: f bound 3}.
\end{proof}


\subsection{Decomposing the boundary data}
\label{sec: Decomposition of f}
The next step in the proof of Proposition~\ref{prop: Estimate close the good boundary} is to decompose the boundary datum $f$ into smaller pieces. We recall that $x_0$ is an $(\epsilon,r)$-good point, that $x\in\Gamma_r(x_0,\epsilon)$ and that $f$ was defined in terms of $x$ and $x_0$. In this subsection we also consider a time $t\in(0,r^2/2]$. The decomposition of $f$ will depend on $t$ and, more precisely, on the parabolic distance between $(t, x)$ and $(0, x_0)$. We recall the notation $I(\rho) = (0, \rho^2) \times (B_{\rho}(x_0) \cap \partial\Omega)$.

Given $x_0,x$ and $t$, our goal in this subsection will be to decompose
\begin{equation}
    \label{eq:fdecomp}
    f(s, y) = \sum_{1 \leq j \leq N}\tilde f_j(s, y) + \sum_{1\leq k \leq M} f_k(s, y) + f^\sharp(s, y) + \check f(s, y)\,.
\end{equation}

Let $M$ be the smallest positive integer so that $r< 2^{M+1}\sqrt{t}$, and let $\rho_k := 2^k \sqrt{t}$ for $k = 0, \ldots, M$. We define
\begin{equation*}
	f_k(s, y) := \1_{I(\rho_k)\setminus I(\rho_{k-1})}(s, y)f(s, y) \quad \mbox{for }k=1, \ldots, M\,,
\end{equation*}
and
\begin{equation*}
	\tilde f(s, y) := \1_{I(\sqrt{t})}(s, y)f(s, y) \quad \mbox{and} \quad \check f(s, y) := \1_{I(\rho_M)^c}(s, y)f(s, y) \,.
\end{equation*}
Thus, we have
\begin{equation*}
	f(s, y) = \sum_{1\leq k\leq M} f_k(s, y) + \tilde f(s, y) + \check f(s, y)\,.
\end{equation*}

Next, we further decompose $\tilde f$. Define $\tilde \rho_0 :=0$ and $\tilde \rho_j := 2^{j} d_\Omega(x)$ for $j \geq 1$. Let $N$ be the smallest nonnegative integer such that $\sqrt{t}<2^{N+2}d_\Omega(x)$. Note, in particular, that $2\tilde \rho_j \leq\sqrt{t}$ for all $j \leq N$ and that $\tilde \rho_N >\sqrt{t}/4$. Also, note that $\tilde f = f$ on $I(\tilde \rho_j)$.

Let $m_0:=0$ and, for $1 \leq j \leq N$, set
\begin{equation*}
	m_j(s, y) := \1_{I(\tilde \rho_j)}(s, y) \, \frac{1}{\tilde\rho_j^2 \, \Haus^{d-1}(B_{\tilde\rho_j}(x_0)\cap\partial\Omega)} \iint_{I(\tilde \rho_j)} f(s',y')\,ds'\,d\Haus^{d-1}(y')\,.
\end{equation*}
We define
\begin{equation*}
	\tilde f_j(s, y) := \1_{I(\tilde \rho_j) \setminus I(\tilde \rho_{j-1})}(s, y) f(s, y) + m_{j-1}(s, y)-m_j(s, y) \quad \mbox{for }j = 1, \ldots, N
\end{equation*}
and
\begin{equation*}
	f^\sharp(s, y) := m_N(s, y) + \1_{I(\tilde \rho_N)^c}(s, y)\tilde f(s, y)= m_N(s, y) + \1_{I(\sqrt{t})\setminus I(\tilde \rho_N)}(s, y) f(s, y)\,.
\end{equation*}
Thus, we have
\begin{equation*}
	\tilde f(s, y) = \sum_{1\leq j\leq N}\tilde f_j(s, y) + f^\sharp(s, y)\,,
\end{equation*}
which gives the desired decomposition of $f$.

We note that if $4d_\Omega(x)>\sqrt{t}$ then $N=0$, and we have $\tilde f = f^\sharp$. 

In the following lemma we collect some of the support properties of the pieces in the decomposition \eqref{eq:fdecomp} and we prove some bounds that will be important in the next subsection.

\begin{lemma}\label{lem:fpiecesbound}
    The pieces in the decomposition \eqref{eq:fdecomp} satisfy:
    \begin{enumerate}
        \item[$(a)$] For every $1\leq j\leq N$, the function $\tilde f_j$ has integral zero and is supported in $I(\tilde\rho_j)$ with $\tilde\rho_j = 2^j d_\Omega(x)$. Moreover, $\tilde f_j \in L^p(I(\tilde\rho_j))$ for every $p\in [1, \infty]$ and
        $$
            \| \tilde f_j \|_{L^1} \lesssim_d \Theta[\partial\Omega\cap B_{\tilde\rho_j}(x_0)] \, (\epsilon + \overline\nu_1(x_0,\tilde\rho_j) ) \,.
        $$
        \item[$(b)$] For every $1\leq k\leq M$, the function $f_k$ is supported in $I(\rho_k)$ with $\rho_k = 2^k\sqrt t$ and satisfies, for any $p,q\geq 1$,
        \begin{equation*}
            \begin{aligned}
	           \|f_k\|_{L^{p}(\partial\Omega; L^q((0, t)))}
                \lesssim_d 
                \Theta[\partial\Omega\cap B_{\rho_k}(x_0)]^{1/p}\rho_k^{-(d+1)+ 2/q+(d-1)/p}(\epsilon +\bar \nu_p(x_0, \rho_k))\,.
            \end{aligned}
        \end{equation*}
        \item[$(c)$] The function $f^\sharp$ is supported in $(0,\infty)\times B_{\sqrt t}(x_0)$ and satisfies, for any $p,q\geq 1$,
        \begin{align*}
            \|f^\sharp\|_{L^{p}(\partial\Omega; L^q((0, t)))}
            & \lesssim_d
            \Theta[\partial\Omega\cap B_{\sqrt t}(x_0)]^{1/p} (\epsilon +\bar \nu_p(x_0, \sqrt t)) \, t^{1/q+(d-1)/(2p)} \\
            & \quad \times \min\{ d_\Omega(x)^{-(d+1)},t^{-(d-1)/2}\}    
        \end{align*}
        \item[$(d)$] The function $\check f$ is supported in $(0,\infty)\times (B_{2r}(x_0)\setminus B_{r/2}(x_0))$ and, for all $(s,y)\in(0,\infty)\times\partial\Omega$,
        $$
        |\check f(s,y)| \lesssim_d r^{-d-1} \,.
        $$       
    \end{enumerate}
\end{lemma}

\begin{proof}
    \emph{Step 1.} We consider the function $\tilde f_j$ with $1\leq j\leq N$. The mean-value zero and the support property follow immediately from the construction. That $\tilde f_j \in L^\infty(I(\tilde \rho_j))$ follows from Lemma~\ref{lem: f bound 2} if $j>1$ and Lemma \ref{lem: f bound 3} if $j=1$. It remains to bound $\|\tilde f_j\|_{L^1}$. We have
    $$
    \|\tilde f_j\|_{L^1} \leq \|f\|_{L^1(I(\tilde \rho_j)\setminus I(\tilde \rho_{j-1}))} + \|m_j\|_{L^1}+ \|m_{j-1}\|_{L^1} \,.
    $$
    Our assumption $t\leq r^2/2$ implies that $\tilde \rho_j \leq\sqrt{t}/2 < r/2$ for each $j$. We also have $\tilde\rho_j = 2^j d_\Omega(x) \geq 2d_\Omega(x)$. Therefore we can apply Lemma~\ref{lem: f bound 1} and obtain
    $$
    \|m_j\|_{L^1} = \left| \iint_{I(\tilde\rho_j)} f(s,y)\,ds\,\Haus^{d-1}(y) \right| \lesssim_d \epsilon \,.
    $$
    The same bound holds for $\|m_{j-1}\|_{L^1}$. (Note that the case $j=1$ is trivial since $m_0=0$.) Since is $(\epsilon, r)$-good we have $\Theta[\partial\Omega \cap B_{\tilde\rho_j}(x_0)]\gtrsim_d 1$. This proves that the contribution from $\|m_j\|_{L^1}+ \|m_{j-1}\|_{L^1}$ to the bound on $\|\tilde f_j\|_{L^1}$ is of the claimed form.

    To bound $\|f\|_{L^1(I(\tilde \rho_j)\setminus I(\tilde \rho_{j-1}))}$, we use Lemma~\ref{lem: f bound 2} if $j>1$ and Lemma~\ref{lem: f bound 3} if $j=1$, noting that $\tilde\rho_j\leq r/2$ and $\tilde\rho_1 = 2d_\Omega(x) \leq 4 d_\Omega(x)$. In both cases, we obtain
    \begin{align*}
        \|f\|_{L^1(I(\tilde \rho_j)\setminus I(\tilde \rho_{j-1}))}
        & \lesssim_d \tilde \rho_j^{-(d-1)}\int_{\partial\Omega \cap B_{\tilde \rho_j}(x_0)} (\epsilon + |\nu(y)-\nu(x_0)|)\,d\Haus^{d-1}(y) \\
        & \leq \tilde \rho_j^{-(d-1)} \Haus^{d-1}(B_{\tilde\rho_j}(x_0)\cap \partial\Omega) \, (\epsilon + \overline\nu_1(x_0,\tilde\rho_j)) \\
        & \leq \Theta[\partial\Omega\cap B_{\tilde\rho_j}(x_0)] (\epsilon + \bar \nu_1(x_0, \tilde \rho_j))\,.
    \end{align*}

    \medskip

    \noindent\emph{Step 2.} We consider the function $f_k$ with $k=1,\ldots, M$.

    Since $\supp f_k\subset I(\rho_k)\setminus I(\rho_{k-1})=I(\rho_k)\setminus I(\rho_k/2)$ and $\rho_k \leq r$ for each $k=1, \ldots, M$, we can apply Lemma~\ref{lem: f bound 2} to deduce that for any $p, q \geq 1$
\begin{equation*}
\begin{aligned}
	\|&f_k\|_{L^{p}(\partial\Omega; L^q((0, t)))}\\
    &\lesssim_d 
    \biggl(\int_{\partial\Omega\cap B_{\rho_k}(x_0)}\biggl(\int_0^{\rho_k^2} (\epsilon + |\nu(y)-\nu(x_0)|)^q \rho_k^{-(d+1)q}\,ds\biggr)^{p/q}\,d\Haus^{d-1}(y)\biggr)^{1/p}\\
    &=
    \rho_k^{-(d+1)+ 2/q}\biggl(\int_{\partial\Omega\cap B_{\rho_k}(x_0)} (\epsilon + |\nu(y)-\nu(x_0)|)^{p} \,d\Haus^{d-1}(y)\biggr)^{1/p}\\
    &\leq
    \rho_k^{-(d+1)+ 2/q}\biggl(\epsilon \Haus^{d-1}(\partial\Omega \cap B_{\rho_k}(x_0))^{1/p} +\biggl(\int_{\partial\Omega\cap B_{\rho_k}(x_0)} \!\!\! |\nu(y)-\nu(x_0)|^p \,d\Haus^{d-1}(y)\biggr)^{1/p}\biggr)\\
    &\leq
    \rho_k^{-(d+1)+ 2/q}\Haus^{d-1}(\partial\Omega \cap B_{\rho_k}(x_0))^{1/p}(\epsilon +\bar \nu_p(x_0, \rho_k))\\
    &\leq 
    \Theta[\partial\Omega\cap B_{\rho_k}(x_0)]^{1/p}\rho_k^{-(d+1)+ 2/q+(d-1)/p}(\epsilon +\bar \nu_p(x_0, \rho_k))\,.
\end{aligned}
\end{equation*}
This proves the claimed bound.

    \medskip

    \noindent\emph{Step 3.} To estimate $f^\sharp$, we distinguish the cases $\sqrt t< 4d_\Omega(x)$ (that is, $N=0$) and $\sqrt t\geq 4d_\Omega(x)$ (that is $N\geq 1$). 

    If $\sqrt t<4d_\Omega(x)$, then
    \begin{equation*}
 	  f^\sharp(s, y) = \1_{I(\sqrt{t})}(s, y)f(s, y)\,.
    \end{equation*}
    From Lemma \ref{lem: f bound 3} we obtain, similarly as in Step 2,
    \begin{equation*}
    \begin{aligned}
	\|f^\sharp&\|_{L^{p}(\partial\Omega; L^q((0, t)))}\\
    &\lesssim_d 
    \biggl(\int_{\partial\Omega\cap B_{\sqrt t}(x_0)}\biggl(\int_0^t (\epsilon + |\nu(y)-\nu(x_0)|)^q d_\Omega(x)^{-(d+1)q}\,ds\biggr)^{p/q}\,d\Haus^{d-1}(y)\biggr)^{1/p}\\
    &\lesssim
    \Theta[\partial\Omega\cap B_{\sqrt t}(x_0)]^{1/p} t^{1/q+(d-1)/(2p)} \, d_\Omega(x)^{-(d+1)}(\epsilon +\bar \nu_p(x_0, \rho_k))\,.
    \end{aligned}
    \end{equation*}
    This proves the claimed bound in this case.
    
    Meanwhile, if $\sqrt t \geq 4 d_\Omega(x)$, then
    \begin{equation*}
	   f^\sharp(s, y) = m_N(s,y) + \1_{I(\sqrt{t})\setminus I(\tilde \rho_N)}(s, y)f(s, y)\,.
    \end{equation*}
    To bound $\|f^\sharp\|_{L^{p}(\partial\Omega; L^q((0, t)))}$, we argue as in Step 1. By Lemma \ref{lem: f bound 1}, we have
    \begin{align*}
        \|m_N\|_{L^{p}(\partial\Omega; L^q((0, t)))} & = \left| \iint_{I(\tilde\rho_N)} f(s,y)\,ds\,d\Haus^{d-1}(y) \right| \tilde\rho_N^{-2/q'} \Haus^{d-1}(B_{\tilde\rho_N}(x_0)\cap\partial\Omega)^{-1/p'} \\
        & \lesssim_d \epsilon \, \tilde\rho_N^{-2/q'} \Haus^{d-1}(B_{\tilde\rho_N}(x_0)\cap\partial\Omega)^{-1/p'} \\
        & \leq \epsilon \, \Theta[B_{\tilde\rho_N}(x_0)\cap\partial\Omega]^{-1/p'} \tilde\rho_N^{-(d-1)/p'-2/q'}
    \end{align*}
    As $\Theta[\Omega \cap B_{\sqrt{t}}(x_0)]\gtrsim_{d} 1$ since $x_0$ is $(\epsilon, r)$-good and as $\tilde\rho_N\leq(1/2) \sqrt t$, this bound is of the claimed form. Furthermore, by Lemma \ref{lem: f bound 2}, we have
    \begin{align*}
        \| \1_{I(\sqrt{t})\setminus I(\tilde \rho_N)}&(s, y)f \|_{L^{p}(\partial\Omega; L^q((0, t)))}\\
        & \lesssim_d t^{-(d+1)/2 + 1/q} \left( \int_{B_{\sqrt t}(x_0)\cap\partial\Omega} (\epsilon + |\nu(y)-\nu(x_0)|)^p \,d\Haus^{d-1}(y) \right)^{1/p} \\
        & \lesssim t^{-(d+1)/2 + 1/q + (d-1)/(2p)} \,\Theta[B_{\sqrt t}(x_0)\cap\partial\Omega]^{1/p} (\epsilon + \bar\nu_p(x_0,\sqrt t)) \,.
    \end{align*}
    This proves the claimed bound also in this case.

    \medskip

    \noindent\emph{Step 4.} Finally, we estimate $\check f$. By construction,
\begin{align*}
	\check f(s, y) = \1_{I(\rho_M)^c} f(s, y) \,.
\end{align*}
Since $\rho_M \geq \frac{r}{2}$ we have
\begin{align*}
	((0, t)\times \partial \Omega) \cap I( \rho_M)^c &= (0, t) \times (\partial\Omega \setminus B_{\rho_M}(x_0)) \\
	& \subseteq (0, t) \times (\partial\Omega \setminus B_{r/2}(x_0))\,.
\end{align*}
We also recall that, by construction, $f$ is supported in $B_{2r}(x_0)$, and so $\supp \check f \subseteq (0, \infty)\times B_{2r}(x_0)\setminus B_{r/2}(x_0)$.

If $y \in \partial \Omega \cap B_{r}(x_0)^c$, then, by \eqref{eq: bound bdry data faraway}, we have
\begin{equation}\label{eq: checkf faraway}
    |\check f(s, y)| \lesssim_d s^{-d/2}(rs^{-1}+r^{-1})e^{-\frac{r^2}{16s}}\,.
\end{equation}

If instead $y \in \partial \Omega \cap (B_{r}(x_0)\setminus B_{r/2}(x_0))$, then, by Lemma~\ref{lem: distance to boundary in good set},
\begin{equation*}
 	\min\{|x-y|^2,|x^*-y|^2\} \geq \frac{1-2\epsilon}{4}r^2 \,,
\end{equation*}
for all $y \in (\partial\Omega \cap B_r(x_0))\setminus B_{r/2}(x_0)$. Therefore, for $\epsilon \leq 1/4$ and $y \in \partial \Omega \cap (B_{r}(x_0)\setminus B_{r/2}(x_0))$ we can estimate
\begin{equation}\label{eq: checkf near}
\begin{aligned}
	|\check f(s, y)| &\lesssim_d s^{-1-\frac{d}{2}}\Bigl(|y-x|e^{- \frac{|x-y|^2}{4s}}+|y-x^*|e^{- \frac{|x^*-y|^2}{4s}}\Bigr)\\
	&\lesssim_d rs^{-1-d/2} e^{-\frac{r^2}{32s}}\,.
\end{aligned}
\end{equation}

Combining \eqref{eq: checkf faraway} and \eqref{eq: checkf near} and the fact that $\sup_{\tau>0} \tau^\alpha e^{-1/\tau}\lesssim_\alpha 1$ for any $\alpha\leq 0$ yields the claimed pointwise bound on $\check f$.    
\end{proof}


\subsection{Completing the proof of Proposition \ref{prop: Estimate close the good boundary}}

We now combine the bounds on the pieces in the decomposition of $f$ proved in the previous subsection with the estimates for the Neumann boundary value problem established in Section \ref{sec: pde theorems}.

\begin{proof}[Proof of Proposition~\ref{prop: Estimate close the good boundary}]
	Fix an $(\epsilon,r)$-good point $x_0\in\partial\Omega$ and let $x\in\Gamma_r(x_0,\epsilon)$ and $t\in(0,r^2/2]\cap(0,\kappa R^2]$. Let $u, f, g$ be defined as after the statement of Proposition \ref{prop: Estimate close the good boundary}. Our first aim is to prove that for any $p>d-1$ we have 
	\begin{equation}\label{eq: final bound on u}
    \begin{aligned}
		t^{d/2} \, |u(t, x)| &\lesssim_{d,\delta, M, p}\|k_\Omega\|_{G^\infty_{\delta,x}((0, 3t/2]\times B_{2r}(x_0))}\Theta[\partial\Omega \cap B_{2r}(x_0)]\\
        &\qquad \times\Bigl((\epsilon +\bar \nu_p(x_0, r))e^{- \frac{d_\Omega(x)^2}{c(1+\delta)t}}+ e^{- \frac{r^2}{c(1+\delta)t}}\Bigr)\,.
    \end{aligned}
	\end{equation}

    We decompose $f$ as in \eqref{eq:fdecomp} and we define $\tilde u_j, u_k, u^\sharp,$ and $\check u$ as the solutions to the boundary value problem~\eqref{eq: Neumann bvp} with the corresponding boundary datum $\tilde f_j, f_k, f^\sharp$, and $\check f$, respectively. Let $u^\flat$ be defined as the solution of 
    \begin{equation*}
        \begin{cases}
            (\partial_s-\Delta) u^\flat (s, y) = g(s, y) & \mbox{for } (s, y)\in (0, \infty)\times \Omega\,,\\
            u^{\flat}(0, y)=0  & \mbox{for } y\in \Omega\,,\\
            \frac{\partial}{\partial \nu_y} u^{\flat}(s, y) = 0 & \mbox{for } (s, y)\in (0, \infty)\times \partial\Omega\,.
        \end{cases}
    \end{equation*}
    By linearity of the heat equation, we have
	\begin{equation*}
		u(t, x) = \sum_{1 \leq j \leq N}\tilde u_j(t, x) + \sum_{1\leq k \leq M} u_k(t, x) + u^\sharp(t, x) + \check u(t, x) + u^\flat(t, x)\,.
	\end{equation*}
	We estimate each of these terms separately using Theorem~\ref{thm: neumann bvp Lp estimates}, Lemma~\ref{lem: a priori bound zero average}, and the bounds on $\tilde f_j, f_k, f^\sharp$, and $\check f$ established in Lemma \ref{lem:fpiecesbound}.

\medskip
{\noindent \it Step 1. Estimating the contribution from $\tilde u_j$.} By construction, each $\tilde f_j$ has integral zero and is supported in $I(\tilde \rho_j)$ with $\tilde \rho_j = 2^j d_\Omega(x)$ and $4\tilde \rho_j^2 \leq t $ for $1\leq j\leq N$. Our aim is to apply Lemma~\ref{lem: a priori bound zero average} with $\rho, R, \rho_0$ in that lemma chosen as $\tilde \rho_j, \sqrt{t/\kappa},$ and $c_{\kappa, M}\sqrt{t}$, respectively, with $c_{\kappa, M}$ to be specified. 

To apply the lemma, we need to verify that
\begin{equation*}
    \sup_{y \in B_{2c_{\kappa,M}\sqrt{t}}(0)\setminus \{y_d>\varphi(y')\}}\bigl(|y|^2+ 4 \varphi(y')(\varphi(y')-y_d)\bigr)<t/\kappa\,.
\end{equation*}
Since $x_0 \in \partial\Omega$ and $\varphi$ is continuous we deduce that $\varphi(0)=0$. Therefore, since $|\nabla \varphi|\leq M$ it holds that
\begin{equation*}
    |\varphi(y')|\leq M|y'| \quad \mbox{for all }y'\in \R^{d-1}\,.
\end{equation*}
Thus, by the Cauchy--Schwarz inequality,
\begin{align*}
    \sup_{y \in B_{2c_{\kappa, M}\sqrt{t}}(0)\setminus \{y_d>\varphi(y')\}}&\bigl(|y|^2+4\varphi(y')(\varphi(y')-y_d)\bigr)\\
    &\leq
        \sup_{y \in B_{2c_{\kappa, M}\sqrt{t}}(0)\setminus \{y_d>\varphi(y')\}} \left( |y|^2+4|\varphi(y')|^2+4|y_d||\varphi(y')| \right)\\
        &\leq
        \sup_{y \in B_{2c_{\kappa, M}\sqrt{t}}(0)\setminus \{y_d>\varphi(y')\}} \left( |y|^2+4M^2|y'|^2+4M|y_d||y'| \right)\\
        &\leq
        \sup_{y \in B_{2c_{\kappa, M}\sqrt{t}}(0)\setminus \{y_d>\varphi(y')\}} \left(|y|^2+4M^2|y'|^2+2M|y|^2\right)\\
    &\leq 
        4c_{\kappa, M}^2(1+4M^2+2M)t\,.
\end{align*}
Therefore, if we choose $c_{\kappa,M}= \frac{1}{4\sqrt{\kappa(1+4M^2+2M)}}< \frac{1}{2\sqrt{\kappa(1+4M^2+2M)}}$ we can apply Lemma~\ref{lem: a priori bound zero average} with the parameters specified above. This yields
\begin{align*}
	|\tilde u_j(t, x)| &\lesssim_{d,\delta,M} \|\tilde f_j\|_{L^1} t^{-d/2}\|k_\Omega\|_{G^\infty_{\delta, x}((t/2, 3t/2)\times B_{\sqrt{t}}(x_0))}e^{-\frac{d_\Omega(x)^2}{12(1+\delta)t}}\min\{1, \tilde\rho_j/(c_{\kappa, M}\sqrt{t})\}^{\alpha} \\
    &\lesssim_{d,\delta,M,\kappa} \|\tilde f_j\|_{L^1} t^{-d/2}\|k_\Omega\|_{G^\infty_{\delta, x}((t/2, 3t/2)\times B_{\sqrt{t}}(x_0))}e^{-\frac{d_\Omega(x)^2}{12(1+\delta)t}}(2^jd_\Omega(x)/\sqrt{t})^{\alpha}\,,
\end{align*}
with $\alpha$ depending only on $d, \delta, M$. Note that the assumption of Lemma \ref{lem: a priori bound zero average}, that is, $t\geq 4\max\{ c_{\kappa,M}^2 t, \tilde\rho_j^2\}$, is satisfied, since $4 c_{\kappa,M}^2\leq 1$ and $4\tilde\rho_j^2\leq t$.

    To bound $\|\tilde f_j\|_{L^1}$, we use Lemma \ref{lem:fpiecesbound} and obtain
    $$
    \|\tilde f_j\|_{L^1} \lesssim_d \Theta[\partial\Omega\cap B_{\tilde\rho_j}(x_0)] \, (\epsilon+\bar{\nu}(x_0,\tilde\rho_j)) 
    \leq \Theta[\partial\Omega\cap B_{\sqrt{t}}(x_0)] \, (\epsilon + \bar \nu_1(x_0,\sqrt t)) \,.
    $$
    
Using $\tilde \rho_N = 2^Nd_\Omega(x)\leq \sqrt{t}$, we can thus estimate
\begin{equation}\label{eq: bound tilde uj}
\begin{aligned}
	\Bigl|\sum_{1\leq j\leq N}\tilde u_j(t, x)\Bigr| 
	&\lesssim_{d,\delta,M,\kappa} \Theta[\partial\Omega \cap B_{\sqrt{t}}(x_0)](\epsilon + \bar \nu(x_0, \sqrt{t}))t^{-d/2}\\
    &\quad \times \|k_\Omega\|_{G^\infty_{\delta, x}((t/2, 3t/2)\times B_{\sqrt{t}}(x_0))}e^{-\frac{d_\Omega(x)^2}{12(1+\delta)t}}\sum_{1\leq j \leq N} (2^jd_\Omega(x)/\sqrt{t})^\alpha\\
    &\lesssim_{d,\delta,M,\kappa} \Theta[\partial\Omega \cap B_{\sqrt{t}}(x_0)](\epsilon + \bar \nu_1(x_0, \sqrt{t}))t^{-d/2}\\
    &\quad \times \|k_\Omega\|_{G^\infty_{\delta, x}((t/2, 3t/2)\times B_{\sqrt{t}}(x_0))}e^{-\frac{d_\Omega(x)^2}{12(1+\delta)t}} (2^Nd_\Omega(x)/\sqrt{t})^\alpha\\
    &\lesssim_{d,\delta,M,\kappa} \Theta[\partial\Omega \cap B_{\sqrt{t}}(x_0)](\epsilon + \bar \nu_1(x_0, \sqrt{t}))t^{-d/2}\\
    &\quad \times \|k_\Omega\|_{G^\infty_{\delta, x}((t/2, 3t/2)\times B_{\sqrt{t}}(x_0))}e^{-\frac{d_\Omega(x)^2}{12(1+\delta)t}}\,. 
\end{aligned}
\end{equation}
This proves the desired bound for the contribution of the $\tilde u_j$'s to $u$.

\medskip

{\noindent \it Step 2. Estimating the contribution from $u_k$.} 
We will apply Theorem~\ref{thm: neumann bvp Lp estimates}. To do so, fix $p\in(d-1,\infty)$ and $q\in(2p/(p-d+1),\infty)$ and note that this choice satisfies the assumptions of the theorem. We also note that $f_k$ is supported in $[0,\infty)\times\omega$ with $\omega = \partial\Omega\cap B_{\rho_k}(x_0)$. It follows that
\begin{align*}
    |u_k(t,x)| & \lesssim_{d,p,q,\delta} \|k_\Omega\|_{G^\infty_{\delta, x}((0, t]\times (B_{\rho_k}(x_0)\cap \partial\Omega))} e^{-\frac{d_\Omega(x)^2}{8(1+\delta)t}} \|f_k \|_{L^{p}(\partial\Omega; L^q((0, t)))} \\
    & \qquad \times \Theta[\partial\Omega\cap B_{\rho_k}(x_0)]^{\frac2{d-1}(\frac d2 - \frac1{q'})} \Haus^{d-1}(\partial\Omega\cap B_{\rho_k}(x_0))^{\frac1{p'} - \frac2{d-1}(\frac d2 - \frac1{q'})} \,.
\end{align*}
Using the bound from Lemma \ref{lem:fpiecesbound} for $\|f_k \|_{L^{p}(\partial\Omega; L^q((0, t)))}$, as well as the bound $\Haus^{d-1}(\partial\Omega\cap B_{\rho_k}(x_0)) \leq \Theta[\partial\Omega\cap B_{\rho_k}(x_0)] \rho_k^{d-1}$ together with the fact that $\frac1{p'} - \frac2{d-1}(\frac d2 - \frac1{q'})> 0$ for $p,q$ in the allowed range, we obtain
\begin{align*}
    |u_k(t,x)| & \lesssim_{d,p,q,\delta} \|k_\Omega\|_{G^\infty_{\delta, x}((0, t]\times (B_{\rho_k}(x_0)\cap \partial\Omega))} e^{-\frac{d_\Omega(x)^2}{8(1+\delta)t}} \, (\epsilon+\bar\nu_p(x_0,\rho_k)) \\
    & \qquad \times \rho_k^{-d} \Theta[\partial\Omega\cap B_{\rho_k}(x_0)] \,. 
\end{align*}

We can thus estimate, using $\rho_k\leq r$,
\begin{equation}\label{eq: bound uk}
\begin{aligned}
	\Bigl|\sum_{1\leq k\leq M}u_k(t, x)\Bigr| 
    &\lesssim_{d, p,\delta} 
    \|k_\Omega\|_{G^\infty_{\delta, x}((0, t]\times B_{r}(x_0)\cap \partial\Omega)}e^{-\frac{d_\Omega(x)^2}{8(1+\delta)t}} (\epsilon +\bar \nu_p(x_0, r))\\
    &\qquad \times \Theta[\partial\Omega\cap B_{r}(x_0)] \, \sum_{1\leq k\leq M}  \rho_k^{-d}\\
    &\lesssim_{d,p,\delta} 
    \|k_\Omega\|_{G^\infty_{\delta, x}((0, t]\times B_{r}(x_0)\cap \partial\Omega)}e^{-\frac{d_\Omega(x)^2}{8(1+\delta)t}} (\epsilon +\bar \nu_p(x_0, r)) \\
    &\qquad \times \Theta[\partial\Omega\cap B_{r}(x_0)] \, t^{-d/2}\sum_{1\leq k\leq M}2^{-dk}\\
    &\lesssim_{d,p,\delta} 
    \|k_\Omega\|_{G^\infty_{\delta, x}((0, t]\times B_{r}(x_0)\cap \partial\Omega)}e^{-\frac{d_\Omega(x)^2}{8(1+\delta)t}} (\epsilon +\bar \nu_p(x_0, r)) \\
    &\qquad \times \Theta[\partial\Omega\cap B_{r}(x_0)] \, t^{-d/2}\,.
\end{aligned}
\end{equation}
This proves the desired bound for the contribution from the $u_k$'s.

\medskip

{\noindent \it Step 3. Estimating the contribution from $u^\sharp$.} 
The function $f^\sharp$ is supported in $(0,\infty)\times B_{\sqrt t}(x_0)$. Applying Theorem \ref{thm: neumann bvp Lp estimates} with $p$ and $q$ as in Step 2, we obtain
\begin{align*}
    |u^\sharp(t, x)| & \lesssim_{d,\delta,p,q} \|k_\Omega\|_{G^\infty_{\delta,x}((0, t]\times B_{\sqrt{t}}(x_0)\cap \partial\Omega)} e^{-\frac{d_\Omega(x)^2}{8(1+\delta)t}} \|f^\sharp\|_{L^p(\partial\Omega;L^q((0,t)))} \\
    & \quad \times \Theta[\partial\Omega\cap B_{\sqrt t}(x_0)]^{\frac2{d-1}(\frac d2 - \frac1{q'})} \Haus^{d-1}(\partial\Omega\cap B_{\sqrt t}(x_0))^{\frac1{p'} - \frac2{d-1}(\frac d2 - \frac1{q'})} \,.
\end{align*}
According to Lemma \ref{lem:fpiecesbound} we have
$$
\|f^\sharp\|_{L^{p}(\partial\Omega; L^q((0, t)))} 
\lesssim_d
            \Theta[\partial\Omega\cap B_{\sqrt t}(x_0)]^{1/p} (\epsilon +\bar \nu_p(x_0, \sqrt t)) \, t^{1/q-(d-1)/(2p')} \,.
$$
Arguing similarly as in Step 2, we arrive at
\begin{equation}\label{eq: bound 1 u sharp}
\begin{aligned}
	|u^\sharp(t, x)| &\lesssim_{d,p,\delta} \|k_\Omega\|_{G^\infty_{\delta,x}((0, t]\times B_{\sqrt{t}}(x_0)\cap \partial\Omega)}(\epsilon+\bar \nu_p(x_0, \sqrt{t}))\Theta[\partial\Omega \cap B_{\sqrt{t}}(x_0)]  t^{-d/2}e^{-\frac{d_\Omega(x)^2}{8(1+\delta)t}}\,.
\end{aligned}
\end{equation}

\medskip
{\noindent \it Step 4. Estimating $\check u$.} We shall apply Theorem~\ref{thm: neumann bvp Lp estimates} with $p=\infty$ and any fixed $q\in (2, \infty)$. Since $\check f$ is supported in $(0,\infty)\times (B_{2r}(x_0)\setminus B_{r/2}(x_0))$, we obtain
\begin{align*}
    |\check u(t, x)| &\lesssim_{d,\delta,q} \|k_\Omega\|_{G^\infty_{\delta, x}((0, t]\times \partial\Omega \cap B_{2r}(x_0))}e^{-\frac{\dist(x, \partial\Omega \setminus B_{r/2}(x_0))^2}{8(1+\delta)t}} \|\check f\|_{L^{\infty}(\partial\Omega\cap B_{2r}(x_0); L^q((0, t))}\\
    &\qquad \times\Theta[\partial\Omega \cap B_{2r}(x_0)] \, t^{1/q'-1/2} \,.    
\end{align*}
By the pointwise bound in Lemma \ref{lem:fpiecesbound}, we have
$$
\|\check f\|_{L^{\infty}(\partial\Omega\cap B_{2r}(x_0); L^q((0, t))}
\lesssim_{d,q} r^{-d-1} t^{1/q} \,.
$$
Note that for any $y \in \partial\Omega \setminus B_{r}(x_0)$ we have $|x-y| \geq r/2$ since $x\in B_{r/2}(x_0)$. From Lemma \ref{lem: distance to boundary in good set} it follows that for any $y \in \partial\Omega \cap B_{r}(x_0)\setminus B_{r/2}(x_0)$ we have
\begin{align*}
    |x-y| \gtrsim |y-x_0| \geq r/2\,.
\end{align*}
Consequently, it holds that $\dist(x, \partial\Omega \setminus B_{r/2}(x_0))\gtrsim r$. When combined with the assumption $t\leq r^2/2$, we arrive at the bound
\begin{equation}\label{eq: bound check u}
\begin{aligned}
	|\check u(t, x)|
    &\lesssim_{d,\delta,q} \|k_\Omega\|_{G^\infty_{\delta, x}((0, t]\times \partial\Omega \cap B_{2r}(x_0))}e^{-\frac{r^2}{c(1+\delta)t}}\Theta[\partial\Omega \cap B_{2r}(x_0)] \, t^{-d/2}\,.
\end{aligned}
\end{equation}

\medskip

{\noindent \it Step 5.} (estimating $u^\flat$) By Duhamel's principle,
\begin{equation*}
    u^\flat(t, x) = \int_0^t \int_\Omega k_\Omega(t-s, x, y)g(s, y)\,dy\,ds\,.
\end{equation*}
From the definition \eqref{eq:defg} and the choice of the cut-off function $\chi$ it follows that
\begin{equation}\label{eq: bound interior data}
    |g(s, y)|\lesssim (r^{-2}s^{-d/2}+s^{-1-d/2})e^{-\frac{r^2}{16s}} \1_{B_{2r}(x_0)\setminus B_r(x_0)}(y) \quad \mbox{for }(s, y) \in (0, \infty)\times \R^d\,.
\end{equation}
By \eqref{eq: bound interior data} and the assumed bound on $k_\Omega$, we find that, since $x \in \Omega \cap B_{r/2}(x_0)$ and $t\in (0, r^2/2]$,
\begin{equation}\label{eq: bound u flat 1}
\begin{aligned}
    |u^\flat(t, x)| &\lesssim_d \|k_\Omega\|_{G^\infty_{\delta, x}((0, t]\times B_{2r}(x_0))}\\
    &\qquad \times \int_0^t \int_{\Omega\cap B_{2r}(x_0)\setminus B_r(x_0)} \frac{s r^{-2}+1}{s^{1+d/2}(t-s)^{d/2}}e^{-\frac{|x-y|^2}{4(1+\delta)(t-s)}-\frac{r^2}{16s}}\,dy\,ds\\
    &\leq 
    r^d\|k_\Omega\|_{G^\infty_{\delta, x}((0, t]\times B_{2r}(x_0))}\\
    &\qquad \times\int_0^t  \frac{1}{s^{1+d/2}(t-s)^{d/2}}e^{-\frac{r^2}{8(1+\delta)(t-s)}-\frac{r^2}{16s}}\,ds\\
    &\leq 
    r^de^{-\frac{r^2}{32(1+\delta)t}}\|k_\Omega\|_{G^\infty_{\delta, x}((0, t]\times B_{2r}(x_0))}\\
    &\qquad \times\int_0^t  \frac{1}{s^{1+d/2}(t-s)^{d/2}}e^{-\frac{r^2}{32(1+\delta)}(\frac{1}{t-s}+\frac{1}{s})}\,ds\,.
\end{aligned}
\end{equation}
The remaining integral we estimate as follows
\begin{equation*}
\begin{aligned}
    \int_0^t  \frac{1}{s^{1+d/2}(t-s)^{d/2}}e^{-\frac{r^2}{32(1+\delta)}(\frac{1}{t-s}+\frac{1}{s})}\,ds&=
    t^{-d}\int_0^1  \frac{1}{\tau^{1+d/2}(1-\tau)^{d/2}}e^{-\frac{r^2}{32(1+\delta)t}(\frac{1}{1-\tau}+\frac{1}{\tau})}\,d\tau\\
    &\lesssim_d
    t^{-d}\int_0^{1/2}  \frac{1}{\tau^{1+d/2}}e^{-\frac{r^2}{32(1+\delta)t\tau}}\,d\tau\\
    &\lesssim_{d,\delta}
    t^{-d/2}r^{-d}\,,
\end{aligned}
\end{equation*}
where the last step used $\sup_{\eta>0}\eta^{1+d/2} e^{-1/\eta} \lesssim_d 1$ and $t\leq r^2/2$. Inserting the last estimate into \eqref{eq: bound u flat 1} yields
\begin{equation}\label{eq: bound u flat 2}
\begin{aligned}
    |u^\flat(t, x)| \lesssim_d 
    t^{-d/2}e^{-\frac{r^2}{32(1+\delta)t}}\|k_\Omega\|_{G^\infty_{\delta, x}((0, t]\times B_{2r}(x_0))}\,.
\end{aligned}
\end{equation}

\medskip

Putting together~\eqref{eq: bound tilde uj}, \eqref{eq: bound uk}, \eqref{eq: bound 1 u sharp}, \eqref{eq: bound check u}, and~\eqref{eq: bound u flat 2}, the fact that $\Theta[\partial\Omega \cap B_{2r}(x_0)]\gtrsim_d 1$, and using that by H\"older's inequality $\bar\nu_1(x_0, s)\leq \bar\nu_p(x_0, s)$ for any $s>0, p\geq 1$, establishes~\eqref{eq: final bound on u}. 

\medskip

Before completing the proof of the proposition, let us note that
\begin{equation}
    \label{eq:distancebdrybound}
    0\leq  e^{- \frac{d_\Omega(x)^2}{t}}- e^{- \frac{|x-x_0|^2}{t}}\leq  4\epsilon e^{- \frac{d_\Omega(x)^2}{2t}}\,.
\end{equation}
Indeed, when $\epsilon d_\Omega(x)^2\leq t$, we use Lemma~\ref{lem: distance to boundary in good set}, a Taylor expansion, and the fact that $\frac{1}{1-2\epsilon} \leq 1+4\epsilon$ for $\epsilon \in (0, \frac{1}{4}]$ to get
\begin{equation*}
	e^{- \frac{d_\Omega(x)^2}{t}} \geq e^{-\frac{|x-x_0|^2}{t}} \geq e^{- \frac{d_\Omega(x)^2}{t}}e^{-\frac{4\epsilon d_\Omega(x)^2}{t}} \geq e^{- \frac{d_\Omega(x)^2}{t}}\Bigl(1- \frac{4\epsilon d_\Omega(x)^2}{t}\Bigr) \geq e^{- \frac{d_\Omega(x)^2}{t}}- 4\epsilon e^{-\frac{d_\Omega(x)^2}{2t}} \,.
\end{equation*}
We also used $0\leq se^{-s/2}\leq 1$ for all $s\geq 0$. If instead $\epsilon d_\Omega(x)^2>t$, then again by Lemma~\ref{lem: distance to boundary in good set} and since $e^{- \frac{1}{2\epsilon}} \leq \epsilon$ for all $\epsilon >0$,
\begin{equation*}
	e^{-\frac{|x-x_0|^2}{t}} \leq e^{- \frac{d_\Omega(x)^2}{t}} \leq e^{-\frac{1}{2\epsilon}}e^{- \frac{d_\Omega(x)^2}{2t}} \leq \epsilon e^{- \frac{d_\Omega(x)^2}{2t}} \,.
\end{equation*}
This proves \eqref{eq:distancebdrybound}.

\medskip

To deduce the bound claimed in Proposition~\ref{prop: Estimate close the good boundary}, we use
\begin{align*}
	k_{\R^d}(t, x, x)+k_{\R^d}(t, x, x^*) &= \frac{1}{(4\pi t)^{d/2}}\Bigl(1+e^{-\frac{|x-x^*|^2}{4t}}\Bigr)
	=
	\frac{1}{(4\pi t)^{d/2}}\Bigl(1+e^{-\frac{|x-x_0|^2}{t}}\Bigr)\,.
\end{align*}
Combining this with \eqref{eq: final bound on u}, \eqref{eq:distancebdrybound} and the fact that $$\Theta[\partial\Omega \cap B_{2r}(x_0)] \gtrsim_d 1\,, \quad \mbox{and} \quad \|k_\Omega\|_{G^\infty_{\delta,x}((0, 3t/2]\times B_{2r}(x_0))}\gtrsim_d 1$$ (the latter follows from that $x \in B_{2r}(x_0)$ and Proposition~\ref{prop: improved bound bulk} as $t\to 0$), this yields
\begin{align*}
	\Bigl|&(4\pi t)^{d/2}k_\Omega(t, x, x) - 1 - e^{-\frac{d_\Omega(x)^2}{t}}\Bigr|\\
    &\lesssim_{d,\delta,M, p} \|k_\Omega\|_{G^\infty_{\delta, x}((0, 3t/2]\times B_{2r}(x_0))}\Theta[\partial\Omega \cap B_{2r}(x_0)]\Bigl((\epsilon +\bar \nu_p(x_0, r))e^{- \frac{d_\Omega(x)^2}{c(1+\delta)t}}+ e^{- \frac{r^2}{c(1+\delta)t}}\Bigr)
\end{align*}
for some universal $c>0$. This completes the proof of Proposition~\ref{prop: Estimate close the good boundary}.
\end{proof}


\section{Constructing a good subset of the boundary}\label{sec:constructiongood}

In the previous section we have proved pointwise bounds for the heat kernel close to good points on the boundary. In the present section we quantify that `most' points of the boundary are good points and that, consequently, the heat kernel bounds from the previous section hold in almost all of the boundary region $d_\Omega\lesssim \sqrt{t}$. 

For bounded sets with Lipschitz-regular boundary Brown argues in \cite{Brown93} that `most' points of the boundary are good points. However, for our proof of Theorem~\ref{thm: main thm convex}, when $\Omega$ is convex, we need to obtain quantitative estimates with more explicit dependence on the geometry. This is the topic of Subsection \ref{sec:goodconvex}. Before that, we review Brown's proof in the Lipschitz case and provide some additional details in Subsection \ref{sec:goodlipschitz}.


\subsection{Good sets that are large}\label{sec:goodlipschitz}

We recall that the notion of an $(\epsilon,r)$-good point $x\in\partial\Omega$ was introduced in Definition \ref{def: good point} and that the corresponding set $\Gamma_r(x,\epsilon)$ appeared in Definition \ref{def: gammarxeps}. 

\begin{definition}
    A subset $G \subset \partial\Omega$ is said to be $(\epsilon, r)$-good if each point $x\in G$ is an $(\epsilon, r)$-good point.
\end{definition}

\begin{definition}
    If $G\subset \partial\Omega$ is an $(\epsilon, r)$-good set we define the associated sawtooth region
    \begin{equation*}
		\mathcal{G} := \bigcup_{x\in G} \Gamma_r(x, \epsilon)\,.
    \end{equation*}
\end{definition}
Note that Proposition \ref{prop: Estimate close the good boundary} provides estimates for $k_\Omega(t, x, x)$ precisely for $x \in \mathcal{G}$.

For our purposes we shall need that there exists an $(\epsilon, r)$-good set $G$ with two properties:
\begin{enumerate}
    \item[(i)] the corresponding sawtooth region $\mathcal G$ covers a large portion of the boundary region $\{x\in\Omega:\ d_\Omega(x)<s\}$ for $s$ sufficiently small, and
    \item[(ii)] we can control the local oscillation of the normal $\bar \nu_p(x, s)$ for $x\in G$ so that the error in Proposition \ref{prop: Estimate close the good boundary} is small.
\end{enumerate}
That this is possible, is guaranteed by the following result.

\begin{lemma}\label{lem: existence of big good set}
    Let $\Omega \subset \R^d$, $d\geq 2$, be an open and bounded Lipschitz set. There exists a collection of $(\epsilon, r)$-good sets $\{G_{\epsilon, r}\}_{\epsilon, r}$ such that, for every $\epsilon\in(0,1/4]$ and $p\in [1, \infty)$, we have
    \begin{equation}\label{eq: assumption variation of normal in good set}
	   \lim_{r\to 0^{\limplus}}\sup_{x\in G_{\epsilon, r}}\bar\nu_p(x, r) =0\,,
    \end{equation}
    and the associated sawtooth region $\mathcal G_{\epsilon,r}$ satisfies, for every $\epsilon\in(0,1/4]$,
    \begin{equation}\label{eq: assumption good set is big}
        \lim_{r\to 0^\limplus}\limsup_{s\to 0^\limplus}\frac{|\{x\in \Omega: d_\Omega(x)<s\}\setminus \mathcal{G}_{\epsilon, r}|}{|\{x\in \Omega: d_\Omega(x)<s\}|}=0 \,.
    \end{equation}
\end{lemma}

As $\Omega$ is bounded and Lipschitz, we have $|\{x\in \Omega: d_\Omega(x)<s\}| = s\Haus^{d-1}(\partial\Omega)(1+o(1))$ as $s\to 0^\limplus$ (see \eqref{eq:thetadiscussion}), thus \eqref{eq: assumption good set is big} is equivalent to the corresponding statement with the denominator replaced by $s\Haus^{d-1}(\partial\Omega)$. 

\begin{proof}
    \emph{Step 1.}
    We begin by proving that if $G\subset \partial\Omega$ is $(\epsilon, r)$-good, then
    \begin{equation}\label{eq: sawtooth region has large measure}
        \limsup_{s\to 0^\limplus}\frac{|\{x\in \Omega: d_\Omega(x)<s\}\setminus \mathcal{G}|}{|\{x\in \Omega: d_\Omega(x)<s\}|}\leq \frac{\Haus^{d-1}(\partial\Omega \setminus G)}{\Haus^{d-1}(\partial\Omega)}\,.
    \end{equation}
    The argument follows closely the proof of \cite[Lemma 3.5]{FrankLarson_Crelle20} but for a one-sided neighborhood of $\partial\Omega$ instead of a two-sided (see also \cite[Proposition 1.3]{Brown93} and \cite[Lemma 2.7]{BanuelosEtAl_09}).
    Write
  \begin{equation}\label{eq:GoodSetIsBig_proof1}
  \begin{aligned}
    |\{x\in \Omega: d_\Omega(x)<s\}\setminus \mathcal{G}| 
    &= |\{x\in \Omega: d_\Omega(x)<s\}|\\
    &\quad-|\{x\in \Omega: d_\Omega(x)<s\}\cap \mathcal{G}|\, .
  \end{aligned}
  \end{equation}

  To analyze the second term, for some $\delta \in (0, 1), \epsilon'>0$ to be determined later, we choose $\nu_1, \ldots, \nu_N\in \S^{d-1}$ and disjoint closed sets $F_1, \ldots, F_N \subset G$ such that $\Haus^{d-1}(\bigcup_{i=1}^N F_i)\geq (1-\delta)\Haus^{d-1}(G)$ and $|\nu(x)-\nu_i|\leq \eps'$ for all $x\in F_i$. Such sets can be constructed as follows. By Lusin's theorem \cite[Theorem 7.10]{FollandRealAnalysis}, there exists a closed set $\Sigma\subset G$ so that $\Haus^{d-1}(\Sigma)\geq (1-\delta/2)\Haus^{d-1}(G)$ and $\nu|_{\Sigma}$ is continuous. The sets $F(\nu) := \{x \in \Sigma: |\nu(x)-\nu|< \epsilon'\}$ form an open cover of $\Sigma$ indexed by $\nu \in \S^{d-1}$. By compactness of $\Sigma$, there exists $\nu_1, \ldots, \nu_N \subset \S^{d-1}$ so that $\{F(\nu_i)\}_{i=1}^N$ is an open cover of $\Sigma$. Define $F_1:=\overline{F(\nu_1)}$. We now proceed inductively to construct disjoint closed and disjoint sets $F_j$ for $j =2, \ldots, N$. Given $F_1, \ldots, F_{j-1}$ we construct $F_j$ as follows. Since $\Haus^{d-1}|_\Sigma$ is a regular measure, we can find a compact set $F_j \subset F(\nu_j) \setminus (\cup_{i<j}F_i)$ such that $\Haus^{d-1}(F_j)\geq (1-\delta/2)\Haus^{d-1}(F(\nu_j) \setminus (\cup_{i<j}F_i))$. Note that it follows from the construction that $\dist(F_i, F_j)>0$ for each $i<j$ and
  \begin{align*}
    \Haus^{d-1}(\cup_{i=1}^N F_i)
    &=
    \sum_{i=1}^N\Haus^{d-1}(F_i)\\
    &\geq 
    (1-\delta/2)\Haus^{d-1}(\Sigma)\\
    &\geq 
    (1-\delta/2)^2\Haus^{d-1}(G)\,.
  \end{align*}
  Since $(1-\delta/2)^2= 1-\delta+\delta^2/4>1-\delta$ the collection $\{F_i\}_{i=1}^N$ satisfies the desired properties.

  Assume that $s_0 \leq r/2$ and $s_0\leq \dist(F_i, F_j)/2$ for all $i \neq j$. 
  We define a map $(\cup_{i=1}^N F_i )\times (0, s_0)\to \R^d$ by $(x, s)\mapsto x-s \nu_i$ if $x \in F_i$. We claim that if $\epsilon'$ is chosen small enough then this map is injective and its range is a subset of $\mathcal{G}$. To observe that the range is a subset of $\mathcal{G}$, we argue that $x-s\nu_i \in \Gamma_{\epsilon, r}(x)$. Indeed, if $x \in F_i$ then, since $|\nu(x)-\nu_i|<\epsilon'$, we find that
  \begin{align*}
    ((x-\nu_i s)-x)\cdot \nu(x) &= -s\nu_i \cdot \nu(x) \\
    &= -s-s(\nu_i-\nu(x)) \cdot \nu(x)\\
    &\leq -s+s\epsilon'\\
    &= -(1-\epsilon')|(x-s\nu_i)-x|\,.
  \end{align*}
  If $\epsilon'\leq 1-\sqrt{1-\epsilon^2}$ this proves that $x-\nu_is \in \Gamma_{\epsilon, r}(x)$ and thus the image of our map is contained in $\mathcal{G}$. It remains to show that the map is injective. Since $|\nu_i|=1$ and $s_0< \dist(F_i, F_j)/2$, it is clear that the images of $F_i \times (0, s_0)$ for different $i$ are disjoint. Aiming for contradiction, suppose that there exists $x, y \in F_i, s,s'\in (0, s_0)$ such that $x-s\nu_i = y-s'\nu_i$. Since $s_0<r/2$ we find that $|x-y|=|s-s'|<s_0<r/2$. Thus, since $x, y$ are $(\epsilon, r)$-good, we deduce
  \begin{align*}
    |s-s'|& = |(s-s')\nu_i\cdot \nu_i|\\
    &= |(x-y)\cdot \nu_i|\\
    &\leq
    |(x-y)\cdot \nu(x)|+|(x-y)\cdot(\nu_i-\nu(x))|\\
    &\leq \epsilon |x-y|+\epsilon' |x-y|\\
    &= (\epsilon+\epsilon')|s-s'|\,.
  \end{align*}
  This is a contradiction provided that $\epsilon +\epsilon'<1$. Thus we conclude that the constructed mapping is injective if $\epsilon' < \min\{1-\sqrt{1-\epsilon^2}, 1-\epsilon\}$.

  The area formula~\cite[Theorem~3.2.3]{Federer_book} now implies that, for $0<s< s_0$, 
  \begin{align*}
    |\{x\in \Omega: d_\Omega(x)<s\}\cap \mathcal{G}| &\geq \sum_{i=1}^N |\{x-\rho\nu_i: x\in F_i, 0<\rho <s\}|\\
    &\geq 
   (1-(\eps')^2/2)\sum_{i=1}^N \int_{\{x-\rho\nu_i: x\in F_i, 0<\rho<s\}} \frac{dy}{\nu_i\cdot\nu(x)} \\
    & = s(1-(\eps')^2/2)\sum_{i=1}^N \Haus^{d-1}(F_i)\\
    &\geq
    s(1-(\eps')^2/2)(1-\delta)\Haus^{d-1}(G)\, .
  \end{align*}

    Choosing $\epsilon' =\sqrt{2\delta}$ for $\delta$ sufficiently small depending on $\epsilon$, inserting the obtained bound into \eqref{eq:GoodSetIsBig_proof1}, and dividing by $|\{x\in \Omega: d_\Omega(x)<s\}|$, we conclude that for $s$ sufficiently small (depending on $\epsilon, \delta, r$)
  \begin{equation*}
    \frac{|\{x\in \R^d: d_\Omega(x)<s\}\setminus \mathcal{G}|}{|\{x\in \R^d: d_\Omega(x)<s\}|}
    \leq
    1-(1-\delta)^2\frac{s\Haus^{d-1}(\partial\Omega)}{|\{x\in \R^d: d_\Omega(x)<s\}|}\frac{\Haus^{d-1}(G)}{\Haus^{d-1}(\partial\Omega)}\,.
  \end{equation*} 
  As we have already remarked in \eqref{eq:thetadiscussion}, since $\Omega$ is bounded and Lipschitz, it holds that
  \begin{equation*}
    \lim_{s\to 0^\limplus}\frac{s\Haus^{d-1}(\partial\Omega)}{|\{x\in \R^d: d_\Omega(x)<s\}|} = 1
  \end{equation*}
  and thus
  \begin{align*}
    \limsup_{s\to 0^\limplus}\frac{|\{x\in \R^d: d_\Omega(x)<s\}\setminus \mathcal{G}|}{|\{x\in \R^d: d_\Omega(x)<s\}|}
    &\leq
    1-(1-\delta)^2\frac{\Haus^{d-1}(G)}{\Haus^{d-1}(\partial\Omega)}\,.
  \end{align*}

  The bound in \eqref{eq: sawtooth region has large measure} follows by letting $\delta \to 0$ and using $\Haus^{d-1}(G) = \Haus^{d-1}(\partial\Omega) - \Haus^{d-1}(\partial\Omega \setminus G)$.

    \medskip

    \noindent\emph{Step 2.}
    Fix $\epsilon\in(0,\frac12]$. We show that for every $r>0$ there is an $(\epsilon,r)$-good set $G_{\epsilon,r}\subset\partial\Omega$ such that
    \begin{equation*}
        \lim_{r\to 0^\limplus} \frac{\Haus^{d-1}(\partial\Omega \setminus G_{\epsilon, r})}{\Haus^{d-1}(\partial\Omega)} =0
    \end{equation*}
    and
    $$
        \lim_{r\to 0^\limplus} \sup_{x\in G_{\epsilon,r}} \bar\nu_1(x,r) = 0 \,.
    $$
    In view of \eqref{eq: sawtooth region has large measure}, the first property implies \eqref{eq: assumption good set is big}, while the second property implies \eqref{eq: assumption variation of normal in good set} for $p=1$. The extension to $p>1$ will be achieved in the next step.

    To construct $G_{\epsilon,r}$, first note that by Rademacher's theorem, for $\Haus^{d-1}$-almost every $x\in \partial\Omega$ there exists a well-defined outwards pointing unit normal. Define $\nu \colon \partial\Omega \to \S^{d-1}$ to be this normal whenever it exists.

Fix $\{r_k\}_{k\geq 1}\subset (0, \infty)$ decreasing with $\lim_{k\to \infty}r_k=0$. Then, by Rademacher's theorem and the Lebesgue differentiation theorem, respectively, the functions $R_k \colon \partial\Omega \to \R$, $\nu_k\colon \partial\Omega \to \R$, defined by
\begin{align*}
    R_k(x)& := \sup_{y\in \partial\Omega \cap B_{r_k}(x)} \frac{|\nu(x)\cdot(y-x)|}{|x-y|}\,,\\
    \nu_k(x) &:= \sup_{s<r_k} \frac{1}{\Haus^{d-1}(\partial\Omega \cap B_{s}(x))}\int_{\partial\Omega \cap B_s(x)}|\nu(x)-\nu(y)|\,d\Haus^{d-1}(y) \,,
\end{align*}
converge to $0$ $\Haus^{d-1}$-almost everywhere on $\partial\Omega$ as $k \to \infty$. By Egorov's theorem, for any $\delta>0$ there exists a set $G_\delta \subset \partial\Omega$ so that $\{R_k\}_k$ and $\{\nu_k\}_k$ converge uniformly to $0$ on $G_\delta$ and $\Haus^{d-1}(\partial\Omega\setminus G_\delta)\leq \delta \Haus^{d-1}(\partial\Omega)$. 

Fix a decreasing sequence $\{\delta_j\}_{j\geq 1}\subset (0, 1)$ with $\delta_j \to 0$. It follows that, for any fixed $\epsilon\in(0,\frac12]$ and any $j$ there is $k_{j,\epsilon}$ such that for all $k \geq k_{j,\epsilon}$ and all $x \in G_{\delta_j}$ we have
\begin{equation*}
    |R_k(x)|< \epsilon \quad \mbox{and}\quad |\nu_k(x)|\leq 1/j \,.
\end{equation*}
Without loss of generality we may assume that $\lim_{j\to \infty}k_{j,\epsilon}=\infty$.
Note that the sequence $\{r_{k_{j,\epsilon}}\}_{j\geq 1}$ tends to zero as $j \to \infty$ and that the bound $|R_k(x)|< \epsilon$ ensures that $G_{\delta_j}$ is an $(\epsilon, r_{k_{j,\epsilon}})$-good set.

Given $\epsilon \in (0, \frac12]$ we for $r\leq\sup_{j\geq 1} r_{k_{j,\epsilon}}$ define $j(r, \epsilon) := \sup\{j:r_{k_{j,\epsilon}}\geq r\}$ and set $G_{\epsilon, r} := G_{\delta_{j(r, \epsilon)}}$. For $r> \sup_{j\geq 1}r_{k_{j,\epsilon}}$ we instead let $G_{\epsilon, r}=\emptyset$. By construction, $G_{\epsilon, r}$ is an $(\epsilon, r)$-good set with
\begin{align*}
    \sup_{x\in G_{\epsilon,r}}\bar \nu_1(x, r)\leq \frac1{j(r, \epsilon)} \qquad \mbox{and}\qquad \frac{\Haus^{d-1}(\partial\Omega \setminus G_{\epsilon, r})}{\Haus^{d-1}(\partial\Omega)}\leq \delta_{j(r, \epsilon)}\,.
\end{align*}
Since we have $\lim_{j\to \infty}\delta_j=0$ and, for any fixed $\epsilon>0$, $j(r, \epsilon)\to \infty$ as $r\to 0$, this shows that the constructed sets have the properties stated at the beginning of this step.

\medskip

\noindent\emph{Step 3.}
To complete the proof, we need to show that \eqref{eq: assumption variation of normal in good set} holds for all $p\in[1,\infty)$. Since, by Step 2, it holds for $p=1$, it suffices to show that, if
\begin{equation*}
    \lim_{r\to 0^\limplus} \sup_{x\in G_{\epsilon, r}}\bar \nu_p(x, r) =0
\end{equation*}
for some $p\in [1, \infty)$, then it holds for all $p \in [1, \infty)$. Indeed, by the H\"older inequality and since $|\nu-\nu'|\leq 2$ for any $\nu, \nu' \in \mathbb{S}^{d-1}$ it holds that for any $1\leq p_1<p_2<\infty$ and all $x_0 \in G_{\epsilon,r}$ and $\eta>0$ that
\begin{align*}
    \biggl(\frac{1}{\Haus^{d-1}(B_{\eta}(x_0)\cap \partial\Omega)}&\int_{B_{\eta}(x_0)\cap \partial\Omega} |\nu(x_0)-\nu(y)|^{p_1}\,d\Haus^{d-1}(y)\biggr)^{1/p_1} \\
    &\hspace{-65pt}\leq
    \biggl(\frac{1}{\Haus^{d-1}(B_{\eta}(x_0)\cap \partial\Omega)}\int_{B_{\eta}(x_0)\cap \partial\Omega} |\nu(x_0)-\nu(y)|^{p_2}\,d\Haus^{d-1}(y) \biggr)^{1/p_2}\\
    &\hspace{-65pt}\leq 
    2^{1-p_1/p_2}\biggl(\frac{1}{\Haus^{d-1}(B_{\eta}(x_0)\cap \partial\Omega)}\int_{B_{\eta}(x_0)\cap \partial\Omega} |\nu(x_0)-\nu(y)|^{p_1}\,d\Haus^{d-1}(y) \biggr)^{1/p_2}.
\end{align*}
Taking the supremum over $0<\eta <r$, yields
\begin{equation*}
    \bar \nu_{p_1}(x_0, r) \leq \bar \nu_{p_2}(x_0, r) \leq 2^{1-p_1/p_2}\bar \nu_{p_1}(x_0, r)^{p_1/p_2}\,,
\end{equation*}
from which the desired conclusion follows immediately. This completes the proof of Lemma \ref{lem: existence of big good set}.
\end{proof}

\subsection{An explicit construction for convex sets}\label{sec:goodconvex}

In this subsection we will refine the results from the previous subsection in the case where the underlying set is convex. Our construction is based, in part, on our earlier work in~\cite{FrankLarson_Crelle20}, but there are some differences in detail and we have tried to present everything in a self-contained manner.

In what follows, when $\Omega$ is convex, we shall work with the following concrete choice of good set.

\begin{definition}
	Let $\Omega \subset \R^d$ be an open convex set. For $\epsilon\in (0, 1/2]$ and $r \in (0, \epsilon r_{\rm in}(\Omega)]$ define
	\begin{equation*}
		G_{\epsilon, r} := \{x \in \partial \Omega: \exists y \in \Omega \mbox{ s.t. }|y-x|=r/\epsilon, B_{r/\epsilon}(y)\subset \Omega\} 
	\end{equation*}
	and let $\mathcal{G}_{\epsilon, r}$ be the corresponding sawtooth region.
\end{definition}

Next, we argue that $G_{\epsilon, r}$ is an $(\epsilon, r)$-good set and that~\eqref{eq: assumption good set is big} and~\eqref{eq: assumption variation of normal in good set} hold. In fact, in the convex case we will be able to deduce quantitative versions of the latter two limiting relations. Specifically, we shall prove the following proposition.

\begin{proposition}\label{prop: G is a good set}
 	Let $\Omega \subset \R^d$ be an open convex set. For $\epsilon\in (0, 1/2]$ and $r \in (0, \epsilon r_{\rm in}(\Omega)]$, the set $G_{\epsilon, r}$ is $(\epsilon, r)$-good. Moreover, if $s \in (0, r/2]$, then
    \begin{equation*}
	   |\{x\in \Omega: d_\Omega(x)<s\}\setminus \mathcal{G}_{\epsilon, r}|\lesssim_d \Haus^{d-1}(\partial\Omega) \, \frac{sr}{\epsilon r_{\rm in}(\Omega)} \,,
    \end{equation*}
    and, if $x \in G_{\epsilon, r}$, then
 	\begin{equation*}
 		\sup_{y \in B_r(x) \cap \partial\Omega}\biggl[\,\sup_{\nu_y \in \{\nu \in \S^{d-1} : (z-y)\cdot \nu < 0\  \forall z \in \Omega\}} |\nu(x)-\nu_y|\biggr] \lesssim_d \epsilon \,.
 	\end{equation*}
 \end{proposition} 

The proof of this proposition relies on two lemmas, which we formulate and prove next.

For $R\in (0, r_{\rm in}(\Omega)]$, define $G^0_R \subset \partial\Omega$ by
\begin{equation}\label{def: G0}
	G^0_R := \{x \in \partial \Omega: \exists y \in \Omega \mbox{ s.t. }|y-x|=R, \ B_R(y)\subset \Omega\}\,.
\end{equation}
Note that for each $y \in \{x\in \Omega: d_\Omega(x) =R\}$ there exists at least one $x \in \partial \Omega$ so that $x, y$ satisfy the conditions of the construction of $G^0_R$. Note that $G_{\epsilon, r}= G^0_{r/\epsilon}.$

\begin{lemma}\label{lem: GR is a good set}
	Let $\Omega \subset \R^d$ be an open convex set and let $G_R^0$ be defined as in~\eqref{def: G0}. If $R>\frac{r}{2\epsilon}$, then $G_R^0$ is an $(\epsilon, r)$-good set.
\end{lemma}

\begin{proof}
Fix $x_0 \in G^0_R$ and let $y\in\Omega$ be such that $|x_0-y|=R$ and $B_R(y) \subset \Omega$. 

We claim that the outward pointing unit normal to $\partial\Omega$ at $x_0$ exists and is equal to $\nu(x_0)=\frac{x_0-y}{|x_0-y|}$. Indeed, any supporting hyperplane to $\Omega$ at $x_0$ must also be a supporting hyperplane for $B_R(y)$ at $x_0$ and since $B_R(y)$ has a unique supporting hyperplane at $x_0$, namely $\{z\in \R^d:(y-x_0)\cdot(z-x_0)=0\}$, the claim follows. Notice that we therefore have $y = x_0 - \nu(x_0)R$.

If $x \in \partial\Omega \cap B_r(x_0)$ then, by construction of $G^0_R$, it follows that $|x-y|\geq R$. Therefore,
\begin{equation*}
	R^2 \leq |x-y|^2 = |x-x_0+\nu(x_0)R|^2 = |x-x_0|^2 + 2(x-x_0)\cdot \nu(x_0)R + R^2 \,,
\end{equation*}
which, since $x\in B_r(x_0)$, implies that
\begin{equation}\label{eq: a priori GR bound}
	(x-x_0)\cdot \nu(x_0) \geq -\frac{|x-x_0|^2}{2R}\geq -\frac{r}{2R} |x-x_0|\,.
\end{equation}
Since $\Omega$ is convex, we have $(x-x_0)\cdot \nu(x_0) \leq0$ whenever $x\in \partial\Omega$. Therefore, $G^0_R$ is an $(\epsilon, r)$-good set provided $R >\frac{r}{2\epsilon}$.
\end{proof}

\begin{lemma}\label{lem: supremum variation of normal}
	Let $\Omega \subset \R^d$ be an open convex set and let $G_R^0$ be defined as in~\eqref{def: G0}. If $x_0\in G_R^0$, $x \in B_s(x_0)\cap \partial\Omega$ with $0<s<R$, and if $\nu\in \S^{d-1}$ satisfies $\Omega \subset \{z\in \R^d: (z-x)\cdot \nu <0\}$, then
	\begin{equation*}
		\nu(x_0)\cdot \nu \geq 1 - c \Bigl(\frac{s}{R}\Bigr)^2 \quad \mbox{and} \quad 
	|\nu(x_0)-\nu|
	\leq 
	\sqrt{2c} \, \frac{s}{R}
	\end{equation*}
	with $c>0$ depending only on $d$.
\end{lemma}

\begin{proof}
Given $x_0\in G_R^0$, let $y\in\Omega$ be such that $|x_0-y|=R$ and $B_R(y)\subset\Omega$. Let $x$ and $\nu$ be as in the lemma.

By construction, it holds that
\begin{equation*}
	\{z \in \R^d: (z-x)\cdot \nu \geq 0\} \subset \Omega^c \subset B_R(y)^c\,.
\end{equation*}
In particular, for any $z$ such that $(z-x)\cdot \nu=0$ we have
\begin{equation*}
	|y-z| \geq R\,.
\end{equation*}
Choosing $z = y+((x-y)\cdot \nu) \nu$ leads to $|(y-x)\cdot \nu| \geq R$. Since $y \in \Omega$ and since $\nu$ is the outward pointing unit normal to a supporting hyperplane through $x$, we have $(x-y)\cdot \nu\geq 0$ and therefore
\begin{equation*}
	(x-y)\cdot \nu \geq R
\end{equation*}
In the previous proof we have shown that $y = x_0 - R \nu(x_0)$, so we have arrived at
\begin{equation*}
	\nu(x_0)\cdot \nu \geq 1-\frac{(x-x_0)\cdot \nu}{R}\,.
\end{equation*}
Setting $\alpha := \nu(x_0)\cdot \nu$ and writing $\nu = \alpha\nu(x_0) + \nu^\perp$ with $\nu^\perp \cdot \nu(x_0)=0$, the above inequality implies that
\begin{equation*}
	\alpha \geq  1- \alpha\frac{(x_0-x)\cdot \nu(x_0)}{R}- \frac{(x_0-x)\cdot\nu^\perp}{R} \geq 1- \frac{|x-x_0|^2}{2R^2}\alpha- \frac{|x-x_0|}{R}\sqrt{1-\alpha^2}\,,
\end{equation*}
where we used $|\nu^\perp| = \sqrt{1-\alpha^2}$ and~\eqref{eq: a priori GR bound}. 

Let us consider the inequality
$$
a \geq 1 - \frac12 \beta^2 a - \beta \sqrt{1-a^2}
$$
for $a\in[0,1]$ and $\beta\in(0,1)$. Since the right-hand side of this inequality is a strictly convex function of $a$ on $[0, 1]$ and the inequality can be seen to be strict for $a$ sufficiently close to $1$ and false for $a$ sufficiently close to zero, it follows that there exists a unique $a^* = a^*(b)$ such that the inequality is valid for all $a \in [a^*, 1]$ and fails for $a \in [0, a^*)$. A computation shows that
\begin{equation*}
	a^*(\beta) = \frac{4+2\beta^2-2\beta^2\sqrt{8+\beta^2}}{4+8\beta^2+\beta^4} \,.
\end{equation*}
By Taylor expansion, there exists a constant $c>0$ so that if $\beta \in (0, 1)$ then
\begin{equation*}
	\alpha^*(\beta) \geq 1- c \beta^2\,.
\end{equation*}

From this discussion we deduce that
\begin{equation*}
	\nu(x_0)\cdot \nu \geq a^*(|x-x_0|/R) \geq 1- c\Bigl(\frac{|x-x_0|}R \Bigr)^2
    \geq 1-c\Bigl(\frac{s}{R}\Bigr)^2\,.
\end{equation*}
This completes the proof of the first inequality.

Using the inequality that we just proved, one finds that
\begin{align*}
	|\nu(x_0)-\nu|
	=
	\sqrt{2-2\nu(x_0)\cdot \nu}\leq 
	\sqrt{2c} \, \frac{s}{R}\,.
\end{align*}
This completes the proof of the lemma.
\end{proof}

\begin{proof}[Proof of Proposition~\ref{prop: G is a good set}]
    The fact that $G_{\epsilon, r} = G_{r/\epsilon}^0$ is an $(\epsilon,r)$-good set follows from Lemma~\ref{lem: GR is a good set}. The bound on $|\{x\in\Omega:\ d_\Omega(x)<s\}\setminus\mathcal G_{\epsilon,r}|$ appears in the proof of~\cite[Lemma 5.4]{FrankLarson_Crelle20}. (It should be noted that the good set constructed in that paper is slightly different from $G_{\epsilon,r}$, but it is contained in $G_{\epsilon, r}$.) The bound on $|\nu(x)-\nu_y|$ follows from Lemma~\ref{lem: supremum variation of normal}.
\end{proof}

As we shall need it below, we record the following consequence of Lemma~\ref{lem: supremum variation of normal} quantifying how $R, \varphi,$ and $\kappa$ appearing in Proposition \ref{prop: Estimate close the good boundary} can be chosen depending on $\Omega, \epsilon, r$ for $x_0$ in the constructed good set.

\begin{lemma}\label{lem: Local parametrization good set convex}
	Let $\Omega \subset \R^d$ be an open convex set. If $\epsilon \in (0, 1/2]$ and $r \in (0, \epsilon r_{\rm in}(\Omega)]$ and $x_0 \in G_{\epsilon, r}$ then there exists a convex function $\varphi \colon \R^{d-1}\to \R$ with $\varphi(0)=|\nabla \varphi(0)|=0$, $0\leq \varphi(x')\leq \epsilon |x'|$, $\|\nabla \varphi\|_\infty \lesssim_d \epsilon$ which parametrizes $\partial\Omega$ in $B_r(x_0)$.
 	In particular, the assumptions of Proposition \ref{prop: Estimate close the good boundary} are met for $R=r$, $\kappa=1/2$, and with $M\lesssim_d \epsilon$.
\end{lemma}

\begin{proof}
	Let $Q$ be chosen to satisfy $Q^{-1}\nu(x_0)=(0, \ldots, 0, -1)$ (such $Q$ is not unique but this does not matter). The existence of a function $\varphi$ with $\varphi(0)=|\nabla \varphi(0)|=0$ which locally parametrizes the boundary in $B_r(x_0)$ after translation and rotation follows from the convexity of $\Omega$ and the assumption that $x_0\in G_{\epsilon, r}$ with $r<r_{\rm in}(\Omega)$. That $0\leq \varphi(x')\leq \epsilon |x'|$ for $|x'|<r$ is true by convexity and the assumption that $x_0$ is $(\epsilon, r)$-good. 

	Assuming that $\nabla \varphi(x')$ is well defined, then the outward pointing unit normal to $\partial\Omega$ at $Q(x', \varphi(x'))+x_0$ exists and is given by
	\begin{equation*}
		\nu = \frac{1}{\sqrt{1+|\nabla \varphi(x')|^2}}Q(\nabla \varphi(x'), -1)\,.
	\end{equation*}
	Therefore, the bound on $|\nabla \varphi(x')|$ follows from Lemma~\ref{lem: supremum variation of normal} and the fact that $\nu(x_0)=Q(0, \ldots, 0, -1)$.
\end{proof}


\section{Proof of Theorems \ref{thm: diagonal bulk Lipschitz}, \ref{thm: diagonal bulk convex}, \ref{thm: diagonal bounds good set Lipschitz} and \ref{thm: diagonal kernel bounds convex}}

In this section we put everything together to obtain the heat kernel bounds stated in the introduction.

\begin{proof}[Proof of Theorem \ref{thm: diagonal bulk Lipschitz}]
    In view of Proposition \ref{prop: improved bound bulk} applied with $R=d_\Omega(x)$,
    \begin{equation*}
        \Bigl|k_\Omega(t, x, x)-(4\pi t)^{-d/2}\Bigr| \lesssim_{\delta,\eta}t^{-d/2}\|k_\Omega\|_{G^\infty_{\delta/2,x}((0,t]\times B_{d_\Omega(x)}(x))}e^{-\frac{d_\Omega(x)^2}{(1+\delta)t}}\,,
    \end{equation*}
    for all $x\in \Omega, t>0$ with $d_\Omega(x)\geq \eta\sqrt{t}$. By Corollary~\ref{cor: unif kernel bound Lip} applied with $t_0=\eta^{-2}r_{\rm in}(\Omega)$ we conclude that
    \begin{equation*}
        \|k_\Omega\|_{G^\infty_{\delta/2,x}((0,t]\times B_{d_\Omega(x)}(x))}\leq C_{\Omega, \delta, \eta}\,,
    \end{equation*}
    where we used $t\leq \eta^{-2}d_\Omega(x)^2 \leq \eta^{-2}r_{\rm in}(\Omega)^2= t_0$. This completes the proof of Theorem~\ref{thm: diagonal bulk Lipschitz}.
\end{proof}

\begin{proof}[Proof of Theorem \ref{thm: diagonal bulk convex}]
    The argument is similar as in the previous proof, except that now we use Corollary \ref{cor: unif kernel bound convex} instead of Corollary \ref{cor: unif kernel bound Lip}. This yields the bound
    \begin{equation*}
        \Bigl|k_\Omega(t, x, x)-(4\pi t)^{-d/2}\Bigr| \lesssim_{d,\eta, \delta}\frac{e^{-\frac{d_\Omega(x)^2}{(1+\delta)t}}}{V_\Omega(x, \sqrt{t})}\,.
    \end{equation*}
    The proof is completed by noting that the assumption $d_\Omega(x)\geq \eta\sqrt{t}$ implies that 
    \begin{equation*}
        V_\Omega(x,\sqrt{t})=|\Omega \cap B_{\sqrt{t}}(x)| \geq |B_{\eta \sqrt{t}}(x)\cap B_{\sqrt{t}}(x)| = |B_1|\min\{\eta^d ,1\}t^{d/2}\,. \qedhere
    \end{equation*}
\end{proof}

\begin{proof}[Proof of Theorem \ref{thm: diagonal bounds good set Lipschitz}]
    We apply Lemma \ref{lem: existence of big good set} for some fixed $p>d-1$. Given a sequence $\{\epsilon_j\}_{j\geq 1}\subset (0, 1/4]$ with $\lim_{j\to \infty}\epsilon_j =0$ and a sequence $\{\delta_j\}_{j\geq 1}\subset (0, 1)$ with $\lim_{j\to \infty}\delta_j =0$, we can find a sequence $\{r_j\}\subset(0,1)$ such that for each $j$ and each $0<r\leq r_j$, there exists an $(\epsilon_j, r)$-good set $G_{\epsilon_j,r}$ satisfying
    $$
    \limsup_{s\to 0} \frac{|\{x\in \Omega: d_\Omega(x)<s\}\setminus \mathcal{G}_{\epsilon_j,r}|}{|\{x\in \Omega: d_\Omega(x)<s\}|} \leq \frac{\delta_j}{2}
    $$
    and
    $$
    \sup_{x\in G_{\epsilon_j,r}} \bar\nu_p(x,r) \leq \delta_j \,.
    $$
    Here $\mathcal G_{\epsilon_j,r}$ is the sawtooth region associated to $G_{\epsilon_j, r}$. Next, we can find a sequence $\{s_j\}\subset(0,\infty)$ such that for each $j$ and each $0\leq r\leq r_j$ we have
    $$
    \sup_{0<s\leq s_j} \frac{|\{x\in \Omega: d_\Omega(x)<s\}\setminus \mathcal{G}_{\epsilon_j,r}|}{|\{x\in \Omega: d_\Omega(x)<s\}|} \leq \delta_j
    $$
    By decreasing $s_j$, if necessary, we may assume that $s_j\leq r_j$ for each $j$ and $\lim_{j\to\infty} s_j=0$.
    
    Now, for given $r>0$ let $j(r)\geq 1$ be the largest integer $j$ that satisfies $s_j\geq r$, and put $\mathcal G_r:= \mathcal G_{\epsilon_{j(r)},r_{j(r)}}\cap\{ x\in\Omega:\ d_\Omega(x)<r\}$. Then
    $$
    \sup_{0<s\leq r} \frac{|\{x\in \Omega: d_\Omega(x)<s\}\setminus \mathcal{G}_r|}{|\{x\in \Omega: d_\Omega(x)<s\}|} \leq 
    \sup_{0<s\leq s_{j(r)}} \frac{|\{x\in \Omega: d_\Omega(x)<s\}\setminus \mathcal{G}_{\epsilon_{j(r)},r_{j(r)}}|}{|\{x\in \Omega: d_\Omega(x)<s\}|} \leq \delta_{j(r)}
    $$
    and
    $$
    \sup_{x\in G_{\epsilon_{j(r)},r_{j(r)}}} \bar\nu_p(x,r_{j(r)}) \leq \delta_{j(r)} \,.
    $$
    Note that $j(r)\to\infty$ as $r\to 0$, since $s_j\to 0$ as $j\to\infty$.

    If $x \in \mathcal{G}_r$ then by construction $x \in \Gamma_{r_j(r)}(x_0, \epsilon_{j(r)})$ for some $x_0 \in G_{\epsilon_{j(r)}, r_{j(r)}}$. Since $\partial\Omega$ is Lipschitz, there for any $r>0$ sufficiently small exists 
    a Lipschitz function $\varphi\colon \R^{d-1}\to \R$ with $\varphi(0)=0$ and $\|\nabla \varphi\|_\infty \lesssim_\Omega 1$ which locally parametrizes $\partial\Omega$ in $B_r(x_0)$.
    Therefore, we can apply Proposition \ref{prop: Estimate close the good boundary} with $\kappa=1/2$ to conclude that
    \begin{align*}
        \Bigl|(4\pi t)^{d/2}k_\Omega(t, x, x)-1-e^{-\frac{d_\Omega(x)^2}{t}}\Bigr|
        &\lesssim_{d,\delta, \Omega, p}
        \|k_\Omega\|_{G^\infty_{\delta,x}((0,3t/2]\times B_{2r_{j(r)}}(x_0))}\Theta[\partial\Omega\cap B_{2r_{j(r)}}(x_0)]\\
        &\quad \times\Bigl((\epsilon_{j(r)}+\delta_{j(r)})e^{-\frac{d_\Omega(x)^2}{c(1+\delta)t}}+e^{-\frac{r_{j(r)}^2}{c(1+\delta)t}}\Bigr)\,,
    \end{align*}
    for all $t\in (0, r_{j(r)}^2/2]$. Since $\partial\Omega$ is Lipschitz, we have $\Theta[\partial\Omega \cap B_{2r_{j(r)}}(x_0)]\lesssim_\Omega 1$ and, by Corollary \ref{cor: unif kernel bound Lip} with $t_0 =1$, it holds that $\|k_\Omega\|_{G^\infty_{\delta,x}((0,3t/2]\times B_{2r_{j(r)}}(x_0))}\lesssim_\Omega 1$. Since $r\leq r_{r(j)}$ and $d_\Omega(x) <r$ for any $x \in \mathcal{G}_r$, we conclude that 
     \begin{align*}
        \Bigl|(4\pi t)^{d/2}k_\Omega(t, x, x)-1-e^{-\frac{d_\Omega(x)^2}{t}}\Bigr|
        &\lesssim_{d,\delta, \Omega, p}
        e^{-\frac{d_\Omega(x)^2}{c(1+\delta)t}}\Bigl(\epsilon_{j(r)}+\delta_{j(r)}+e^{-\frac{r^2}{c(1+\delta)t}}\Bigr)
    \end{align*}
    for all $x \in \mathcal{G}_r$ and $t\in (0, r^2/2]$. Since by construction $\lim_{r\to 0^\limplus}\epsilon_{j(r)}+\delta_{j(r)}=0$, the proof is completed by letting $\mathcal{E}(r):= \epsilon_{j(r)}+\delta_{j(r)}$.
\end{proof}

\begin{proof}[Proof of Theorem \ref{thm: diagonal kernel bounds convex}]
    For $\epsilon, r$ as in the theorem let $G_{\epsilon, r}$ the $(\epsilon, r)$-good set constructed in Section \ref{sec:goodconvex} and $\mathcal{G}_{\epsilon, r}$ the associated sawtooth region. By Proposition \ref{prop: G is a good set}, we have
    \begin{equation}\label{eq: G is large first estimate}
        \frac{|\{x\in \Omega: d_\Omega(x)<s\}\setminus \mathcal{G}_{\epsilon, r}|}{|\{x\in \Omega: d_\Omega(x)<s\}|}\lesssim_d \frac{sr\Haus^{d-1}(\partial\Omega)}{\epsilon r_{\rm in}(\Omega)|\{x\in \Omega: d_\Omega(x)<s\}|}
    \end{equation}
    for all $s<r/2$, and
    \begin{equation}\label{eq: nup convex}
        \sup_{x\in G_{\epsilon, r}}\bar \nu_p(x, r) \lesssim_d \epsilon\,,
    \end{equation}
    for any $p\geq 1$. By Lemma \ref{lem: volume bdry neighborhood bound} and the co-area formula,
    \begin{align*}
        |\{x\in \Omega: d_\Omega(x)<s\}| 
        &\geq s\Haus^{d-1}(\partial\Omega)\Bigl(1-\frac{s}{r_{\rm in}(\Omega)}\Bigr)^{d-1}\\
        &\gtrsim_d s\Haus^{d-1}(\partial\Omega)\,,
    \end{align*}
    where the second step used that by assumption $0<s<r/2\leq \epsilon r_{\rm in}(\Omega)/2\leq r_{\rm in}(\Omega)/8$. When combined with \eqref{eq: G is large first estimate}, this established \eqref{eq: Good set is big0}.
    
    For any $x \in \mathcal{G}_{\epsilon,r}$ there exists an $x_0\in G_{\epsilon,r}$ so that $x\in \Gamma_r(x_0, \epsilon)$. By Lemma \ref{lem: Local parametrization good set convex}, we can apply Proposition \ref{prop: Estimate close the good boundary} with $R=r, \kappa=1/2$, $M \lesssim_d 1$ to deduce that
    \begin{equation}\label{eq: kernel bound good set convex 1}
     \begin{aligned}
        \Bigl|(4\pi t)^{d/2}k_\Omega(t, x, x)-1-e^{-\frac{d_\Omega(x)^2}{t}}\Bigr|
        &\lesssim_{d,\delta,p}
        \|k_\Omega\|_{G^\infty_{\delta,x}((0,3t/2]\times B_{2r}(x_0))}\Theta[\partial\Omega\cap B_{2r}(x_0)]\\
        &\quad \times\Bigl((\epsilon+\bar\nu_p(x_0, r)e^{-\frac{d_\Omega(x)^2}{c(1+\delta)t}}+e^{-\frac{r^2}{c(1+\delta)t}}\Bigr)\,,
    \end{aligned}
    \end{equation}
    for all $t\in (0, r^2/2]$. 
    
    It remains to bound some of the terms appearing in the right-hand side of \eqref{eq: kernel bound good set convex 1}. By \eqref{eq: nup convex}, it holds that $\bar\nu_p(x_0, r)\lesssim_d \epsilon$. Since $\Omega$ is convex, Lemma \ref{lem: local monotonicity of perimieter} implies that $\Theta[\partial\Omega \cap B_{2r}(x_0)] \lesssim_d 1$. By Corollary \ref{cor: unif kernel bound convex}, it holds that
    \begin{equation*}
        \|k_\Omega\|_{G^\infty_{\delta,x}((0, 3t/2]\times B_{2r}(x_0))} \lesssim_{d,\delta} \frac{t^{d/2}}{V_\Omega(x, \sqrt{t})}\,.
    \end{equation*}
    By Lemma~\ref{lem: Bishop-Gromov monotonicity} and since $\Omega \cap B_{r/2}(x_0)\subset \Omega \cap B_{r}(x)$,
    \begin{equation*}
        \frac{t^{d/2}}{V_\Omega(x, \sqrt{t})} \leq \frac{r^d}{V_\Omega(x, r)} \leq \frac{r^d}{V_\Omega(x_0, r/2)}\,.
    \end{equation*}
    Since $x_0$ is $(\epsilon, r)$-good with $\epsilon\in (0,1/4]$, it holds that
	\begin{align*}
		\Omega \cap B_{r/2}(x) &\supseteq  \{y \in \R^d: (y-x)\cdot \nu(x)<-\epsilon |y-x| \}\cap B_{r/2}(x)\\
		&\supseteq\{y \in \R^d: (y-x)\cdot \nu(x)<-|y-x|/4 \}\cap B_{r/2}(x)\,,
	\end{align*}
	therefore
	\begin{equation*}
		V_\Omega(x_0, r/2)\gtrsim_d r^d\,.
	\end{equation*}
    Thus we conclude that $\|k_\Omega\|_{G^\infty_{\delta,x}((0, 3t/2]\times B_{2r}(x_0))} \lesssim_{d,\delta} 1$.

    Using $d_\Omega(x)<r/2$ for every $x\in \mathcal{G}_{\epsilon, r}$, we have arrived at
     \begin{align*}
        \Bigl|(4\pi t)^{d/2}k_\Omega(t, x, x)-1-e^{-\frac{d_\Omega(x)^2}{t}}\Bigr|
        &\lesssim_{d,\delta,p}
        e^{-\frac{d_\Omega(x)^2}{c(1+\delta)t}}\Bigl(\epsilon+e^{-\frac{r^2}{c(1+\delta)t}}\Bigr)
    \end{align*}
    for all $x \in \mathcal{G}_{\epsilon, r}$ and $t\in (0, r^2/2]$. This completes the proof of Theorem \ref{thm: diagonal kernel bounds convex}.
\end{proof}

\appendix


\section{Convex geometry toolbox}
\label{app: Convex geometry}

In this appendix, which is similar but not identical to the appendix in \cite{FrankLarson_Inventiones25}, we record some elements of convex geometry. Most of the results that we shall make use of are either well known to experts or follow from well-known results.

We begin by recalling the fact that perimeter is monotonically increasing under inclusion of convex sets.
\begin{lemma}\label{lem: local monotonicity of perimieter}
	If $\Omega' \subset \Omega \subset \R^d$ are convex sets, then 
	$$
		\Haus^{d-1}(\partial\Omega') \leq \Haus^{d-1}(\partial\Omega)\,.
	$$ 
	In particular, if $\Omega \subset \R^d$ is convex and $x\in \R^d, r>0$ then
	\begin{equation*}
		\Haus^{d-1}(\partial\Omega \cap B_r(x)) \leq \Haus^{d-1}(\partial B_r(x)) =\Haus^{d-1}(\partial B_1(0)) r^{d-1}\,.
	\end{equation*}
\end{lemma}

The second lemma provides an upper and a lower bound for the size of the level set of the distance function $d_\Omega$. It is stated in terms of the inradius, defined in \eqref{eq:inrad}. The lower bound is proved in \cite{LarsonJFA} and the upper is a consequence of Lemma~\ref{lem: local monotonicity of perimieter}.

\begin{lemma}\label{lem: inner parallel perimeter bounds}
	If $\Omega \subset \R^d$ is a bounded convex set, then for any $s \in (0, r_{\rm in}(\Omega)]$
\begin{equation*}
	\Bigl(1-\frac{s}{r_{\rm in}(\Omega)}\Bigr)^{d-1}\Haus^{d-1}(\partial\Omega)\leq \Haus^{d-1}(\{x\in \Omega: d_\Omega(x)=s\}) \leq \Haus^{d-1}(\partial\Omega)\,.
\end{equation*}
\end{lemma}

By an application of the co-area formula and the bounds in Lemma~\ref{lem: inner parallel perimeter bounds}, one can prove the next lemma.
\begin{lemma}\label{lem: volume bdry neighborhood bound}
	If $\Omega \subset \R^d$ is a bounded convex set, then for any $s \in [0, r_{\rm in}(\Omega)]$
\begin{equation*}
	s\Haus^{d-1}(\partial\Omega)\Bigl(1-\frac{s}{r_{\rm in}(\Omega)}\Bigr)^{d-1}\leq |\{x\in \Omega: d_\Omega(x)<s\}| \leq s\Haus^{d-1}(\partial\Omega)\,.
\end{equation*}
\end{lemma}
\begin{proof}
    By the co-area formula, we can write
\begin{equation*}
    |\{x\in \Omega: d_\Omega(x)<s\}| = \int_0^s \Haus^{d-1}(\{x\in \Omega: d_\Omega(x)=\tau\})\,d\tau\,.
\end{equation*}
By Lemma \ref{lem: inner parallel perimeter bounds}, we deduce that
\begin{equation*}
    \Haus^{d-1}(\partial\Omega)\int_0^s\Bigl(1-\frac{\tau}{r_{\rm in}(\Omega)}\Bigr)_\limplus^{d-1}\,d\tau \leq |\{x\in \Omega: d_\Omega(x)<s\}| \leq s \Haus^{d-1}(\partial\Omega)\,.
\end{equation*}
Since $d-1\geq 0$ we can use monotonicity to bound
\begin{equation*}
    \int_0^s\Bigl(1-\frac{\tau}{r_{\rm in}(\Omega)}\Bigr)_\limplus^{d-1}\,d\tau \geq s \Bigl(1-\frac{s}{r_{\rm in}(\Omega)}\Bigr)_\limplus^{d-1} \,.
\end{equation*}
This completes the proof of Lemma \ref{lem: volume bdry neighborhood bound}.
\end{proof}

In the paper we several times make use of the following consequence of the Bishop--Gromov comparison theorem. An elementary proof can be found in \cite[Lemma A.5]{FrankLarson_Inventiones25}.

\begin{lemma}\label{lem: Bishop-Gromov monotonicity}
	If $\Omega \subset \R^d$ is convex, then the function
\begin{equation*}
	(0, \infty) \ni r \mapsto \frac{|\Omega \cap B_r(a)|}{r^d}
\end{equation*}
is nonincreasing for any fixed $a\in \overline{\Omega}$.
\end{lemma}

The following assertion appears as \cite[Proposition A.6]{FrankLarson_Inventiones25}, where also a proof can be found.

\begin{proposition}\label{prop: Minkowski sum bounds}
	Fix $d \geq 1$. Given $c_1>0$ there exists $c_2>0$ with the following properties. 
	If $\Omega \subset \R^d$ is open, bounded, and convex, then
	\begin{align*}
		|\Omega + B_r|-|\Omega| & \geq r\Haus^{d-1}(\partial\Omega) \quad \mbox{for all }r>0\,,\\
		|\Omega + B_r|-|\Omega|&\lesssim_d \Haus^{d-1}(\partial\Omega)r\Bigl[1+\Bigl(\frac{r}{r_{\rm in}(\Omega)}\Bigr)^{d-1}\Bigr] \quad \mbox{for all }r>0\,,\\
		|\Omega+B_r| &\leq c_2 \Haus^{d-1}(\partial\Omega) r \Bigl(\frac{r}{r_{\rm in}(\Omega)}\Bigr)^{d-1} \quad \mbox{for all } r\geq c_1 r_{\rm in}(\Omega)\,.
	\end{align*}
\end{proposition}

The final result in this appendix concerns integrals of the reciprocal of the local measure $V_\Omega(x, r)$ over various subsets of $\Omega$. The crucial part of the result is that it shows that the region where $V_\Omega(x, r)$ becomes small has relatively small measure.

\begin{proposition}\label{prop: bounds integral of local volume}
    Let $\Omega \subset\R^d$, $d\geq 2$, be an open, bounded, and convex set. Then for any $r>0$ and measurable $\omega \subset \Omega$ it holds that
    \begin{equation*}
        \int_\omega \frac{r^d}{V_\Omega(x, r)}\,dx \leq \frac{4^d}{|B_1(0)|} |\omega+B_{r}(0)|
    \end{equation*}
    and
    \begin{equation*}
        \int_\omega \frac{r^d}{V_\Omega(x, r)}\,dx \leq \frac{4^d}{|B_1(0)|} \biggl[|\omega| + \frac{r^2\Haus^{d-1}(\partial\Omega)}{r_{\rm in}(\Omega)}\Bigl(2(d-1)+ \Bigl(\frac{r}{r_{\rm in}(\Omega)}\Bigr)^{d-2}\Bigr)\biggr]\,.
    \end{equation*}
\end{proposition}

\begin{proof}
    We begin by proving the first bound.

    Let $\{x_k\}_{k=1}^{M} \subset \overline \omega$ be a collection of points satisfying that $B_{r}(x_k)\cap B_{r}(x_j)=\emptyset$ if $k\neq j$ and $\overline \omega \subset \cup_{k=1}^{M} B_{2r}(x_k)$. The existence of such a collection can be proven by induction as follows. Start by choosing an arbitrary point in $\overline{\omega}$; then, given a finite collection of points with pairwise distances $\geq 2r$, we can either find a new point in $\overline{\omega}$ whose distance to all of the previous points is $\geq 2r$ or the collection satisfies that $\overline{\omega}\subset \cup_k B_{2r}(x_k)$. Since $\omega$ is bounded, the algorithm terminates after a finite number of steps. 

    We then have that
    \begin{align*}
	   \int_\omega \frac{r^{d}}{V_\Omega(x, r)}\,dx 
	   &\leq 
	   \sum_{k=1}^{M} \int_{\omega \cap B_{2r}(x_k)} \frac{r^d}{V_\Omega(x, r)}\,dx \,.
    \end{align*}
    By using Lemma~\ref{lem: Bishop-Gromov monotonicity} and the fact that by inclusion $V_\Omega(x, 4r)\geq V_\Omega(x_k, 2r)$ for all $x \in B_{2r}(x_k)$, we find
    \begin{align*}
	   \sum_{k=1}^M\int_{\omega \cap B_{2r}(x_k)} \frac{r^{d}}{V_\Omega(x, r)}\,dx
	   &\leq
	   \sum_{k=1}^M\int_{\omega \cap B_{2r}(x_k)} \frac{(4\sqrt{t})^{d}}{V_\Omega(x, 4r)}\,dx\\
	   &\leq
	   \sum_{k=1}^M\int_{\omega \cap B_{2r}(x_k)} \frac{4^dr^d}{V_\Omega(x_k, 2r)}\,dx\\
       &=
	   4^dr^d\sum_{k=1}^M \frac{|\omega \cap B_{2r}(x_k)|}{|\Omega \cap B_{2r}(x_k)|}\\
	   &\leq 4^dr^dM\,.
    \end{align*}

    Since $\{B_r(x_k)\}_{k=1}^M$ are disjoint and $\cup_{k=1}^M B_r(x_k)\subset \omega+B_r(0)$, we find
    \begin{equation*}
        M =\frac{\sum_{k=1}^M|B_r(x_k)|}{|B_r(0)|} = \frac{|\cup_{k=1}^M B_r(x_k)|}{r^d|B_1(0)|} \leq \frac{|\omega+ B_r(0)|}{r^d|B_1(0)|}\,.
    \end{equation*}
    We have therefore shown that for any $r>0, \omega \subset \Omega$ it holds that
    \begin{equation*}
        \int_\omega \frac{r^{d}}{V_\Omega(x, r)}\,dx \leq \frac{4^d}{|B_1(0)|}|\omega + B_r(0)|\,,
    \end{equation*}
    which completes the proof of the first bound in the proposition.

    For the second bound we need to be more precise in what happens close to the boundary of $\Omega$. To this end, we decompose $\omega$ into two disjoint sets given by
    \begin{equation*}
        \omega_1 := \{x\in \omega: \dist(x, \{d_\Omega>r\})<3r\}\,, \quad \omega_2 := \{x\in \omega: \dist(x, \{d_\Omega>r\})\geq 3r\}\,.
    \end{equation*}

    To bound the integral over $\omega_1$ we use that if $x\in \omega_1$ then there exists $y\in \{z\in \Omega: d_\Omega(z)>r\}$ with $|x-y|<3r$. Lemma \ref{lem: Bishop-Gromov monotonicity} implies that
    \begin{equation*}
        \frac{r^d}{V_\Omega(x, r)} \leq \frac{(4r)^d}{V_\Omega(x, 4r)}  \leq \frac{(4r)^d}{V_\Omega(y, r)} = \frac{4^d}{|B_1(0)|}\,,
    \end{equation*}
    where the second step uses $B_r(y)\subset B_{4r}(x)$ and the third uses $B_r(y)\subset \Omega$. Consequently,
    \begin{equation}\label{eq: omega1 bound}
        \int_{\omega_1} \frac{r^d}{V_\Omega(x, r)}\,dx \leq \frac{4^d}{|B_1(0)|}|\omega_1|\,.
    \end{equation}

    To estimate the integral over $\omega_2$ we proceed as in the first bound of the proposition. In the same manner as before we can choose $\{x_k\}_{k=1}^{M} \subset \overline \omega_2$ a finite collection of points satisfying that $B_{r}(x_k)\cap B_{r}(x_j)=\emptyset$ if $k\neq j$ and $\overline \omega_2 \subset \cup_{k=1}^{M} B_{2r}(x_k)$ and conclude that
    \begin{equation}\label{eq: omega2 initial bound}
	   \int_{\omega_2}\frac{r^d}{V_\Omega(x, r)}\,dx\leq  4^dr^dM\,.
    \end{equation}
    It thus remains to bound $M$. It remains to control $M$. To this end note that, since $|x_k-x_j|\geq 2r$ for all $k\neq j$,
    \begin{align*}
	   M = \frac{\sum_{k=1}^M|B_{r}(x_k)|}{|B_1(0)|r^d} =  \frac{\bigl|\cup_{k=1}^MB_{r}(x_k)\bigr|}{|B_1(0)|r^d}\,.
    \end{align*}
    Moreover, observe that since $x_k\in \omega_2\subset \{x\in \Omega: \dist(x, \{d_\Omega>r\})\geq 3r\}$
    \begin{equation*}
    	B_{r}(x_k) \cap (\{x \in \Omega: d_\Omega(x)<r\}+B_{2r}(0)) = \emptyset 
    \end{equation*}
    and 
    \begin{equation*}
    	B_{r}(x_k) \subset \omega_2 + B_{r}(0)\,.
    \end{equation*}
    Therefore, in the notation of Minkowski summation we have
    \begin{equation}\label{eq: first bound M}
    \begin{aligned}
    	|B_1(0)|r^{d}M &= \bigl|\cup_{k=1}^M B_{r}(x_k)\bigr| \\
	   &\leq |(\omega_2 + B_{r}(0))\setminus (\{x\in \Omega: d_\Omega(x) >r\}+ B_{2r}(0))|\\
        &\leq |(\Omega + B_{r}(0))\setminus (\{x\in \Omega: d_\Omega(x) >r\}+ B_{2r}(0))|\\
	   &= |\Omega + B_{r}(0)|-|\{x\in \Omega:d_\Omega(x) >r\}+ B_{2r}(0)|\,,
    \end{aligned}
    \end{equation}
    where the last step used $\{x\in \Omega: d_\Omega(x)>r\} + B_{2r}(0) \subset \Omega+ B_{r}(0)$.

    By the first and second bounds in Proposition~\ref{prop: Minkowski sum bounds},
    \begin{align*}
	   |\Omega + B_{r}|&\leq |\Omega| + \sqrt{t}\Haus^{d-1}(\partial\Omega)\Bigl[1 +  \Bigl(\frac{r}{r_{\rm in}(\Omega)}\Bigr)^{d-1}\Bigr]\,, \\
	   |\{x\in \Omega: d_\Omega(x) >r\} + B_{2r}(0)| 
	   &\geq |\{x\in \Omega : d_\Omega(x) >r\}|\\
       &\quad + 2r\Haus^{d-1}(\partial\{x\in \Omega:d_\Omega(x) >r\})\,.
    \end{align*}

    By Lemma~\ref{lem: volume bdry neighborhood bound},
    \begin{equation*}
	   |\{x\in \Omega: d_\Omega(x) >r\}| = |\Omega| - |\{x\in \Omega: d_\Omega(x)<r\}| \geq |\Omega|- r\Haus^{d-1}(\partial\Omega)\,.
    \end{equation*}
    By Lemma~\ref{lem: inner parallel perimeter bounds},
    \begin{equation*}
    \begin{aligned}
	   \Haus^{d-1}(\partial\{x\in \Omega:d_\Omega(x) >r\}) &\geq \Haus^{d-1}(\partial\Omega)\Bigl(1- \frac{r}{r_{\rm in}(\Omega)}\Bigr)_\limplus^{d-1}\\
       &\geq \Haus^{d-1}(\partial\Omega)\Bigl(1- \frac{(d-1)r}{r_{\rm in}(\Omega)}\Bigr)\,.
    \end{aligned}
    \end{equation*}

    Plugging these bounds into~\eqref{eq: first bound M} yields
    \begin{align*}
	   |B_1(0)|r^{d}M
	   &\leq 
	   \Bigl(|\Omega|+ r\Haus^{d-1}(\partial\Omega)\Bigl[1+ \Bigl(\frac{r}{r_{\rm in}(\Omega)}\Bigr)^{d-1}\Bigr]\Bigr)\\
	   &\quad - \Bigl(|\Omega|- r\Haus^{d-1}(\partial\Omega)+ 2r\Haus^{d-1}(\partial\Omega)\Bigl(1- \frac{(d-1)r}{r_{\rm in}(\Omega)}\Bigr)\Bigr)\\
	   &=
	   \frac{r^2 \Haus^{d-1}(\partial\Omega)}{r_{\rm in}(\Omega)} \Bigl[2(d-1)+\Bigl(\frac{r}{r_{\rm in}(\Omega)}\Bigr)^{d-2}\Bigr]\,,
    \end{align*}
    Inserting this estimate into \eqref{eq: omega2 initial bound} and combining the resulting bound with \eqref{eq: omega1 bound} completes the proof of the proposition.
\end{proof}



\def\myarXiv#1#2{\href{http://arxiv.org/abs/#1}{\texttt{arXiv:#1\,[#2]}}}

\end{document}